\documentclass[psamsfonts]{amsart}
\usepackage{amssymb,amsfonts}
\usepackage[all,arc]{xy}
\usepackage{enumerate}
\usepackage{color}

\usepackage{mathrsfs}

\usepackage{picture}
\usepackage{graphicx}

\usepackage{lscape}
\usepackage{geometry}

\usepackage{pstricks,color}
\usepackage{pst-plot}

\newrgbcolor{tauzerocolor}{0.5 0.5 0.5}

\newrgbcolor{twoextncolor}{0.7 0.7 0}
\newrgbcolor{etaextncolor}{0.5 0 1}
\newrgbcolor{nuextncolor}{0.6 0.3 0}

\newgray{gridline}{0.8}

\newcommand{\cirrad}{0.06}

\newpsobject{twoextncurve}{psbezier}{linecolor=tauzerocolor}
\newpsobject{etaextncurve}{psbezier}{linecolor=tauzerocolor}
\newpsobject{nuextncurve}{psbezier}{linecolor=tauzerocolor}

\newtheorem{thm}{Theorem}[section]
\newtheorem{cor}[thm]{Corollary}
\newtheorem{prop}[thm]{Proposition}
\newtheorem{lem}[thm]{Lemma}
\newtheorem{conj}[thm]{Conjecture}

\theoremstyle{definition}
\newtheorem{defn}[thm]{Definition}

\newtheorem{exmp}[thm]{Example}

\newtheorem{question}[thm]{Question}

\newtheorem{rem}[thm]{Remark}

\newtheorem{notation}[thm]{Notation}

\makeatletter
\let\c@equation\c@thm
\makeatother
\numberwithin{equation}{section}

\makeatletter
\ifx\SK@label\undefined\let\SK@label\label\fi
 \let\your@thm\@thm
 \def\@thm#1#2#3{\gdef\currthmtype{#3}\your@thm{#1}{#2}{#3}}
 \def\mylabel#1{{\let\your@currentlabel\@currentlabel\def\@currentlabel
  {\currthmtype~\your@currentlabel}
 \SK@label{#1@}}\label{#1}}
 
\makeatother

\title{The triviality of the 61-stem in the stable homotopy groups of spheres}

\author{Guozhen Wang}
\thanks{The first author was partially supported by the Danish National Research Foundation through the Centre for Symmetry and Deformation (DNRF92).}
\address{Shanghai Center for Mathematical Sciences, Fudan University, Shanghai, China, 200433}
\address{Department of Mathematics, University of Copenhagen, Universitetsparken 5, 2100 Copenhagen, Denmark}
\email{wangguozhen@fudan.edu.cn}

\author{Zhouli Xu}
\address{Department of Mathematics, The University of Chicago, Chicago, IL 60637}
\email{xu@math.uchicago.edu}

\begin{document}

\maketitle

\begin{abstract}
We prove that the 2-primary $\pi_{61}$ is zero. As a consequence, the Kervaire invariant element $\theta_5$ is contained in the strictly defined 4-fold Toda bracket $\langle 2, \theta_4, \theta_4, 2\rangle$.

Our result has a geometric corollary: the 61-sphere has a unique smooth structure and it is the last odd dimensional case - the only ones are $S^1, S^3, S^5$ and $S^{61}$.


Our proof is a computation of homotopy groups of spheres. A major part of this paper is to prove an Adams differential $d_3(D_3) = B_3$. We prove this differential by introducing a new technique based on the algebraic and geometric Kahn-Priddy theorems. The success of this technique suggests a theoretical way to prove Adams differentials in the sphere spectrum inductively by use of differentials in truncated projective spectra.
\end{abstract}

\tableofcontents

\section{Introduction}

In 1904, Poincar\'{e} proposed the following famous conjecture:
\begin{conj}
Let $M$ be a closed 3-manifold. If $M$ is simply connected, then $M$ is homeomorphic to the 3-sphere.
\end{conj}
This is the celebrated Poincar\'{e} conjecture. It was proved by Perelman \cite{Per} in 2002, using geometric analytic methods. Note that a closed 3-manifold is simply connected if and only if it is homotopy equivalence to the 3-sphere.

This conjecture can be generalized to higher dimensions as the following question.
\begin{question}
Let $M$ be a closed $n$-manifold. Suppose $M$ is homotopy equivalent to $S^n$. Is $M$ homeomorphic to $S^n$?
\end{question}
The answer turns out to be yes for all dimensions. For $n=4$, it was proved by Freedman \cite{Fre} in 1982. For $n\geq5$, it was proved by Smale \cite{Sma} in 1962, using the theory of $h$-cobordisms, and by Newman \cite{New} in 1966 and by Connell \cite{Con} in 1967. The statement Smale proved assumes further that the $n$-manifold $M$ admits a smooth structure, while the statement Newman and Connell proved does not require such a condition.

In summary, we have the following theorem:
\begin{thm}(\cite{Sma, New, Con, Fre, Per})
Any closed $n$-manifold that is homotopy equivalent to $S^n$ is homeomorphic to $S^n$.
\end{thm}

We can also generalize this question into the smooth category.
\begin{question}
Let $M$ be a closed $n$-manifold. Suppose $M$ is homeomorphic to $S^n$. Is $M$ diffeomorphic to $S^n$?
\end{question}

For $n=3$, the answer is yes. It is due to Moise \cite{Moi} that every closed 3-manifold has a unique smooth structure. In particular, the 3-sphere has a unique smooth structure. For $n=4$, this question is wildly open.

For higher dimensions, Milnor \cite{Mil} constructed an exotic smooth structure on $S^7$. Furthermore, Kervaire and Milnor \cite{KM2} showed that the answer is not true in general for $n\geq 5$.

Since the answer to Question 1.4 is not true in general, there come two natural questions:
\begin{question}
How many exotic structures are there on $S^n$?
\end{question}

\begin{question}
For which $n$'s does there exist a unique smooth structure on $S^n$?
\end{question}

Kervaire and Milnor reduced Question 1.5 to a computation of the stable homotopy groups of spheres. In fact, Kervaire and Milnor constructed a group $\Theta_n$, which is the group of h-cobordism classes of homotopy $n$-spheres. The group $\Theta_n$ classifies the differential structures on $S^n$ for $n\geq5$. This group $\Theta_n$ has a subgroup $\Theta_n^{bp}$, which consists of homotopy spheres that bound parallelizable manifolds. The relation between $\Theta_n$ and $\pi_n$ (the $n$-th stable homotopy group of the spheres) can be summarized by the following theorem.

\begin{thm}\label{km}(Kervaire-Milnor \cite{KM2})
Suppose that $n\geq 5$.
\begin{enumerate}
\item
The subgroup $\Theta_n^{bp}$ is cyclic, and has the following order:
\begin{equation*}
|\Theta_n^{bp}|=\left\{
\begin{split}
1 & , ~~\text{if~~}n\text{~~is even,} \\
1 \text{~~or~~} 2 & , ~~\text{if~~}n = 4k+1,\\
2^{2k-2}(2^{2k-1}-1)B(k) & , ~~\text{if~~}n = 4k-1.
\end{split}
\right.
\end{equation*}
Here $B(k)$ is the numerator of $4B_{2k}/k$ and $B_{2k}$ is the Bernoulli number.\\
\item
For $n \not\equiv 2 ~(mod ~4)$, there is an exact sequence
\begin{displaymath}
    \xymatrix{
   0 \ar[r] & \Theta_n^{bp} \ar[r] & \Theta_n \ar[r] & \pi_n/J \ar[r] & 0.
    }
\end{displaymath}
Here $\pi_n/J$ is the cokernel of the $J$-homomorphism.\\

\item
For $n \equiv 2 ~(mod ~4)$, there is an exact sequence
\begin{displaymath}
    \xymatrix{
   0 \ar[r] & \Theta_n^{bp} \ar[r] & \Theta_n \ar[r] & \pi_n/J \ar[r]^\Phi & \mathbb{Z}/2 \ar[r] & \Theta_{n-1}^{bp} \ar[r] & 0.
    }
\end{displaymath}
Here the map $\Phi$ is the Kervaire invariant.
\end{enumerate}
\end{thm}

\begin{rem}
In the first part of Theorem 1.7, the case $n\equiv 3 ~(mod ~4)$ depends on the computation of the order of the image of the $J$-homomorphism. The case $n\equiv 1 ~(mod ~4)$ depends on the Kervaire invariant in dimension $n+1$.
The computation of the image of the $J$-homomorphism at $4k-1$ stems is a special case of the Adams conjecture. The proof was
completed by Mahowald \cite{Mah}, and the full Adams conjecture was proved by Quillen \cite{Qui}, Sullivan \cite{Sul}, and by Becker-Gottlieb \cite{BG}.
\end{rem}

For Question 1.6, it is clear from Theorem 1.7 that, for $n=4k+3$ with $k\geq1$, the smooth structure on the $n$-sphere is never unique. For $n=4k+1$ with $k\geq1$, the answer depends on the existence of the Kervaire invariant elements. In 2009, Hill, Hopkins and Ravenel \cite{HHR} showed that the only dimensions in which the Kervaire invariant elements exist are 2, 6, 14, 30, 62 and possibly 126. That is, in other dimensions, the Kervaire invariant map
\begin{displaymath}
    \xymatrix{
   \pi_n/J \ar[r]^\Phi & \mathbb{Z}/2
    }
\end{displaymath}
in part $(3)$ of Theorem 1.7 is always zero and the group $\Theta_{n-1}^{bp}$ is $\mathbb{Z}/2$. Therefore, the only odd dimensional spheres that could have a unique smooth structure are $S^1, S^3, S^5, S^{13}, S^{29}, S^{61}$ and $S^{125}$. Further, the cases $S^{13}$ and $S^{29}$ can be ruled out by May's \cite{May2} 3-primary computation of the stable homotopy groups of spheres. \\

For dimension 61, we have the main theorem of this paper.

\begin{thm} \label{s61}
The 2-primary $\pi_{61} = 0$, and therefore the sphere $S^{61}$ has a unique smooth structure.
\end{thm}

We postpone the proof of the first claim of Theorem 1.9 to Section 2, and present the proof of the second claim now.

\begin{proof}
In \cite{BJM}, Barratt, Jones and Mahowald showed that the Kervaire invariant element $\theta_5$ exists. The second author gave a new proof in \cite{Xu}. By Theorem \ref{km}, this implies that $\Theta^{bp}_{61}=0$.

At an odd prime $p$, the first nontrivial element in the cokernel of $J$ is $\beta_1$, which lies in the stem $2p^2-2p-2$. (This is proved in Section 4 of \cite{Rav}.) This value is $82$ if $p=7$. For $p=3$ and $p=5$, the table in Appendix A3 of Ravenel's green book \cite{Rav} shows that the cokernel of $J$ in dimension 61 vanishes. Therefore, the cokernel of $J$ in dimension 61 vanishes at all odd primes.

Combining the first claim of Theorem 1.9 with Theorem 1.7, this proves the second claim of Theorem 1.9.
\end{proof}

There is an important corollary of our theorem, regarding the Kervaire invariant element $\theta_5 \in \pi_{62}$.

\begin{cor}
The Kervaire invariant class $\theta_5 \in \pi_{62}$ is contained in the strictly defined 4-fold Toda bracket $\langle 2, \theta_4, \theta_4, 2\rangle$.
\end{cor}

\begin{proof}
We first check this 4-fold Toda bracket is strictly defined. In \cite{Xu}, the second author showed that $\theta_4^2=0$. Note that the 3-fold Toda bracket $\langle 2, \theta_4, \theta_4\rangle$ is contained in $\pi_{61}=0$. Therefore, this 4-fold Toda bracket is strictly defined.
In the Adams $E_3$ page, we have a Massey product
$$\langle h_0, h_4^2, h_4^2, h_0\rangle = h_5^2,$$
because of the Adams differential $d_2(h_5) = h_0h_4^2$. Then the theorem follows from Moss's Theorem \cite[Theorem 1.2]{Mos}.
\end{proof}

\begin{rem}
When computing stable stems, it is crucial to understand Toda brackets decompositions of multiplicatively indecomposable classes. A theorem of Joel Cohen \cite{Coh} says that any classes in the stable homotopy groups of spheres can be decomposed as a (matric) Toda bracket starting only from the classes that correspond to the Hopf maps. However, in practice, it is usually hard to find such a description. For the Kervaire invariant class $\theta_5$, our Corollary 1.10 gives the first known Toda bracket of it. Note that $\theta_4$ was known to have multiple Toda bracket decompositions using the Hopf maps.

By a theorem of Barratt, Jones and Mahowald \cite{BJM2}, if $\theta_5$ has order 2 and $\theta_5^2 =0$, then $\theta_6$ exists and has order 2. It is proved by the second author \cite{Xu} that $\theta_5$ has order 2. Our Toda bracket of $\theta_5$ in Corollary 1.10 therefore leads us to consider the Toda bracket $\langle \theta_5, 2, \theta_4\rangle$ in $\pi_{93}$, which is in a much lower stem than $\theta_6$ itself. Using obstruction theory as Barratt-Jones-Mahowald did in \cite{BJM}, one can show that if the Toda bracket $\langle \theta_5, 2, \theta_4\rangle$ contains zero, then $\theta_6$ exists. The Toda bracket of $\theta_5$ in Corollary 1.10 has also been very helpful in ongoing work of Isaksen and the authors of extending computations of stable stems.
\end{rem}

For dimension 125, we have the following proposition.

\begin{prop}
The sphere $S^{125}$ does not have a unique smooth structure.
\end{prop}

\begin{proof}
This proof uses the Hurewicz image of $tmf$ (the spectrum of topological modular forms). See \cite{Bau, Hen} for computations of the homotopy groups of tmf.

Let $\{w\}\in\pi_{45}$ be the unique homotopy class detected by $w$ in Adams filtration 9. It is known that both $\overline{\kappa}\in\pi_{20}$ and $\{w\}$ are detected by tmf, that is, they map nontrivially under the following map:
$$\pi_\ast S^0 \longrightarrow \pi_\ast tmf.$$
We have that $\overline{\kappa}^4\{w\} \neq 0$ in $\pi_{125} tmf$. Therefore, $\overline{\kappa}^4\{w\} \neq 0$ in $\pi_{125} S^0$ and it lies in the cokernel of $J$. This shows that $S^{125}$ does not have a unique smooth structure.
\end{proof}

Therefore, we have the following corollary.

\begin{cor} \label{ods}
The only odd dimensional spheres with a unique smooth structure are $S^1, S^3, S^5$ and $S^{61}$ .
\end{cor}

For even dimensions, since the subgroup $\Theta_n^{bp}$ is always zero, we need to understand the cokernel of the $J$-homomorphism.\\

In \cite{Mil2}, Milnor states that up to dimension $64$, the only dimensions where the $n$-sphere has a unique smooth structure are $n=1,2,3,5,6,12,61$ and possibly $n=4$. This observation is based on the computation of 2-primary stable homotopy groups of spheres up to the 64 stem by Kochman and Mahowald \cite{KM} from 1995. Recently, Isaksen \cite{Isa} discovered several errors in Kochman and Mahowald's computations, and he was able to give rigorous proofs of computations through the 59 stem. One major correction is that, instead of having order 4, $\pi_{56}$ is of order $2$ and is generated by a class in the image of $J$. Consequently, we have the following theorem:

\begin{thm} \label{s56} (Isaksen)
The sphere $S^{56}$ has a unique smooth structure.
\end{thm}

\begin{proof}
It is clear from Theorem 1.7 that $\Theta^{bp}_{56}=0$. Ravenel's computation \cite{Rav} shows that the cokernel of $J$ in dimension 56 vanishes at odd primes. Recent computation of Isaksen \cite{Isa} shows that the cokernel of $J$ in dimension 56 vanishes at the prime 2. Then this theorem follows from part $(2)$ of Theorem 1.7.
\end{proof}

The technique used by Kochman and Mahowald \cite{KM} is quite different from the classical technique used by Barratt, Bruner, Mahowald, May and Tangora \cite{May2, MT, BMT, Tan1, Tan2, Tan3, Br1} through dimension 45, and the motivic technique used by Isaksen and the second author \cite{Isa, IX} through dimension 59. For more details of known techniques, see Section 2.\\

Based on Isaksen's computation, we give rigorous proofs regarding $\pi_{60}$ and $\pi_{61}$. Besides the classical technique of Toda brackets, one of our proofs relies heavily on the transfer map from the infinite real projective spectrum to the sphere spectrum. The success of this technique suggests a theoretical way to improve our understanding through a bigger range.\\

Combining our computations with the previous knowledge of $\pi_\ast$, we have another corollary of the main theorem.

\begin{cor}
For $5\leq n\leq61$, the only dimensions that $S^n$ has a unique smooth structure are $n=5,6,12,56$ and $61$.
\end{cor}

\begin{proof}
The range for $n<19$ was known to Kervaire and Milnor. For even dimensions between $20$ and $60$, it is straightforward to check that at $p=2$, the only dimension in which the cokernel of $J$ vanishes is 56. Note that the Kervaire invariant $\theta_4$ exists in dimension 30. In fact, Barratt, Mahowald and Tangora \cite{BMT} showed that $\pi_{30}$ is $\mathbb{Z}/2$, generated by $\theta_4$. Therefore, we need to consider odd primary computations in this dimension. May \cite{May2} showed that at the prime 3, the cokernel of $J$ in dimension 30 is $\mathbb{Z}/3$, which implies that $S^{30}$ does not have a unique smooth structure. Combining with Theorems 1.7 and 1.9 and Corollary 1.13, this completes the proof.
\end{proof}

\begin{rem}
Recent work of Behrens, Hill, Hopkins and Mahowald \cite{BHHM} shows that the next sphere with a unique smooth structure, if exists, is in dimension at least 126.
\end{rem}

Based on our current knowledge on $\pi_\ast$, we have the following conjecture.

\begin{conj}
For dimensions greater than 4, the only spheres with a unique smooth structure are $S^5$, $S^6$, $S^{12}$, $S^{56}$, and $S^{61}$.
\end{conj}

The rest of this paper is organized as follows. \\

In Section 2, we give a brief review of the stem-wise computation of $\pi_\ast$ with a focus on the prime 2. We compare the known techniques. We reduce $\pi_{61}=0$ to three Adams differentials. \\

From Section 3 to Section 10, we present the proof of the hardest differential $d_3(D_3) = B_3$. In Section 3, we summarize the strategy of our technique and explain how we organize the details of the proof in Sections 4 through 10. The intuition behind part of this proof is included in Appendix II, which is Section 14.\\

We present the proof of the other two differentials in Sections 11 and 12. The targets of these two differentials detect certain homotopy classes. We use the theory of Toda brackets to show that these homotopy classes must vanish. \\

\textbf{Acknowledgement}: The authors would like to thank Mark Behrens for introducing and suggesting this problem and for many
helpful conversations. The authors would like to thank Dan Isaksen for discussing and sharing lots of his computations.
We are also especially indebted to him for his very careful checking of our proofs. Any errors that remain are not his
fault. The authors thank Agnes Beaudry and Peter May for helping edit and reorganize this paper. Both have read more
drafts than they care to remember. We'd also like to thank Paul Goerss and Jesper Grodal for their support. Finally, we owe a great debt of gratitude to Mark Mahowald for his tenacious
exploration of the stable stems and his generosity in sharing his ideas with us.

\section{The stable homotopy groups of spheres}

The computation of the stable homotopy groups of spheres is a long standing and very challenging problem in algebraic topology. We will first give a brief review of the history from the stem-wise point of view, and then talk about some recent progress.\\

After the geometric computation of the first three stems \cite{Hop, Freu, Whi, Pon, Rok}, Serre \cite{Ser} did the computation of $\pi_n$ for $n<9$ with the aid of the Serre spectral sequence and the Eilenberg-Maclane spectra. Serre also showed that these stable groups are finite in positive stems, so we can compute them one prime at a time. Afterwards, at each prime, Adams \cite{Ada1} constructed the Adams spectral sequence whose $E_2$-term encodes the information that we could obtain via primary cohomology operations. The Adams spectral sequence gives an upper bound on $\pi_n$ and therefore determining the Adams differentials becomes a major method in computing the stable homotopy groups. Generalizing Adams's idea, Novikov constructed the Adams-Novikov spectral sequence using the complex cobordism spectrum.\\

There is another method using the EHP sequence, which computes the unstable homotopy groups inductively. Using this method, together with the Toda bracket operations, Toda \cite{Tod} succeeded to do the computation of $\pi_n$ for $n\leq19$.\\

It turns out that the Adams-Novikov spectral sequence is more successful at odd primes than at the prime 2. In the 1980's, using the Adams-Novikov spectral sequence, Ravenel \cite{Rav} computed up to the 108-stem at the prime 3, and the 999-stem at the prime 5. Previously, the computation was due independently to Nakamura \cite{Nak} and Tangora \cite{Tan4} up to the 103-stem at the prime 3, and to Aubry \cite{Aub} up to the 760-stem at the prime 5.\\

At the prime 2, the Adams spectral sequence is still the most efficient way. In \cite{May2}, May constructed the May spectral sequence, which converges to the $E_2$-page of the Adams spectral sequence. This works at all primes. In particular, May computed $\pi_n$ for $n\leq28$ at the prime 2. In the 1960's, using the Adams spectral sequence, and with the aid of the technique of Toda brackets, Barratt, Mahowald and Tangora \cite{BMT} determined the differentials in the Adams spectral sequence up to the $45$-stem. About one and a half decades later, Bruner \cite{Br1} discovered a gap in \cite{BMT}, and proved a new Adams differential in the 38-stem. Bruner's differential therefore corrected the result of $\pi_{37}$ and $\pi_{38}$, and along with that corrected some relations in the stable homotopy ring.\\

In 1990, based on the Atiyah-Hirzebruch spectral sequence of the Brown-Peterson spectrum, Kochman \cite{Koc} made an algorithm and implemented it into computer programmes. In this way, he produced a table of $\pi_n$ up to the 64-stem. However, his method is not completely reviewed by others due to its complexity, and his result is not fully accepted by the experts. In 1995, Kochman and Mahowald \cite{KM} made a few corrections to \cite{Koc}, in the range from 52 to 64. A tentative chart of the Adams spectral sequence is included in the appendix of \cite{Koc} and \cite{KM} without proofs. Note that the Adams differentials in this chart are deduced from the stable homotopy groups, not the other way around. \\

For about two decades, much of our knowledge regarding $\pi_n$, in the range from $45$ to $64$, relied on \cite{KM}. Recently, by comparing the motivic Adams spectral sequence and the classical Adams spectral sequence, Isaksen \cite{Isa} gave rigorous proofs to all but one Adams differentials up to the $59$ stem. The exception was later proved by the second author \cite{IX} based on Isaksen's motivic computation. Along with a few corrections to some relations in the stable homotopy ring, Isaksen proved a new Adams differential in the 57-stem, which was not included in \cite{KM}. This also corrects $\pi_{56}$ and $\pi_{57}$ as we used in the proof of Theorem 1.12.\\

In the range beyond the 59-stem, Isaksen \cite{Isa} also proved a few differentials. The part which Isaksen did not fully understand can be summarized in his Adams $E_\infty$ chart \cite{Isa2}, which we include in the following page.\\

Note that we do not include elements in filtration higher than 16. Those elements are detected by the $K(1)$-local sphere, and are not relevant to our proof. Here we use dashed curved lines to denote some known nontrivial 2, $\eta$ and $\nu$-extensions. Note that because of differentials unknown to Isaksen, the actual $E_\infty$-page beyond the 59-stem is a subquotient of what is shown in this chart.

\psset{linewidth=0.3mm}

\psset{unit=1.1cm}
\begin{pspicture}(55,0)(66,18)

\psgrid[unit=2,gridcolor=gridline,subgriddiv=0,gridlabelcolor=white](28,0)(33,8)

\twoextncurve[linestyle=dashed](60,12)(60.5,13)(60.5,14)(60,15)
\twoextncurve[linestyle=dashed](64,2)(64.5,2.5)(64.5,3.5)(64,4)
\etaextncurve[linestyle=dashed](59,13)(59,13.5)(59.5,14.5)(60,15)
\etaextncurve[linestyle=dashed](65,13)(65,13.5)(65.5,15)(66,15)
\nuextncurve[linestyle=dashed](57,10)(57,12)(59,15)(60,15)

\scriptsize

\rput(56,-1){56}
\rput(57,-1){57}
\rput(58,-1){58}
\rput(59,-1){59}
\rput(60,-1){60}
\rput(61,-1){61}
\rput(62,-1){62}
\rput(63,-1){63}
\rput(64,-1){64}
\rput(65,-1){65}
\rput(66,-1){66}

\rput(55,0){0}
\rput(55,2){2}
\rput(55,4){4}
\rput(55,6){6}
\rput(55,8){8}
\rput(55,10){10}
\rput(55,12){12}
\rput(55,14){14}
\rput(55,16){16}

\pscircle*(57,10){\cirrad}
\uput{\cirrad}[-90](57,10){$h_0h_2h_5i$}

\pscircle*(58,8){\cirrad}
\uput{\cirrad}[-90](58,8){$h_1Q_2$}

\psline[linecolor=tauzerocolor](59,10)(60,11)
\psline[linecolor=tauzerocolor](59,10)(62,11)

\pscircle*(59,10){\cirrad}
\uput{\cirrad}[-90](59,10){$B_{21}$}

\pscircle*(59,13){\cirrad}
\uput{\cirrad}[-90](59,13){$d_0w$}

\psline[linecolor=tauzerocolor](60,7)(61,8)

\pscircle*(60,7){\cirrad}
\uput{\cirrad}[-90](60,7){$B_3$}

\pscircle*(60,11){\cirrad}

\pscircle*(60,12){\cirrad}
\uput{\cirrad}[-90](60,12){$g^3$}

\pscircle*(60,15){\cirrad}
\uput{\cirrad}[90](60,15){$d_0^2l$}

\psline[linecolor=tauzerocolor](61,4)(62,5)
\psline[linecolor=tauzerocolor](61,4)(64,5)

\pscircle*(61,4){\cirrad}
\uput{\cirrad}[-90](61,4){$D_3$}

\psline[linecolor=tauzerocolor](61,6)(64,7)

\pscircle*(61,6){\cirrad}
\uput{\cirrad}[-90](61,6){$A'$}

\pscircle*(61,8){\cirrad}

\pscircle*(61,14){\cirrad}
\uput{\cirrad}[-90](61,14){$gz$}

\psline[linecolor=tauzerocolor](62,2)(63,3)
\psline[linecolor=tauzerocolor](62,2)(64.90,3)

\pscircle*(62,2){\cirrad}
\uput{\cirrad}[-90](62,2){$h_5^2$}

\pscircle*(62,5){\cirrad}

\psline[linecolor=tauzerocolor](62,6)(65,7)

\pscircle*(62,6){\cirrad}
\uput{\cirrad}[-90](62,6){$h_5n$}

\psline[linecolor=tauzerocolor](62,8)(63,9)

\pscircle*(62,8){\cirrad}
\uput{\cirrad}[-90](62,8){$E_1+C_0$}

\psline[linecolor=tauzerocolor](61.90,10)(63,11)

\pscircle*(61.90,10){\cirrad}
\uput{\cirrad}[-120](61.90,10){$h_1X_1$}

\pscircle*(62.10,10){\cirrad}
\uput{\cirrad}[0](62.10,10){$R$}

\pscircle*(62,11){\cirrad}

\pscircle*(62,16){\cirrad}
\uput{\cirrad}[0](62,16){$d_0^3g$}

\psline[linecolor=tauzerocolor](63,3)(64,4)

\pscircle*(63,3){\cirrad}

\pscircle*(63,6){\cirrad}
\uput{\cirrad}[-90](63,6){$h_1H_1$}

\psline[linecolor=tauzerocolor](62.90,7)(63.90,8)

\pscircle*(62.90,7){\cirrad}
\uput{\cirrad}[-180](62.90,7){$X_2$}

\psline[linecolor=tauzerocolor](63.10,7)(66,8)

\pscircle*(63.10,7){\cirrad}
\uput{\cirrad}[-90](63.10,7){$C'$}

\pscircle*(63,9){\cirrad}

\pscircle*(63,11){\cirrad}

\psline[linecolor=tauzerocolor](64,2)(65.10,3)

\pscircle*(64,2){\cirrad}
\uput{\cirrad}[-90](64,2){$h_1h_6$}

\pscircle*(64,4){\cirrad}

\pscircle*(64,5){\cirrad}

\pscircle*(64,7){\cirrad}

\pscircle*(63.90,8){\cirrad}

\psline[linecolor=tauzerocolor](64.10,8)(65,9)

\pscircle*(64.10,8){\cirrad}
\uput{\cirrad}[-90](64.10,8){$h_3Q_2$}

\psline[linecolor=tauzerocolor](64,10)(64.90,11)

\pscircle*(64,10){\cirrad}
\uput{\cirrad}[-90](64,10){$q_1$}

\pscircle*(64.90,3){\cirrad}

\psline[linecolor=tauzerocolor](65.10,3)(66,4)

\pscircle*(65.10,3){\cirrad}

\pscircle*(65,7){\cirrad}

\pscircle*(65,9){\cirrad}

\psline[linecolor=tauzerocolor](65,10)(65.10,11)
\psline[linecolor=tauzerocolor](65,10)(65.90,11)

\pscircle*(65,10){\cirrad}
\uput{\cirrad}[-90](65,10){$B_{23}$}

\psline[linecolor=tauzerocolor](64.90,11)(66,12)

\pscircle*(64.90,11){\cirrad}

\pscircle*(65.10,11){\cirrad}

\pscircle*(65,12){\cirrad}
\uput{\cirrad}[-90](65,12){$h_5Pj$}

\pscircle*(65,13){\cirrad}
\uput{\cirrad}[-90](65,13){$gw$}

\pscircle*(66,4){\cirrad}

\pscircle*(66,6){\cirrad}
\uput{\cirrad}[-90](66,6){$r_1$}

\pscircle*(66,8){\cirrad}

\psline[linecolor=tauzerocolor](66,10)(66.10,11)

\pscircle*(66,10){\cirrad}
\uput{\cirrad}[-90](66,10){$B_5+D'_2$}

\pscircle*(65.90,11){\cirrad}

\psline[linecolor=tauzerocolor](66.10,11)(66,12)

\pscircle*(66.10,11){\cirrad}

\pscircle*(66,12){\cirrad}

\pscircle*(66,15){\cirrad}
\uput{\cirrad}[90](66,15){$g^2j$}

\end{pspicture}


Now we reduce the first claim of Theorem 1.9, i.e., $\pi_{61}=0$, to three Adams differentials.

\begin{proof}
It is proven in Theorem 3.1 (and this is the crux of the paper) that
$$d_3(D_3) = B_3$$
and therefore
$$d_3(h_1D_3) = h_1B_3.$$
It is proven in Theorem \ref{A} that
$$d_5(A') = h_1B_{21}.$$
It is proven in Theorem \ref{gz} that the element $gz$ must be killed by some Adams differential.

There are no elements left in the $E_\infty$-page of the 61-stem.
\end{proof}

It is clear that these differentials also settle $\pi_{60}$.

\begin{cor}
The 2-primary $\pi_{60}$ is $\mathbb{Z}/4$, generated by $\overline{\kappa}^3$.
\end{cor}

\begin{proof}
The elements $g^3$ and $d_0^2l$ are the only elements left, and there is a hidden 2-extension between them. The element $g$ detects $\overline{\kappa} \in \pi_{20}$. Therefore, the 2-primary group $\pi_{60}$ is $\mathbb{Z}/4$, generated by $\overline{\kappa}^3$.
\end{proof}

\section{Intuition and the proof of the differential $d_3(D_3)=B_3$}

We have developed a general method to prove a differential in the Adams spectral sequence of the sphere spectrum. The strategy can be summarized in three parts:

\begin{enumerate}
\item Using the algebraic Kahn-Priddy theorem, we pullback a differential in the Adams spectral sequence
of the sphere spectrum to one in the Adams spectral sequence of the suspension spectrum of $RP^\infty$.

\item Using our knowledge of the cell structure of $RP^\infty$ and the algebraic Atiyah-Hirzebruch spectral sequence, we deduce the Adams differential in $RP^\infty$ from one in a certain $H\mathbb{F}_2$-subquotient of $RP^\infty$.

\item Using our knowledge of the Adams spectral sequence of the sphere spectrum, and the cell structure of this $H\mathbb{F}_2$-subquotient, we reduce the computation of the Adams differential in this $H\mathbb{F}_2$-subquotient to that of a product (or more generally a Toda bracket) in a \emph{lower} stem of the stable homotopy groups of spheres.

\end{enumerate}

Intuitively, an $H\mathbb{F}_2$-subquotient of a CW complex is a subquotient to the eyes of mod 2 homology, in a sense that will be made precise in Definition 4.1.

The technical heart of the paper, explained in Sections 3 - 10, is to apply this method to prove the following theorem.

\begin{thm}
We have the Adams differential: $d_3(D_3)=B_3.$
\end{thm}

With notations to be explained, here is a ``road map" of the proof.

\begin{displaymath}
    \xymatrix{
  Ext(S^0) & & & & & & Ext(\Sigma^{14} C\eta) \ar[ddd] \\
  & & & & & & \\
  & & & & & & \\
  Ext(P_1^\infty) \ar[uuu] & & & & & & Ext(\widetilde{X}) \ar[ddd] \\
  & & & & & & \\
  & & & & & & \\
  Ext(P_1^{23}) \ar[uuu] \ar[rrr] & & & Ext(P_{14}^{23}) \ar[rrr] & & & Ext(X)
  }
\end{displaymath}

\begin{displaymath}
  \xymatrix{
  B_3 &      &    &    &     B_1[14] \ar@{|->}[dd] & \\
      & D_3 \ar@{-->}[ul]^{d_3} & & & &  h_4^3[16] \ar@{-->}[ul]^{d_4} \ar@{|->}[dd]\\
  *\txt{\\$G[6]$} \ar@{|->}[uu]  &      &    &    &     B_1[14] \ar@{|->}[dd] & \\
      & h_1h_3h_5[22] \ar@{-->}[ul]^{d_3} \ar@{|->}[uu] & & & &  h_4^3[16] \ar@{|->}[dd] \ar@{-->}[ul]^{d_4}\\
*\txt{\\$G[6]$} \ar@{|->}[uu] &      &  B_1[14] \ar@{|->}[rr]  &    &     B_1[14] & \\
      & *\txt{\\$h_1h_3h_5[22]$} \ar@{-->}[ul]^{d_3} \ar@{|->}[uu] \ar@{|->}[rr] & & h_1h_3h_5[22] \ar@{-->}[ul]^{d_4}  & & *\txt{ $h_4^3[16]$\\ \\$\underline{h_1h_3h_5[22]}$ is a cycle} \ar@{-->}[ul]^{d_4}
    }
\end{displaymath}

The first part of this ``road map" describes seven Adams spectral sequences and maps among them; the second part describes certain Adams $d_3$ or $d_4$ differentials in the 61-stem of each of the spectral sequences and maps in the Adams $E_2$-page among the sources and targets of these differentials.

\begin{notation}
All spectra are localized at the prime 2. Suppose $Z$ is a spectrum. Let $Ext(Z)$ denote its Adams $E_2$-page.

For spectra, let $S^0$ be the sphere spectrum, and $P_1^\infty$ be the suspension spectrum of $RP^\infty$. In general, we use $P_n^{n+k}$ to denote the suspension spectrum of $RP^{n+k}/RP^{n-1}$. Recall that we have the James periodicity for the stunted projective spectra:
$$\Sigma^{\phi(k)} P_n^{n+k} \simeq P_{n+\phi(k)}^{n+k+\phi(k)},$$
where $\phi(k) = 2^{\psi(k)}$, and
\begin{equation*}
 \psi(k)= \lfloor\frac{k}{2}\rfloor + \left\{
 \begin{aligned}
    -1 &, ~~~~~~~~k\equiv 0 ~mod~ 8\\
    0 &, ~~~~~k\equiv 1 \\
    0 &, ~~~~~k\equiv 2 \\
    1&, ~~~~~k\equiv 3 \\
    0&, ~~~~~k\equiv 4 \\
    1&, ~~~~~k\equiv 5 \\
    0&, ~~~~~k\equiv 6 \\
    0&, ~~~~~k\equiv 7.
      \end{aligned}
 \right.
\end{equation*}
For example, $\phi(7) = 2^{\psi(7)} = 8$, hence we have $P_{16}^{23} \simeq \Sigma^8 P_8^{15} \simeq \Sigma^{16} P_0^7$.

The spectrum $X$ is a quotient spectrum of $P_{14}^{23}$ and $\widetilde{X}$ is a subspectrum of $X$. The spectrum $C\eta$ is the cofiber of $\eta \in \pi_1$, and $\Sigma^{14} C \eta$ turns out to be a subspectrum of $\widetilde{X}$. The precise definitions of the spectra $X$ and $\widetilde{X}$ can be found in Definition 5.1.
\end{notation}

For sources and targets of these differentials, we use the following way to denote the elements in the Adams $E_2$-page of $P_1^\infty$ and its $H\mathbb{F}_2$-subquotients. One way to compute $Ext(P_1^\infty)$ is to use the algebraic Atiyah-Hirzebruch spectral sequence.

\begin{displaymath}
    \xymatrix{
  E_1 = \bigoplus_{n=1}^\infty Ext(S^n) \ar@{=>}[r] & Ext(P_1^\infty)
    }
\end{displaymath}


\begin{notation}
We denote any element in $Ext(S^n)$ to be $a[n]$, where $a\in Ext(S^0)$, and $n$ suggests that it comes from $Ext(S^n)$. We will abuse notation and write the same symbol $a[n]$ for an element of $Ext(P_1^\infty)$ detected by the element $a[n]$ of the Atiyah-Hirzebruch $E_\infty$ page. Thus, there is indeterminacy in the notation $a[n]$ that is detected by Atiyah-Hirzebruch $E_\infty$ elements in lower filtration. When $a[n]$ is the element of lowest Atiyah-Hirzebruch filtration in the Atiyah-Hirzebruch $E_\infty$ page in a given bidegree $(s,t)$, then $a[n]$ also is a well-defined element of $Ext(P_1^\infty)$. Sometimes we will need to be precise about a particular element of $Ext(P_1^\infty)$ detected by $a[n]$. We will use the notation $\underline{a[n]}$ to denote a particular choice, and we must provide a definition that specifies $\underline{a[n]}$ in this case. We use this same notation for all $H\mathbb{F}_2$-subquotients of $P_1^\infty$. There won't be any confusion on the index $n$ since any $H\mathbb{F}_2$-subquotient contains at most one cell in each dimension.
\end{notation}

\begin{rem}
In \cite{WX}, we computed the Adams $E_2$-page of $P_1^\infty$ in the range of $t<72$ by the Lambda algebra. This Lambda algebra computation gives us a lot of information on the algebraic Atiyah-Hirzebruch spectral sequence. In particular, there is a one-to-one correspondence between the differentials in the Lambda algebra computation and differentials in the algebraic Atiyah-Hirzebruch spectral sequence.
\end{rem}

\begin{rem}
Despite the indeterminacy in Notation 3.3, there is a huge advantage of it. Suppose $f: Q\rightarrow Q'$ is a map between two $H\mathbb{F}_2$-subquotients of $P_1^\infty$, which is a composite of inclusion and quotient maps. Suppose further that there exists an element $a[n]$ which is a generater of both $Ext^{s,t}(Q)$ and $Ext^{s,t}(Q')$ for some bidegree $(s,t)$ (this implies both $Q$ and $Q'$ have a cell in dimension $n$). We therefore must have that, with the right choices, $a[n]$ in $Ext^{s,t}(Q)$ maps to $a[n]$ in $Ext^{s,t}(Q')$. This property follows from the naturality of the algebraic Atiyah-Hirzebruch spectral sequence.

\begin{displaymath}
    \xymatrix{
  \underset{i \in I}\bigoplus Ext(S^i) \ar@{=>}[dd] \ar[rr] & & \underset{i \in I'}\bigoplus Ext(S^i) \ar@{=>}[dd] \\
  & & \\
  Ext(Q) \ar[rr] & & Ext(Q') \\
 a[n] \ar@{|->}[rr] & & a[n]
    }
\end{displaymath}
\end{rem}

\begin{exmp}
As an example, the group $Ext^{3, 64}(X) = \mathbb{Z}/2\oplus\mathbb{Z}/2\oplus\mathbb{Z}/2$, is generated by $h_4^3[16], \ h_1h_3h_5[22]$ and $h_0h_3h_5[23]$, as explained in Table 6 in Section 9. The element $h_4^3[16]$ is uniquely determined by our notation, since it has the lowest Atiyah-Hirzebruch filtration. In fact, the 16-skeleton of $X$ is $\Sigma^{14}C \eta$. The inclusion map specifies the element $h_4^3[16]$ in $Ext^{3, 64}(X)$ as the image of the element $h_4^3[16]$ in $Ext^{3, 64}(\Sigma^{14}C \eta)$.

\begin{displaymath}
    \xymatrix{
  Ext(\Sigma^{14}C \eta) \ar[rr] & & Ext(X) \\
 h_4^3[16] \ar@{|->}[rr] & & h_4^3[16]
    }
\end{displaymath}

As a comparison, the element $h_1h_3h_5[22]$ in our notation does not specify a unique element in $Ext^{3, 64}(X)$. In fact, suppose $A$ and $B$ are elements in $Ext^{3, 64}(X)$, which are detected by $h_4^3[16]$ and $h_1h_3h_5[22]$ in the algebraic Atiyah-Hirzebruch spectral sequence of $X$. The element $A+B$ is therefore also detected by $h_1h_3h_5[22]$. Our notation $h_1h_3h_5[22]$ in $Ext^{3, 64}(X)$ does \emph{not} distinguish the elements $B$ and $A+B$.

It turns out making a choice for $h_1h_3h_5[22]$ is essential to our proof. In fact, we use a 4 cell complex $X^{22}$ (see Definition 5.6) to specify such a choice. The complex $X^{22}$ is an $H\mathbb{F}_2$-subcomplex of $X$, and contains a cell in dimension 22, but not in dimension 16. The group $Ext^{3, 64}(X^{22}) = \mathbb{Z}/2$, generated by $h_1h_3h_5[22]$, as explained in Table 4 in Section 8. We denote the image of $h_1h_3h_5[22]$ in $Ext^{3, 64}(X^{22})$ to be $\underline{h_1h_3h_5[22]}$ in $Ext^{3, 64}(X)$.

\begin{displaymath}
    \xymatrix{
  Ext(X^{22}) \ar[rr] & & Ext(X) \\
 h_1h_3h_5[22] \ar@{|->}[rr] & & \underline{h_1h_3h_5[22]}
    }
\end{displaymath}
\end{exmp}

Now, we explain the main steps of the proof for the Adams differential $d_3(D_3) = B_3$.

\begin{enumerate}

\item \textbf{\underline{Step 1}}: We establish a $d_4$ differential in the Adams spectral sequence of $\Sigma^{14} C \eta$:
$$d_4(h_4^3[16]) = B_1[14].$$
This is stated as Theorem 7.1 and proved in Section 7.

\item \textbf{\underline{Step 2}}: Using the inclusion map $\Sigma^{14} C \eta \rightarrow \widetilde{X}$, we push forward the Adams $d_4$ differential in Step 1 to an Adams $d_4$ differential in $\widetilde{X}$:
$$d_4(h_4^3[16]) = B_1[14].$$
This is stated as Theorem 8.1 and proved in Section 8.

\item \textbf{\underline{Step 3}}: Using the inclusion map $\widetilde{X} \rightarrow X$, we push forward the Adams $d_4$ differential in Step 2 to an Adams $d_4$ differential in $X$:
$$d_4(h_4^3[16]) = B_1[14].$$
This is stated as Theorem 9.1 and proved in Section 9.

\item \textbf{\underline{Step 4}}: We show that the chosen element $\underline{h_1h_3h_5[22]}$ (as explained in Example 3.6) is a permanent cycle in the Adams spectral sequence of $X$. This is stated as Theorem 9.2 and proved in Section 9.

    Combining with Step 3, we have an immediate Adams $d_4$ differential in $X$:
$$d_4(\underline{h_1h_3h_5[22]} + h_4^3[16]) = B_1[14].$$
This is stated as Corollary 9.3.

\item \textbf{\underline{Step 5}}: Using the quotient map $P_1^{23} \rightarrow X$, we pull back the Adams $d_4$ differential in Step 4 to an Adams $d_3$ differential in $P_1^{23}$:
$$d_3(h_1h_3h_5[22]) = G[6].$$
This is stated as Theorem 10.1 and proved in Section 10.

\item \textbf{\underline{Step 6}}: Using the inclusion map $P_1^{23} \rightarrow P_1^\infty$ and the transfer map $P_1^\infty \rightarrow S^0$, we push forward the Adams $d_3$ differential in Step 4 to an Adams $d_3$ differential in $S^0$:
$$d_3(D_3) = B_3.$$
This is our main theorem and is proved in this section.
\end{enumerate}

We have several comments before we dive into the details of the proofs.

\begin{rem}
Step 1 is the origin of all our differentials. It follows essentially from a relation in the stable homotopy groups of spheres:
there is a nontrivial $\eta$-extension from $h_4^3$ to $B_1$.
\end{rem}

\begin{rem}
Intuitively, the most mysterious step is Step 5. The intuition behind such an argument is explained in detail in Section 14, which is Appendix II. But note that the intuition is irrelevant to our proofs. For the proof, when we pull back a $d_4$ differential, the preimage of the source must support a $d_2$, $d_3$ or $d_4$ differential. To get the $d_3$ differential as claimed in Step 5, we rule out all other possibilities.
\end{rem}

\begin{rem}
Logically, the most complicated step is Step 2. The intuition seems straightforward: we push forward a $d_4$ differential to get a $d_4$ differential. But note that we need to show that the image of the target survives to the $E_4$ page, i.e., it is not killed by a $d_2$ or $d_3$ differential. It turns out in the corresponding bidegrees, there are 10 elements which have the potential to support a $d_2$ or $d_3$ differential. To rule out these possibilities, we will show in Section 8 that 9 elements out of the 10 are permanent cycles, and the other one supports a $d_2$ differential which is irrelevant. Our way to show these elements are permanent cycles is by showing they are permanent cycles in some $H\mathbb{F}_2$-subcomplexes of $X$. For this purpose, in Section 5, we study the cell structure of $X$, as well as its several $H\mathbb{F}_2$-subcomplexes.
\end{rem}

\begin{rem}
The intuitive reason why this method works is due to the geometric and algebraic Kahn-Priddy theorems. It is because of Step 6 that we can reduce the computation of an Adams differential in $S^0$ to one in $P_1^\infty$, and further to one in a \emph{lower} stem of $S^0$.
\end{rem}

In the rest of this section, we prove Step 6.

Recall that we have the Kahn-Priddy Theorem \cite{KP}, stated as follows.

\begin{thm}
The transfer map $P_1^\infty \rightarrow S^0$ induces a surjection on homotopy groups in positive stems.
\end{thm}

We also have the algebraic Kahn-Priddy Theorem due to Lin \cite{Lin}.

\begin{thm}
The transfer map also induces a surjection:
$$Ext^{s,t}(P_1^\infty) \rightarrow Ext^{s+1, t+1} (S^0)$$
for $t-s>0$.
\end{thm}

Now we prove Step 6.

\begin{proof}

For the purpose of the differential $d_3(D_3)=B_3$, we check the two tables in the appendix of \cite{WX}. See \cite{WX} for more details of the Lambda algebra notation we used here. We rewrite $Ext^{(s,t)}$ as $Ext^{(s,s+(t-s))}$ to indicate that it is in stem $t-s$. \\

The element $D_3$ is in $Ext^{4,61+4}(S^0) = \mathbb{Z}/2$. Checking the table for $P_1^\infty$, we have that
\begin{equation*}
\begin{split}
Ext^{3,61+3}(P_1^\infty) = \mathbb{Z}/2, \text{~~generated by~~} & (22)~ 21~ 11~ 7, \\
Ext^{3,61+3}(P_1^{23}) = (\mathbb{Z}/2)^2, \text{~~generated by~~} & (22)~ 21~ 11~ 7,\\ & (23)~ 22~ 13~ 3.
\end{split}
\end{equation*}
The element $21~ 11~ 7$ lies in
$$Ext^{3,39+3}(S^0) = \mathbb{Z}/2, \text{~~ generated by~~} h_1h_3h_5.$$
Therefore, the element $h_1h_3h_5[22]$ maps to $D_3$.\\

The element $B_3$ is in $Ext^{7,60+7}(S^0) = \mathbb{Z}/2$. Checking the table for $P_1^\infty$, we have that
\begin{equation*}
\begin{split}
Ext^{6,60+6}(P_1^\infty) = (\mathbb{Z}/2)^2, \text{~~ generated by~~} & (6)~ 2~ 4~ 7~ 11~ 15~ 15, \ (20)~ 5~ 5~ 9~ 7~ 7~ 7,\\
Ext^{6,60+6}(P_1^{23}) = (\mathbb{Z}/2)^4, \text{~~ generated by~~} & (6)~ 2~ 4~ 7~ 11~ 15~ 15, \ (20)~ 5~ 5~ 9~ 7~ 7~ 7, \\
& (22)~ 3~ 5~ 9~ 7~ 7~ 7, \ (23)~ 13~ 2~ 3~ 5~ 7~ 7.
\end{split}
\end{equation*}
In the table for the transfer, we have that the element $(20)~ 5~ 5~ 9~ 7~ 7~ 7$ (with certain choice) maps to $0$. Due to the algebraic Kahn-Priddy Theorem, we must have the element $(6)~ 2~ 4~ 7~ 11~ 15~ 15$ maps to $B_3$. The element $2~ 4~ 7~ 11~ 15~ 15$ lies in
$$Ext^{6,54+6}(S^0) = \mathbb{Z}/2, \text{~~ generated by~~} G.$$
Therefore, the element $G[6]$ maps to $B_3$.

\begin{displaymath}
    \xymatrix{
  Ext^{3,61+3}(P_1^{23}) \ar[rr] & & Ext^{3,61+3}(P_1^{\infty}) \ar[rr] & & Ext^{4,61+4}(S^0) \\
 h_1h_3h_5[22] \ar@{|->}[rr] & & h_1h_3h_5[22] \ar@{|->}[rr] & & D_3\\
 Ext^{6,60+6}(P_1^{23}) \ar[rr] & & Ext^{6,60+6}(P_1^{\infty}) \ar[rr] & & Ext^{7,60+7}(S^0) \\
 G[6] \ar@{|->}[rr] & & G[6] \ar@{|->}[rr] & & B_3
    }
\end{displaymath}

Note that in both $Ext(P_1^\infty)$ and $Ext{(P_1^{23})}$, the elements $h_1h_3h_5[22]$ and $G[6]$ are uniquely determined by our notation, since they have the lowest Atiyah-Hirzebruch filtrations in their bidegrees.

In the Adams spectral sequence for $S^0$, the element $B_3$ survives to the $E_3$-page: there is no element that could kill $B_3$ by a $d_2$ differential. Therefore, the Adams $d_3$ differential in $P_1^{23}$:
$$d_3(h_1h_3h_5[22]) = G[6]$$
in Step 5 (Theorem 10.1) implies the Adams $d_3$ differential in $S^0$:
$$d_3(D_3) = B_3.$$
\end{proof}

\section{$H\mathbb{F}_2$-subquotients for CW spectra}

In this section, we introduce the definitions of $H\mathbb{F}_2$-subcomplexes and $H\mathbb{F}_2$-quotient complexes for CW spectra. We also discuss an important $H\mathbb{F}_2$-subcomplex of $P_1^6$ in Theorem 4.7.

\begin{defn}
Let $A$, $B$, $C$ and $D$ be CW spectra, $i$ and $q$ be maps
\begin{displaymath}
    \xymatrix{
A \ar@{^{(}->}[r]^-i & B, & B \ar@{->>}[r]^-q & C
    }
\end{displaymath}
We say that $(A, i)$ is an $H\mathbb{F}_2$-subcomplex of $B$, if the map $i$ induces an injection on mod 2 homology. We denote an $H\mathbb{F}_2$-subcomplex by an hooked arrow as above.

We say that $(C, q)$ is an $H\mathbb{F}_2$-quotient complex of $B$, if the map $q$ induces a surjection on mod 2 homology. We denote an $H\mathbb{F}_2$-quotient complex by a double headed arrow above.

When the maps involved are clear in the context, we also say $A$ is an $H\mathbb{F}_2$-subcomplex of $B$, and $C$ is an $H\mathbb{F}_2$-quotient complex of $B$.

Furthermore, we say $D$ is an $H\mathbb{F}_2$-subquotient of $B$, if $D$ is an $H\mathbb{F}_2$-subcomplex of an $H\mathbb{F}_2$-quotient complex of $B$, or an $H\mathbb{F}_2$-quotient complex of an $H\mathbb{F}_2$-subcomplex of $B$.
\end{defn}

\begin{rem}
Note that our definitions of $H\mathbb{F}_2$-subcomplexes and $H\mathbb{F}_2$-quotient complexes are \emph{not} necessarily subcomplexes and quotient complexes on the point set level. Our definitions should be thought as in the homological or homotopical sense. Here is a motivating example of why we use these definitions. The top cell of the spectrum $P_1^3$ splits off, therefore there is a map from $S^3$ to $P_1^3$ that induces an injection on mod 2 homology. This is an $H\mathbb{F}_2$-subcomplex in our sense. However, on the point set level, the image of the attaching map is not a point, therefore $S^3$ is not a subcomplex of $P_1^3$ in the classical sense.
\end{rem}

\begin{rem}
It follows directly from Definition 4.1 that if $(A, i)$ is an $H\mathbb{F}_2$-subcomplex of $B$, then the cofiber of $i$ is an $H\mathbb{F}_2$-quotient complex of $B$, which we sometimes denote as $B/A$. Dually, if $(C, q)$ is an $H\mathbb{F}_2$-quotient complex of $B$, then the fiber of $q$ is an $H\mathbb{F}_2$-subcomplex of $B$.
\end{rem}

The following lemma is useful in constructing $H\mathbb{F}_2$-subquotients.

\begin{lem} \label{gll}
Suppose $(A, i)$ is an $H\mathbb{F}_2$-subcomplex of $B$. Let $C$ be the cofiber of $i$. Let $(D, j)$ be an $H\mathbb{F}_2$-subcomplex of $C$. Define $E$ to be the homotopy pullback of $D$ along $B\rightarrow C$. We have that $E$ is an $H\mathbb{F}_2$-subcomplex of $B$. Moreover, $A$ is an $H\mathbb{F}_2$-subcomplex of $E$ with quotient $D$.

Dually, suppose $(C, q)$ is an $H\mathbb{F}_2$-quotient complex of $B$. Let $A$ be the fiber of $q$. let $(F, p)$ be an $H\mathbb{F}_2$-quotient complex of $A$. Define $G$ to be the homotopy pushout of $F$ along $A \rightarrow B$. We have that $G$ is an $H\mathbb{F}_2$-quotient complex of $B$. Moreover, $C$ is an $H\mathbb{F}_2$-quotient complex of $G$ with fiber $F$.
\end{lem}

\begin{proof}
This follows from the short exact sequences of homology induced by the following commutative diagrams of cofiber sequences and diagram chasing.
\begin{displaymath}
    \xymatrix{
A \ar@{^{(}->}[r] \ar@{=}[d] & E \ar@{->>}[r] \ar@{^{(}->}[d] & D \ar@{^{(}->}[d]^-j\\
A \ar@{^{(}->}[r]^-i & B \ar@{->>}[r] & C \\
A \ar@{^{(}->}[r] \ar@{->>}[d]^-p & B \ar@{->>}[r]^-q \ar@{->>}[d] & C \ar@{=}[d] \\
F \ar@{^{(}->}[r] & G \ar@{->>}[r] & C
    }
\end{displaymath}
\end{proof}

We first study the spectrum $P_1^6$. For attaching maps, we abuse notation and refer to a homotopy class by its detecting element in the $E_1$-page of the Atiyah-Hirzebruch spectral sequence. We use similar notation as in the algebraic case in Notation 3.3. The readers who are familiar with the notation of cell diagrams from \cite{BJM} should compare with the cell diagrams in Remark 4.8 for the intuition of the following Lemmas 4.5, 4.6 and Theorem 4.7.

\begin{lem}
There is an $H\mathbb{F}_2$-subcomplex of $P_1^5$ with a $3$-cell and a $5$-cell that forms $\Sigma^3 C\eta$.
\end{lem}

\begin{proof}
Firstly, by the solution of the Hopf invariant one problem, the top cell of $P_1^3$ splits off. It follows that $S^3$ is an $H\mathbb{F}_2$-subcomplex of $P_1^3$, and therefore an $H\mathbb{F}_2$-subcomplex of $P_1^5$.

Secondly, we consider the $H\mathbb{F}_2$-quotient complex $P_1^5/S^3$. We claim the top cell of $P_1^5/S^3$ splits off. We prove this claim by showing the attaching map is homotopic to zero. In fact, the following composition is trivial:
$$S^4 \rightarrow P_1^4/S^3 \rightarrow S^4,$$
where the second map is the quotient map. Otherwise, we would have a nontrivial
$$Sq^1:H^4(P_1^5/S^3)\rightarrow H^5(P_1^5/S^3),$$
which we don't. This shows that the attaching map factors through $P_1^2$.
\begin{displaymath}
	\xymatrix{
		S^4 \ar[rr] \ar@{-->}[dr] & & P_1^4/S^3 \\
		& P_1^2 \ar[ur] &
		}
\end{displaymath}
The group $\pi_4(P_1^2)$ is generated by $\eta^2[2]$ and $\nu[1]$. However, the element $\eta^2[2]$ is killed by $\eta[4]$ in the Atiyah-Hirzebruch spectral sequence of $P_1^4/S^3$. The element $\nu[1]$ does not detect the attaching map either, since otherwise we would have a nontrivial
$$Sq^4:H^1(P_1^5/S^3)\rightarrow H^5(P_1^5/S^3),$$
which we don't. Therefore, the attaching map $S^4 \rightarrow P_1^4/S^3$ is trivial, and $S^5$ is an $H\mathbb{F}_2$-subcomplex of $P_1^5/S^3$.

Now we pull back $S^5$ along the quotient map $P_1^5 \rightarrow P_1^5/S^3$. We claim that we have $\Sigma^3 C\eta$ as an $H\mathbb{F}_2$-subcomplex of $P_1^5$.
\begin{displaymath}
    \xymatrix{
S^3 \ar@{^{(}->}[r] \ar@{=}[d] & \Sigma^3 C\eta \ar@{->>}[r] \ar@{^{(}->}[d] & S^5 \ar@{^{(}->}[d]\\
S^3 \ar@{^{(}->}[r] & P_1^5 \ar@{->>}[r] & P_1^5/S^3
    }
\end{displaymath}
In fact, by Lemma 4.4, we have an $H\mathbb{F}_2$-subcomplex of $P_1^5$ with nontrivial $H^3$ and $H^5$. Since there is a nontrivial
$$Sq^2:H^3(P_1^5)\rightarrow H^5(P_1^5),$$
we must have $\Sigma^3 C\eta$ as the $H\mathbb{F}_2$-subcomplex.
\end{proof}

\begin{lem}
If we quotient out the $H\mathbb{F}_2$-subcomplex $\Sigma^3 C\eta$ in $P_1^6$, then the $6$-cell splits off. Therefore, $S^6$ is an $H\mathbb{F}_2$-subcomplex of $P_1^6/\Sigma^3 C\eta$.
\end{lem}

\begin{proof}
We claim that the attaching map $S^5 \rightarrow P_1^4/S^3$ is trivial.

In fact, the group $\pi_5(P_1^4/S^3)\cong \mathbb{Z}/2$, generated by $\eta[4]$. To compute it, note that the $E_1$-page of the Atiyah-Hirzebruch spectral sequence of $P_1^4/S^3$ is $\pi_5(S^1)\oplus\pi_5(S^2)\oplus\pi_5(S^4)=\mathbb{Z}/8\oplus\mathbb{Z}/2$, generated by $\nu[2]$ and $\eta[4]$. We have the following Atiyah-Hirzebruch differentials:
\begin{equation*}
\begin{split}
\nu[2] \rightarrow & 2\nu[1]\\
2\nu[2] \rightarrow & 4\nu[1]\\
\eta^2[4] \rightarrow & 4\nu[2] = \eta^3[2]
\end{split}
\end{equation*}
Therefore, the element $\eta[4]$ is the only one left in the $E_\infty$-page.

Since we have
$$Sq^2 = 0 : H^4(P_1^6)\rightarrow H^6(P_1^6),$$
we must have
$$Sq^2 = 0 : H^4(P_1^6/\Sigma^3 C\eta)\rightarrow H^6(P_1^6/\Sigma^3 C\eta).$$
Therefore, the attaching map is not detected by $\eta[4]$, and is trivial. This proves that $S^6$ is an $H\mathbb{F}_2$-subcomplex of $P_1^6/\Sigma^3 C\eta$.
\end{proof}

\begin{thm} \label{p356}
There is an $H\mathbb{F}_2$-subcomplex $Y$ of $P_1^6$ consisting of the $3$-cell, $5$-cell and the $6$-cell, which is the pullback of $S^6$ along the quotient map $P_1^6 \rightarrow P_1^6/\Sigma^3 C\eta$.
\begin{displaymath}
    \xymatrix{
\Sigma^3 C\eta \ar@{^{(}->}[r] \ar@{=}[d] & Y \ar@{->>}[r] \ar@{^{(}->}[d] & S^6 \ar@{^{(}->}[d]\\
\Sigma^3 C\eta \ar@{^{(}->}[r] & P_1^5 \ar@{->>}[r] & P_1^5/\Sigma^3 C\eta
    }
\end{displaymath}
\end{thm}

\begin{proof}
This follows directly from Lemmas 4.3 and 4.6.
\end{proof}

\begin{rem}
The cell diagrams of the cofiber sequences in Theorem 4.7 are the following:
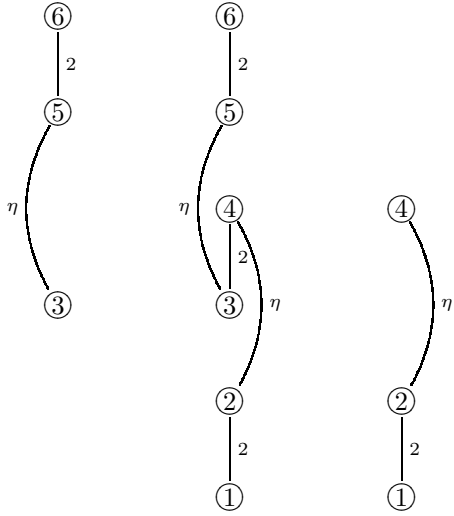
\begin{displaymath}
    \xymatrix{
   *+[o][F-]{6} \ar@{-}[d]^{2} & & *+[o][F-]{6} \ar@{-}[d]^{2} & & \\
   *+[o][F-]{5} \ar@{-}@/_1pc/[dd]_{\eta} & & *+[o][F-]{5} \ar@{-}@/_1pc/[dd]_{\eta} & & \\
   & & *+[o][F-]{4} \ar@{-}[d]^{2} \ar@{-}@/^1pc/[dd]^{\eta} & & *+[o][F-]{4} \ar@{-}@/^1pc/[dd]^{\eta}\\
   *+[o][F-]{3} & & *+[o][F-]{3} & & \\
   & & *+[o][F-]{2} \ar@{-}[d]^{2} & & *+[o][F-]{2} \ar@{-}[d]^{2} \\
   & & *+[o][F-]{1} & & *+[o][F-]{1}
    }
\end{displaymath}

\end{rem}

\section{Some $H\mathbb{F}_2$-subquotients of $P_1^\infty$}

In this section, we discuss the cell structures of certain $H\mathbb{F}_2$-subquotients of $P_1^\infty$. All of them turn out to be $H\mathbb{F}_2$-subcomplexes of a 9 cell complex $X$. The existence of these $H\mathbb{F}_2$-subquotients is used extensively in the proofs in Sections 8, 9 and 10. For illustration purpose, we include the cell diagrams of these $H\mathbb{F}_2$-subquotients. The definition of cell diagrams is reviewed in Section 13, which is Appendix I.

We define the 9 cell complex $X$.

\begin{defn}
Recall that the 15-skeleton of $P_{14}^{23}$ is $P_{14}^{15} = S^{14} \vee S^{15}$. The complex $X$ is defined to be the cofiber of the inclusion map $S^{15}\hookrightarrow P_{14}^{23}$, i.e., $X$ fits into the cofiber sequence
\begin{displaymath}
    \xymatrix{
S^{15} \ar@{^{(}->}[r] & P_{14}^{23} \ar@{->>}[r] & X.
}
\end{displaymath}
We also define the 22-skeleton of $X$ to be $\widetilde{X}$. In other words, $\widetilde{X}$ fits into the cofiber sequence
\begin{displaymath}
    \xymatrix{
S^{15} \ar@{^{(}->}[r] & P_{14}^{22} \ar@{->>}[r] & \widetilde{X}.
}
\end{displaymath}
\end{defn}

Now we establish the following lemmas on the cell structure of $X$.

\begin{lem} \label{p16}
There is a quotient map $X \twoheadrightarrow S^{16}$.
\end{lem}

\begin{proof}
There is a quotient map $P_0^7\twoheadrightarrow S^0$, since the bottom cell splits off. By James periodicity, this gives a quotient map $P_{16}^{23}\twoheadrightarrow S^{16}$. Since the 14-skeleton of $X$ is $S^{14}$, we have a quotient map to its cofiber $P_{16}^{23}$.
\begin{displaymath}
    \xymatrix{
S^{14} \ar@{^{(}->}[r] & X \ar@{->>}[r] & P_{16}^{23}.
}
\end{displaymath}
Pre-composing the quotient map $P_{16}^{23}\twoheadrightarrow S^{16}$ with the quotient map $X\twoheadrightarrow P_{16}^{23}$, we get the desired quotient map $X\twoheadrightarrow S^{16}$.
\end{proof}

\begin{lem}
We have $S^{17}$ as an $H\mathbb{F}_2$-subcomplex of $\widetilde{X}$ and of $X$.
\end{lem}

\begin{proof}
We claim that the top cell of the 17-skeleton of $\widetilde{X}$ splits off, and therefore $S^{17}$ is an $H\mathbb{F}_2$-subcomplex of $\widetilde{X}$ and $X$.

The 16-skeleton of $\widetilde{X}$ is $\Sigma^{14} C\eta$ because of the nontrivial $Sq^2$. The group $\pi_{16}(\Sigma^{14} C\eta)$ is generated by $2[16]$. Note that in the Atiyah-Hirzebruch spectral sequence, the element $\eta^2[14]$ is killed by $\eta[16]$. Therefore, it follows from James periodicity that the attaching map is trivial.
\end{proof}

Now we define some $H\mathbb{F}_2$-subcomplexes of $X$. The relationships among the $H\mathbb{F}_2$-subcomplexes are summarized in Remark 5.12. The reader should compare with the cell diagrams in Remark 5.13 for the intuition of the following definitions.

\begin{defn}
We define $\widehat{X^{20}}$ to be the 20-skeleton of $X$, and $X^{20}$ to be the fiber of the following composition:
\begin{displaymath}
    \xymatrix{
\widehat{X^{20}} \ar@{^{(}->}[r]  & \widetilde{X} \ar@{->>}[r] & S^{16}.
    }
\end{displaymath}
Note that the composition is a quotient map, and therefore $X^{20}$ is an $H\mathbb{F}_2$-subcomplex of $\widehat{X^{20}}$.
\end{defn}

\begin{defn}
Quotienting out the 16-skeleton of $\widetilde{X}$, we have the $H\mathbb{F}_2$-quotient complex $P_{17}^{22}$. We define $\widehat{X^{22}}$ to be the pullback of $\Sigma^{16} Y$ along the quotient map $\widetilde{X} \rightarrow P_{17}^{22}$. Note that by Theorem 4.7 and James periodicity, $\Sigma^{16} Y$ is an $H\mathbb{F}_2$-subcomplex of $P_{17}^{22}$.

\begin{displaymath}
    \xymatrix{
\Sigma^{14} C\eta \ar@{^{(}->}[r] \ar@{=}[d] & \widehat{X^{22}} \ar@{->>}[r] \ar@{^{(}->}[d] & \Sigma^{16} Y \ar@{^{(}->}[d]\\
\Sigma^{14} C\eta \ar@{^{(}->}[r] & \widetilde{X} \ar@{->>}[r] & P_{17}^{22} = \Sigma^{16} P_1^6
    }
\end{displaymath}
\end{defn}

\begin{defn}
We define $X^{22}$ to be the fiber of the following composition:
\begin{displaymath}
    \xymatrix{
\widehat{X^{22}} \ar@{^{(}->}[r]  & \widetilde{X} \ar@{->>}[r] & S^{16}.
    }
\end{displaymath}
Note that the composition is a quotient map, and therefore $X^{22}$ is an $H\mathbb{F}_2$-subcomplex of $\widehat{X^{22}}$.
\end{defn}

\begin{defn}
We define $\widehat{X^{21}}$ to be the 21-skeleton of $\widehat{X^{22}}$, and $X^{21}$ to be the 21-skeleton of $X^{22}$.
\end{defn}

\begin{rem}
Note that $S^{19}$ is an $H\mathbb{F}_2$-subcomplex of $X^{21}$. In fact, the 19-skeleton of $X^{21}$ is $S^{19}\vee S^{14}$. The attaching map $S^{18} \rightarrow S^{14}$ is trivial since $\pi_4=0$.
\end{rem}

\begin{defn}
The top cell of $P_1^7$ splits off due to the solution of the Hopf invariant one problem. By James periodicity, this implies that the top cell of $P_{17}^{23}$ splits off. Therefore, $S^{23}$ is an $H\mathbb{F}_2$-subcomplex of $P_{17}^{23}$.

We define $\widehat{X^{23}}$ to be the pullback of $S^{23}$ along the quotient map $X \rightarrow P_{16}^{23}$.
\begin{displaymath}
    \xymatrix{
\Sigma^{14} C\eta \ar@{^{(}->}[r] \ar@{=}[d] & \widehat{X^{23}} \ar@{->>}[r] \ar@{^{(}->}[d] & S^{23} \ar@{^{(}->}[d]\\
\Sigma^{14} C\eta \ar@{^{(}->}[r] & X \ar@{->>}[r] & P_{17}^{23} = \Sigma^{16} P_1^7
    }
\end{displaymath}
\end{defn}

\begin{defn}
We define $X^{23}$ to be the fiber of the following composition:
\begin{displaymath}
    \xymatrix{
\widehat{X^{23}} \ar@{^{(}->}[r]  & X \ar@{->>}[r] & S^{16}.
    }
\end{displaymath}
Note that the composition is a quotient map, and therefore $X^{23}$ is an $H\mathbb{F}_2$-subcomplex of $\widehat{X^{23}}$.
\end{defn}

\begin{rem}
We do not know if the top cell of $X^{23}$ splits off. If not, then the attaching map is detected by a nontrivial homotopy class in $\pi_8$. Since homotopy classes in $\pi_8$ have Adams filtration at least 2, $Ext(X^{23})$ splits as a direct sum of $Ext(S^{14})$ and $Ext(S^{23})$ in either case.
\end{rem}

\begin{rem}
We summarize in the following diagram the relationships among the $H\mathbb{F}_2$-subcomplexes defined in Definitions 5.4, 5.5, 5.6, 5.7, 5.9 and 5.10. For the name convention, we have been using the notation $X^n$, not to be confused with the $n-$skeleton of $X$, to indicate a kind of ``$n-$skeleton" to the eyes of mod 2 homology, and the notation $\widehat{X^n}$ to indicate ``adding" the 16-cell to $X^n$. The cases for $n = 23$ do not necessarily follow this convention, since we do not know if the top cell of $X^{23}$ splits off.
\begin{displaymath}
    \xymatrix{
    & & & & P_{14}^{23} \ar@{->>}[d] \\
    & X^{23} \ar@{^{(}->}[r] & \widehat{X^{23}} \ar@{^{(}->}[r] & \widetilde{X} \ar@{=}[d] \ar@{^{(}->}[r] & X \\
    & X^{22} \ar@{^{(}->}[r] & \widehat{X^{22}} \ar@{^{(}->}[r] & \widetilde{X} \ar@{=}[d] & \\
    S^{19} \ar@{^{(}->}[r] & X^{21} \ar@{^{(}->}[r] \ar@{^{(}->}[u] & \widehat{X^{21}} \ar@{^{(}->}[r] \ar@{^{(}->}[u] & \widetilde{X} \ar@{=}[d] & \\
    & X^{20} \ar@{^{(}->}[r] & \widehat{X^{20}} \ar@{^{(}->}[r] & \widetilde{X} & \\
    }
\end{displaymath}
In Section 8, we need to show certain elements in $Ext(X)$ are permanent cycles. We will show these elements are permanent cycles in the corresponding $H\mathbb{F}_2$-subcomplexes, and use the naturality of Adams spectral sequences and the algebraic Atiyah-Hirzebruch spectral sequences to show they are permanent cycles in $X$. The intuition of finding these $H\mathbb{F}_2$-subcomplexes is due to the rearrangement of the cell diagram of $\widetilde{X}$. Following the cell diagram, one could reconstruct $\widetilde{X}$ layer by layer. Firstly, consider the cells in the bottom layer: $S^{14}\vee S^{17} \vee S^{19}$. Secondly, attach the cells in the next layer: the ones in dimension 16, 18 and 21. Lastly, attach the cells in dimension 20 and 22. Any $H\mathbb{F}_2$-subcomplex consists of a collection of cells, such that for each cell contained in this collection, any cells in lower layers that this cell is attached to are also contained in this collection. The reader should compare this with the cell diagrams in Remark 5.13.
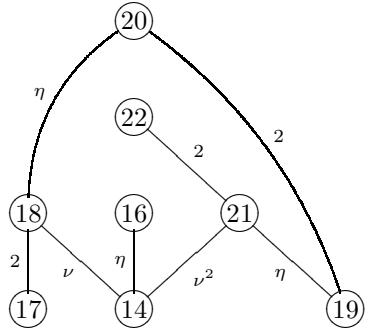
\begin{displaymath}
    \xymatrix{
    & *+[o][F-]{20} \ar@{-}@/_1pc/[ddl]_{\eta} \ar@{-}@/^1pc/ [dddrr]^{2} & &\\
    & *+[o][F-]{22} \ar@{-}[dr]^{2} & & \\
   *+[o][F-]{18} \ar@{-}[d]_{2} \ar@{-}[dr]_{\nu} & *+[o][F-]{16} \ar@{-}[d]_{\eta} & *+[o][F-]{21} \ar@{-}[dr]_{\eta} \ar@{-}[dl]^{\nu^2} & \\
  *+[o][F-]{17}  & *+[o][F-]{14} & & *+[o][F-]{19}
    }
\end{displaymath}
\end{rem}

\begin{rem}
For readers who are familiar with the notation of cell diagrams from \cite{BJM}, we include the cell diagrams as illustrations of the $H\mathbb{F}_2$-subcomplexes we defined. The definition and some examples of cell diagrams are explained in Appendix I.

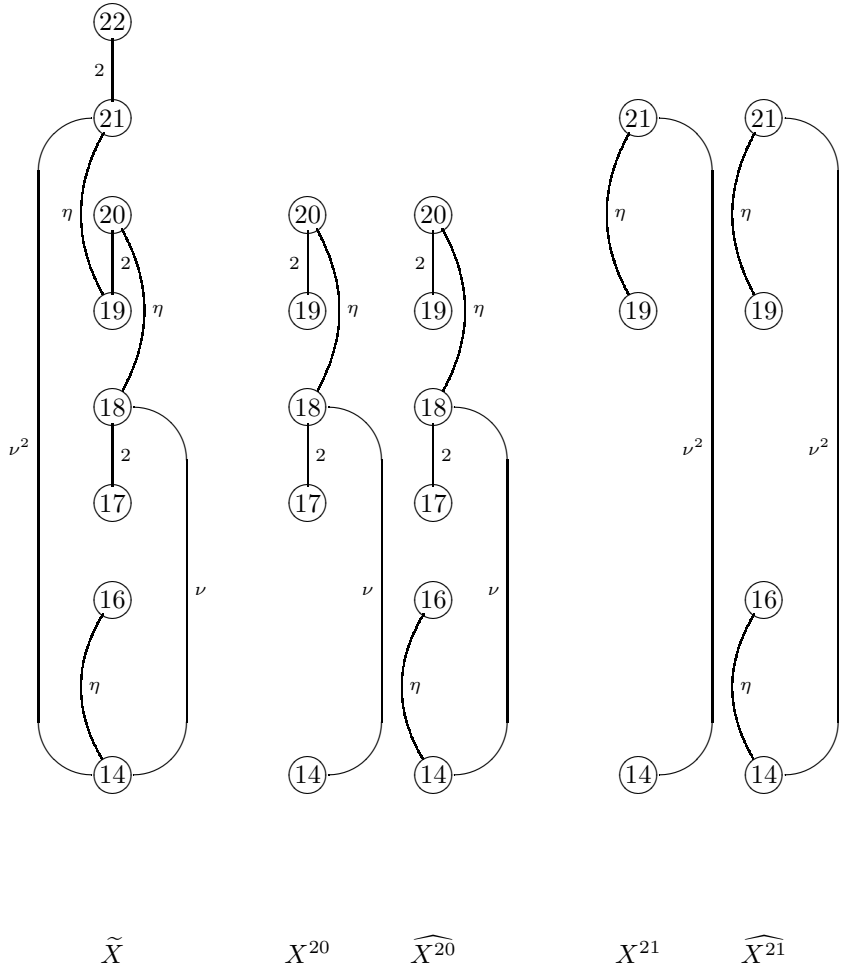
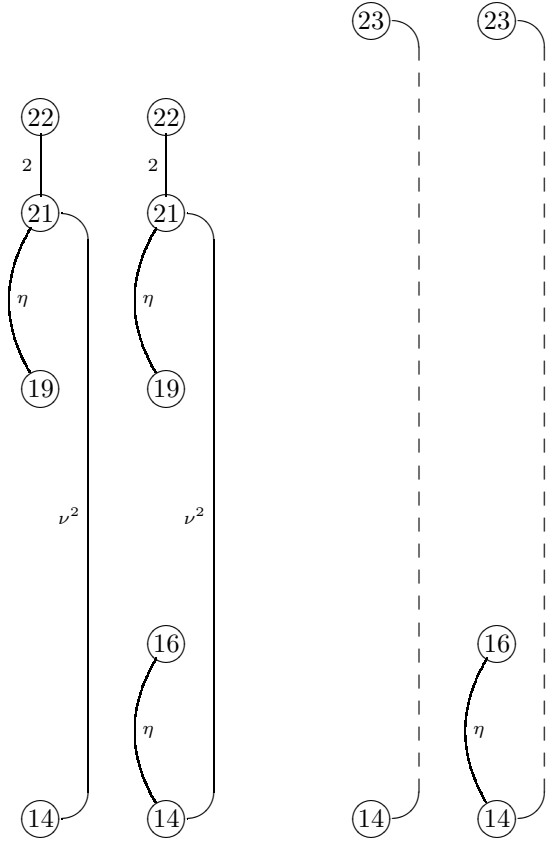
\begin{displaymath}
    \xymatrix{
    *+[o][F-]{22} \ar@{-}[d]_{2} & & & & & &  \\
    *+[o][F-]{21} \ar@{-} `l/20pt[ddddddd] `[ddddddd]_{\nu^2} [ddddddd] \ar@{-}@/_1pc/[dd]_{\eta} & & & & & *+[o][F-]{21} \ar@{-}@/_1pc/[dd]^{\eta} \ar@{-}`r[ddddddd] `[ddddddd]_{\nu^2} [ddddddd] & *+[o][F-]{21} \ar@{-}@/_1pc/[dd]^{\eta} \ar@{-}`r[ddddddd] `[ddddddd]_{\nu^2} [ddddddd]  \\
    *+[o][F-]{20} \ar@{-}[d]^{2} \ar@{-}@/^1pc/[dd]^{\eta} & & *+[o][F-]{20} \ar@{-}[d]_{2} \ar@{-}@/^1pc/[dd]^{\eta} & *+[o][F-]{20} \ar@{-}[d]_{2} \ar@{-}@/^1pc/[dd]^{\eta}  & & &  \\
    *+[o][F-]{19} & & *+[o][F-]{19} & *+[o][F-]{19} & & *+[o][F-]{19} & *+[o][F-]{19} \\
    *+[o][F-]{18} \ar@{-}[d]^{2} \ar@{-} `r[dddd] `[dddd]^{\nu} [dddd] & & *+[o][F-]{18} \ar@{-}[d]^{2} \ar@{-} `r[dddd] `[dddd]_{\nu} [dddd] & *+[o][F-]{18} \ar@{-}[d]^{2} \ar@{-} `r[dddd] `[dddd]_{\nu} [dddd] & & &  \\
    *+[o][F-]{17} & & *+[o][F-]{17} & *+[o][F-]{17} & & & \\
    *+[o][F-]{16} \ar@{-}@/_1pc/[dd]^{\eta} & &  & *+[o][F-]{16} \ar@{-}@/_1pc/[dd]^{\eta} & & & *+[o][F-]{16} \ar@{-}@/_1pc/[dd]^{\eta} \\
   & & & & & &  \\
    *+[o][F-]{14} & & *+[o][F-]{14} & *+[o][F-]{14} & & *+[o][F-]{14} & *+[o][F-]{14} \\
   & & & & & &  \\
    \widetilde{X} & & X^{20} & \widehat{X^{20}} & & X^{21} & \widehat{X^{21}} 
    }
\end{displaymath}

\begin{displaymath}
    \xymatrix{
    & & & *+[o][F-]{23} \ar@{--}`r[ddddddddd] `[ddddddddd] [ddddddddd] & *+[o][F-]{23} \ar@{--}`r[ddddddddd] `[ddddddddd] [ddddddddd]\\
     *+[o][F-]{22} \ar@{-}[d]_{2} & *+[o][F-]{22} \ar@{-}[d]_{2} & & & \\
     *+[o][F-]{21} \ar@{-}@/_1pc/[dd]^{\eta} \ar@{-}`r[ddddddd] `[ddddddd]_{\nu^2} [ddddddd] & *+[o][F-]{21} \ar@{-}@/_1pc/[dd]^{\eta} \ar@{-}`r[ddddddd] `[ddddddd]_{\nu^2} [ddddddd] & & & \\
     & & & & \\
     *+[o][F-]{19} & *+[o][F-]{19} & & &\\
    & & & & \\
    & & & &\\
   & *+[o][F-]{16} \ar@{-}@/_1pc/[dd]^{\eta} & & & *+[o][F-]{16} \ar@{-}@/_1pc/[dd]^{\eta}\\
    & & & & \\
     *+[o][F-]{14} & *+[o][F-]{14} & & *+[o][F-]{14} & *+[o][F-]{14} \\
 & & & & \\
     X^{22} & \widehat{X^{22}} & & X^{23} & \widehat{X^{23}}
    }
\end{displaymath}
Here the dashed lines in $X^{23}$ and $\widehat{X^{23}}$ mean some possible attaching maps as explained in Remark 5.11.

For the cell diagram of $\widetilde{X}$, note that we have a nonzero $Sq^8$ on $H^{14}(\widetilde{X})$. However, $\Sigma^{14} C \sigma$ is not an $H\mathbb{F}_2$-subquotient of $\widetilde{X}$, we therefore do not draw the attaching map $\sigma$. The non-existence of the $H\mathbb{F}_2$-subquotient is due to the existence of the attaching map $\nu^2$, which is proved in the following Theorem 5.14.
\end{rem}

By Remark 5.8, we have $S^{19}$ as an $H\mathbb{F}_2$-subcomplex of $X^{21}$. The cofiber $X^{21}/S^{19}$ is therefore a 2 cell complex with cells in dimension 14 and 21. We have the following theorem.

\begin{thm}
The complex $X^{21}/S^{19}$ is $\Sigma^{14} C \nu^2$, where $C \nu^2$ is the cofiber of $\nu^2$.
\end{thm}

This theorem implies the following corollary.

\begin{cor}
The complex $\Sigma^{14} C \nu^2$ is an $H\mathbb{F}_2$-subquotient of $X^{21}$, $\widehat{X^{21}}$, $X^{22}$ and $\widehat{X^{22}}$.
\end{cor}

In the rest of this section, we prove Theorem 5.14. Note that since $\pi_6 = \mathbb{Z}/2$ is generated by $\nu^2$, the complex $X^{21}/S^{19}$ is either $\Sigma^{14} C \nu^2$ or $S^{14}\vee S^{21}$. Theorem 5.14 and Corollary 5.15 will be used in several proofs in Section 6. However, the proofs in Section 6 do not depend on these results. In fact, if the complex $X^{21}/S^{19}$ were $S^{14}\vee S^{21}$, the proofs in Section 6 would be strictly much easier. The reader should feel free to skip the proof of Theorem 5.14: knowing either case could be true is good enough for the proofs in Section 8. Since this theorem may be of other interest, we include the proof of Theorem 5.14 for completeness.

To prove Theorem 5.14, we first consider the spectrum $\mathbb{C}P_1^3$, which is the suspension spectrum of $\mathbb{C}P^3$. As we will explain in Example 13.5, the top cell does not split off and is attached to $\mathbb{C}P_1^2$ via $2\nu[2]$. We have a standard quotient map $P_1^7 \rightarrow \mathbb{C}P_1^3$, which is induced by the quotient map on the space level. Then pre-composing it with the inclusion map, we have a map
$$q : P_1^6 \twoheadrightarrow \mathbb{C}P_1^3.$$
Recall that in Theorem 4.7, we showed that there exists a 3 cell complex $Y$, which is an $H\mathbb{F}_2$-subcomplex of $P_1^6$.

\begin{thm}
The composition
\begin{displaymath}
	\xymatrix{
S^3 \ar@{^{(}->}[r] & Y \ar@{^{(}->}[r] & P_1^6 \ar@{->>}[r]^q & \mathbb{C}P_1^3
}
\end{displaymath}
is trivial, therefore the composition
\begin{displaymath}
	\xymatrix{
 Y \ar@{^{(}->}[r] & P_1^6 \ar@{->>}[r]^q & \mathbb{C}P_1^3
}
\end{displaymath}
maps through $P_5^6$. Furthermore, the composition
\begin{displaymath}
	\xymatrix{
S^5 \ar@{^{(}->}[r] & P_5^6 \ar[r] & \mathbb{C}P_1^3
}
\end{displaymath}
is nontrivial, and detected by $\nu[2]$ in the Atiyah-Hirzebruch spectral sequence of $\mathbb{C}P_1^3$.
\end{thm}

\begin{rem}
We have the following commutative diagram:
\begin{displaymath}
	\xymatrix{
		Y \ar@{^{(}->}[r] \ar@{->>}[d] & P_1^6 \ar@{->>}[r]^q & \mathbb{C}P_1^3 \\
		P_5^6 \ar[rru] \\
		S^5 \ar@{^{(}->}[u] \ar[rr]^{\nu} & & S^2 \ar@{^{(}->}[uu]
		}
\end{displaymath}
In other words, the cell diagrams of the composition $Y \rightarrow \mathbb{C}P_1^3$ can be described as follows:
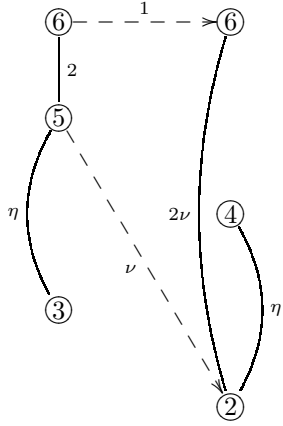
\begin{displaymath}
	\xymatrix{
  *+[o][F-]{6} \ar@{-}[d]^2 \ar@{-->}[rr]^{1} & &  *+[o][F-]{6} \ar@{-}@/_1pc/[dddd]_{2\nu} \\
   *+[o][F-]{5} \ar@{-}@/_1pc/[dd]_{\eta} \ar@{-->}[rrddd]_{\nu} & & \\
   & & *+[o][F-]{4} \ar@{-}@/^1pc/[dd]^{\eta} \\
   *+[o][F-]{3} & & \\
  & &  *+[o][F-]{2}
    }
\end{displaymath}
\end{rem}

\begin{proof}
The first claim of Theorem 5.16 follows from the fact that $\pi_3(\mathbb{C}P_1^3) = 0$. In fact, in the $E_1$-page of the Atiyah-Hirzebruch spectral sequence of $\mathbb{C}P_1^3$, there is only one candidate that lies in the degree that converges to $\pi_3$: $\eta[2]$. However, because of the attaching map in $\mathbb{C}P_1^2$, we have an Atiyah-Hirzebruch differential
$$1[4] \rightarrow \eta[2].$$
Therefore, $\pi_3(\mathbb{C}P_1^3) = 0$.

For the second claim, we first show that the composition
\begin{displaymath}
	\xymatrix{
S^5 \ar@{^{(}->}[r] & P_5^6 \ar[r] & \mathbb{C}P_1^3
}
\end{displaymath}
maps through $S^2$. This follows from the fact that $\pi_5(\mathbb{C}P_1^3) = \mathbb{Z}/2$, generated by $\nu[2]$. In fact, because of the attaching maps in $\mathbb{C}P_1^3$, we have the Atiyah-Hirzebruch differentials
\begin{equation*}
\begin{split}
1[6] \rightarrow & 2\nu[2]\\
2[6] \rightarrow & 4\nu[2]\\
\eta[4] \rightarrow & \eta^2[2],
\end{split}
\end{equation*}
which leave $\nu[2]$ as the only nontrivial element in the Atiyah-Hirzebruch $E_\infty$-page that converges to $\pi_5(\mathbb{C}P_1^3)$.

Next, we consider the following commutative diagram of cofiber sequences
\begin{displaymath}
	\xymatrix{
S^5 \ar[r]^{2\nu[2]} & \mathbb{C}P_1^2 \ar@{^{(}->}[r] & \mathbb{C}P_1^3 \ar@{->>}[r] & S^6 \\
S^5 \ar[r]^2 \ar@{-->}[u] & S^5 \ar@{^{(}->}[r] \ar[u] & P_5^6 \ar[u]  \ar@{->>}[r] & S^6 \ar@{-->}[u]
}
\end{displaymath}
Since the composition
\begin{displaymath}
	\xymatrix{
S^5 \ar@{^{(}->}[r] & P_5^6 \ar[r] & \mathbb{C}P_1^3 \ar@{->>}[r] & S^6
}
\end{displaymath}
is trivial, it maps through the quotient $P_5^6/S^5 = S^6$. Since the map $P_5^6 \rightarrow \mathbb{C}P_1^3$ induces an isomorphism on $H^6$, so does $S^6\dashrightarrow S^6$. Therefore, we can choose it to be the identity map. To make the left square commute, we must identify the map $S^5 \rightarrow \mathbb{C}P_1^2$ as $\nu[2]$ modulo the indeterminacy $2\nu[2]$. Therefore, the composition
\begin{displaymath}
	\xymatrix{
S^5 \ar@{^{(}->}[r] & \mathbb{C}P_1^2 \ar[r] & \mathbb{C}P_1^3
}
\end{displaymath}
is nontrivial, and detected by $\nu[2]$ in the Atiyah-Hirzebruch spectral sequence of $\mathbb{C}P_1^3$.
\end{proof}

\begin{proof}[Proof of Theorem 5.14]
We show that there is an attaching map $\nu^2$ in $X^{21}$.

Firstly, we have a quotient map
$$P_{-2}^{6}\rightarrow \mathbb{C}P_{-1}^{3},$$
which is induced by the quotient map $\mathbb{R}P_{14}^{22} \rightarrow \mathbb{C}P_7^{11}$ on the space level and James periodicity. It maps through $\Sigma^{-16}\widetilde{X}$, since $\pi_{-1}(\mathbb{C}P_{-1}^{3}) = 0$. In fact, in the Atiyah-Hirzebruch spectral sequence of $\mathbb{C}P_{-1}^{3}$, we have a differential
$$1[0] \rightarrow \eta[-2],$$
which kills the only nontrivial element $\eta[-2]$ in the $E_1$-page.
\begin{displaymath}
	\xymatrix{
S^{-1} \ar@{^{(}->}[r] & P_{-2}^{6} \ar@{->>}[r] \ar@{->>}[d] & \Sigma^{-16}\widetilde{X} \ar@{->>}[dl] \\
& \mathbb{C}P_{-1}^{3} &
}
\end{displaymath}

Secondly, by Theorem 5.16, we have the following commutative diagram
\begin{displaymath}
	\xymatrix{
P_5^6 \ar[rrd] && S^5 \ar[ll] \ar[d] \ar[rd]^\nu & \\
Y \ar@{->>}[u] \ar@{^{(}->}[r]	& P_1^6 \ar@{->>}[r] \ar[d] & \mathbb{C}P_1^3 \ar[d] & S^2 \ar@{^{(}->}[l] \ar[ld]^\nu \\
& S^{-1} \ar[r]^{id} & S^{-1} & 	
		}
\end{displaymath}
where the map $\nu: S^2 \rightarrow S^{-1}$ is due to the nontrivial $Sq^4$ on $H^{-2}(\mathbb{C}P_{-1}^{3})$.

Therefore, in the cofiber of the composition
\begin{displaymath}
	\xymatrix{
Y \ar@{^{(}->}[r] & P_1^6 \ar[r] & S^{-1},
}
\end{displaymath}
we have $\nu^2$ as an attaching map. Since this cofiber is $\Sigma^{-15}X^{22}$, this proves the attaching map $\nu^2$ in $X^{21}$.
\end{proof}

\section{Two lemmas on Atiyah-Hirzebruch differentials}

In this section, we establish two general lemmas regarding the relationship of 3-fold Toda brackets and differentials in the Atiyah-Hirzebruch spectral sequences of certain 3 and 4 cell complexes. As examples, we use these lemmas to prove Proposition 6.3 and 6.4, whose statements will be used in Section 8.

We recall some facts from the construction of the Atiyah-Hirzebruch spectral sequence. Let $X$ be a complex with at most one cell in each dimension. Let $X^n$ denote its $n$-skeleton. Not to be confused with the notation we use in the rest of this paper, the $n$-skeleton notation only applies in the next four pages.

We have the following facts about the Atiyah-Hirzebruch spectral sequence of $X$:
\begin{enumerate}
\item The $E_1$-page is
$$E_1^{s,t}=\pi_{t}(X^{s}/X^{s-1}).$$
As used in the previous two sections, we denote any element in the $E_1$-page to be $\alpha[s]$, where $\alpha$ is an element in the stable homotopy groups of spheres, and $s$ suggests its Atiyah-Hirzebruch filtration. We will abuse the notation and write the same symbol $\alpha[s]$ for an element in $\pi_\ast(X)$.
\item The $E_r$-page is
$$E_r^{s,t}={{Im(\pi_t(X^s/X^{s-r}) \rightarrow\pi_t(X^s/X^{s-1}))}\over{Im(\pi_{t+1}(X^{s+r-1}/X^s)\rightarrow\pi_t(X^s/X^{s-1}))}},$$
where the top map is induced by the quotient map
$$X^s/X^{s-r} \twoheadrightarrow X^s/X^{s-1},$$
and the bottom map is induced by the attaching map in the cofiber sequence
\begin{displaymath}
    \xymatrix{
    X^s/X^{s-1} \ar@{^{(}->}[r] & X^{s+r-1}/X^{s-1} \ar@{->>}[r] & X^{s+r-1}/X^s \ar[r] & \Sigma X^s/X^{s-1}.
    }
\end{displaymath}
\item The differential
$$d_r:E_r^{s,t}\rightarrow E_r^{s-r,t-1}$$
is defined as the following. Let $\widetilde{\alpha}$\ be a class in $\pi_t(X^s/X^{s-r})$, such that it maps to $\alpha[s]\in E_r^{s,t}$ under the projection to the top cell: $X^s/X^{s-r} \twoheadrightarrow X^s/X^{s-1}$. We define $d_r(\alpha[s])$ to be the composition of $\widetilde{\alpha}$ with the attaching map $X^s/X^{s-r}\rightarrow \Sigma X^{s-r}/X^{s-r-1}$.
\begin{displaymath}
    \xymatrix{
    S^t \ar[r]^-{\widetilde{\alpha}} & X^s/X^{s-r} \ar[r] & \Sigma X^{s-r}/X^{s-r-1}.
    }
\end{displaymath}
One can check that this is well-defined.
\item Suppose we have a nontrivial differential in the Atiyah-Hirzebruch spectral sequence of $X$:
$$d_{s_1-s_2}(\alpha[s_1]) = \beta[s_2],$$
where $\alpha \in \pi_\ast(X^{s_1}/X^{s_1-1})$ and $\beta \in \pi_\ast(X^{s_2}/X^{s_2-1})$. This implies that, in the Atiyah-Hirzebruch spectral sequence of $X^{s_1-1}$, the element $\beta[s_2]$ is a permanent cycle. Furthermore, under the attaching map $S^{s_1-1} \rightarrow X^{s_1-1}$, the image of $\alpha[s_1]$ is detected by $\beta[s_2]$.
\end{enumerate}

We have the following lemma to compute differentials in the Atiyah-Hirzebruch spectral sequence of 3 cell complexes:
\begin{lem}
Let $T$ be a three cell complex with cells in dimensions $t_1, t_2, t_3$, where $t_3 < t_2 < t_1$. Suppose we have cofiber sequences
\begin{displaymath}
    \xymatrix{
    \Sigma^{t_3} C \gamma \ar@{^{(}->}[r]^-{i_1} & T \ar@{->>}[r]^{q_1} & S^{t_1} \ar[r]^-{a_1} & \Sigma^{t_3 + 1} C \gamma \\
    S^{t_3} \ar@{^{(}->}[r]^{i_2} & T \ar@{->>}[r]^-{q_2} & \Sigma^{t_2} C \beta \ar[r]^{a_2} & \Sigma S^{t_3},
    }
\end{displaymath}
where $C \beta$ is the cofiber of $\beta \in \pi_{t_1 - t_2 -1}$, $C \gamma$ is the cofiber of $\gamma \in \pi_{t_2 - t_3 -1}$ and $\beta$, $\gamma$ are nontrivial classes such that $\beta \cdot \gamma = 0$. In other words, the cell diagram of $T$ is the following:
\begin{displaymath}
    \xymatrix{
    *+[o][F-]{t_1} \ar@{-}[d]_{\beta}  \\
    *+[o][F-]{t_2} \ar@{-}[d]_{\gamma} \\
    *+[o][F-]{t_3}
    }
\end{displaymath}
Suppose the class $\alpha\in\pi_{t_0}$ satisfies the condition: $\alpha \cdot \beta = 0$ in $\pi_{t_0 + t_1 - t_2 - 1}$.
Then we have an Atiyah-Hirzebruch differential:
$$d_{t_1 - t_3}(\alpha[t_1]) \subseteq \langle \alpha, \beta ,\gamma \rangle[t_3].$$
If moreover
$\alpha \cdot \pi_{t_1 - t_3 - 1}\subseteq \gamma \cdot \pi_{t_0 + t_1 - t_2}$ in $\pi_{t_0 + t_1 - t_3 - 1}$, then we have an Atiyah-Hirzebruch differential:
$$d_{t_1 - t_3}(\alpha[t_1]) = \langle \alpha, \beta ,\gamma \rangle[t_3].$$
Here the indeterminacy of $\langle \alpha, \beta ,\gamma \rangle[t_3]$ is zero in the $E_{t_1-t_3}$-page.

Furthermore, in the latter case, if $0 \in \langle \alpha, \beta ,\gamma \rangle$, then $\alpha[t_1]$ is a permanent cycle in the Atiyah-Hirzebruch spectral sequence of $T$.
\end{lem}

\begin{proof}
Following the condition $\alpha \cdot \beta = 0$, $\alpha[t_1]$ survives in the Atiyah-Hirzebruch spectral sequence of $\Sigma^{t_2} C \beta$. In fact, this follows from the long exact sequence of homotopy groups associated to the cofiber sequence
\begin{displaymath}
    \xymatrix{
    S^{t_2} \ar@{^{(}->}[r] & \Sigma^{t_2} C \beta \ar@{->>}[r] & S^{t_1}.
    }
\end{displaymath}
By naturality of the Atiyah-Hirzebruch spectral sequence induced by the quotient map $T \twoheadrightarrow \Sigma^{t_2} C \beta$, we have the differential in the Atiyah-Hirzebruch spectral sequence of $T$:
$$d_{t_1 - t_2} (\alpha[t_1]) = 0.$$

Now consider any class in $\pi_{t_0 + t_1}(\Sigma^{t_2} C \beta)$ which is detected by $\alpha[t_1]$. We abuse the notation to denote such a class by $\alpha[t_1]$. By the definition of the Toda bracket $\langle \alpha, \beta ,\gamma \rangle$, the class $a_{2\ast}(\alpha[t_1])$ is an element in $\langle \alpha, \beta ,\gamma \rangle[t_3]$.
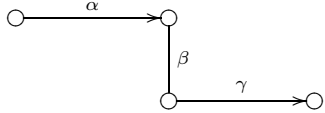
\begin{displaymath}
    \xymatrix{
 *+[o][F-]{} \ar[r]^{\alpha} & *+[o][F-]{} \ar@{-}[d]^{\beta}  &   \\
   & *+[o][F-]{} \ar[r]^{\gamma} & *+[o][F-]{} \\
   S^{t_0 + t_1} \ar[r] & \Sigma^{t_2} C \beta \ar[r]^{a_2} & \Sigma S^{t_3}
   }
\end{displaymath}
The indeterminacy of this Toda bracket is $\alpha \cdot \pi_{t_1 - t_3 - 1} + \gamma \cdot \pi_{t_0 + t_1 - t_2}$.
From the construction of the Atiyah-Hirzebruch spectral sequence, $a_{2\ast}(\alpha[t_1])$ is also a representative for $d_{t_1-t_3}(\alpha[t_1])$. The indeterminacy of the target of this differential is the image of
$$d_{t_2-t_3}:\pi_{t_0 + t_1 -t_2 + t_3 + 1}(S^{t_2})\rightarrow \pi_{t_0 + t_1}(\Sigma S^{t_3}),$$
which is $\gamma \cdot \pi_{t_0 + t_1 - t_2}$, since it is induced by multiplication by $\gamma$ map. Hence the first claim.

If $\alpha \cdot \pi_{t_1 - t_3 - 1}\subseteq \gamma \cdot \pi_{t_0 + t_1 - t_2}$ in $\pi_{t_0 + t_1 - t_3 - 1}$,
then $d_{t_1 - t_3}(\alpha[t_1])$ and $\langle \alpha, \beta ,\gamma \rangle[t_3]$ have a common element with the same indeterminacy. Hence the second statement.

The third statement follows directly from the second one, since the $E_{t_1 - t_3 + 1}$-page is the $E_\infty$-page for the Atiyah-Hirzebruch spectral sequence of $T$.
\end{proof}

\begin{lem}
Let $U$ be a four cell complex with cells in dimensions $t_1$, $t_2$, $t_3$, $t_4$, where $t_4 < t_3 < t_2 < t_1$. Suppose we have cofiber sequences
\begin{displaymath}
    \xymatrix{
    S^{t_3} \vee S^{t_4} \ar@{^{(}->}[r]^-{i_3} & U \ar@{->>}[r]^{q_3} & \Sigma^{t_2} C \beta \ar[r]^-{a_3} & \Sigma S^{t_3} \vee \Sigma S^{t_4} \\
    V \ar@{^{(}->}[r]^{i_4} & U \ar@{->>}[r]^-{q_4} & S^{t_1} \ar[r]^{a_4} & \Sigma V \\
    S^{t_3} \vee S^{t_4} \ar@{^{(}->}[r]^-{i_5} & V \ar@{->>}[r]^{q_5} & S^{t_2} \ar[r]^-{a_5} & \Sigma S^{t_3} \vee \Sigma S^{t_4}
    }
\end{displaymath}
where $C \beta$ is the cofiber of $\beta \in \pi_{t_1 - t_2 -1}$, the map $a_5: S^{t_2} \rightarrow \Sigma S^{t_3} \vee \Sigma S^{t_4}$ is defined component-wise by multiplication by $\gamma \in \pi_{t_2 - t_3 -1}$ and $\delta \in \pi_{t_2 - t_4 -1}$ map, and $\beta$, $\gamma$, $\delta$ are nontrivial classes such that $\beta \cdot \gamma = 0$, $\beta \cdot \delta = 0$. In other words, the cell diagram of $U$ is the following:
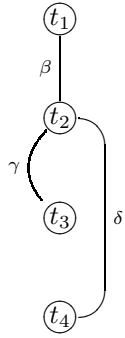
\begin{displaymath}
    \xymatrix{
    *+[o][F-]{t_1} \ar@{-}[d]_{\beta}  \\
    *+[o][F-]{t_2} \ar@{-}@/_1pc/[d]_{\gamma} \ar@{-}`r[dd] `[dd]^{\delta} [dd] \\
    *+[o][F-]{t_3} \\
    *+[o][F-]{t_4}
    }
\end{displaymath}
Suppose the class $\alpha\in\pi_{t_0}$ satisfies the following conditions:
\begin{enumerate}
\item $\alpha \cdot \beta = 0$ in $\pi_{t_0 + t_1 - t_2 - 1}$,
\item $\alpha \cdot \pi_{t_1 - t_3 - 1}\subseteq \gamma \cdot \pi_{t_0 + t_1 - t_2}$ in $\pi_{t_0 + t_1 - t_3 - 1}$,
\item $0 \in \langle \alpha, \beta, \gamma\rangle$ in $\pi_{t_0 + t_1 - t_3 - 1}$.
\end{enumerate}
We then have an Atiyah-Hirzebruch differential
$$d_{t_1 - t_4} (\alpha[t_1]) \subseteq \langle \alpha, \beta, \delta\rangle [t_4].$$
If furthermore the following two conditions are satisfied:
\begin{enumerate}
\setcounter{enumi}{3}
\item $\alpha \cdot \pi_{t_1 - t_4 - 1} = 0$ in $\pi_{t_0 + t_1 - t_4 - 1}$,
\item $\delta \cdot \pi_{t_0 + t_1 - t_2} = 0$ in $\pi_{t_0 + t_1 - t_4 - 1}$,
\end{enumerate}
then we have an Atiyah-Hirzebruch differential
$$d_{t_1 - t_4} (\alpha[t_1]) = \langle \alpha, \beta, \delta\rangle [t_4].$$
Moreover, in the latter case, if $0 \in \langle \alpha, \beta ,\delta \rangle$, then $\alpha[t_1]$ is a permanent cycle in the Atiyah-Hirzebruch spectral sequence of $U$.
\end{lem}

\begin{proof}
We consider the following two cofiber sequences:
\begin{displaymath}
    \xymatrix{
    S^{t_3}  \ar@{^{(}->}[r] & U \ar@{->>}[r]^{p_3} & T' \\
    S^{t_4}  \ar@{^{(}->}[r] & U \ar@{->>}[r]^{p_4} & T''.
    }
\end{displaymath}
Both 3 cell complexes $T'$ and $T''$ (with the following cell diagrams) satisfy the assumptions in Lemma 6.1.
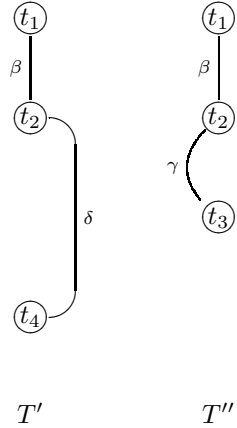
\begin{displaymath}
    \xymatrix{
    *+[o][F-]{t_1} \ar@{-}[d]_{\beta} & &  *+[o][F-]{t_1} \ar@{-}[d]_{\beta} \\
    *+[o][F-]{t_2} \ar@{-}`r[dd] `[dd]^{\delta} [dd] & & *+[o][F-]{t_2} \ar@{-}@/_1pc/[d]_{\gamma} \\
     & & *+[o][F-]{t_3}\\
    *+[o][F-]{t_4} & & \\
    T' & & T''
    }
\end{displaymath}
By Lemma 6.1, in the Atiyah-Hirzebruch spectral sequence of $T''$, we have a differential
$$d_{t_1 - t_3} (\alpha[t_1]) = \langle \alpha, \beta, \gamma\rangle[t_3] = 0.$$
The last equality follows from condition $(3)$. Using the naturality for the quotient map $p'': U \twoheadrightarrow T''$, we pull back a differential in the Atiyah-Hirzebruch spectral sequence of $U$:
$$d_{t_1 - t_3} (\alpha[t_1]) = 0.$$
By Lemma 6.1, in the Atiyah-Hirzebruch spectral sequence of $T'$, we have a differential
$$d_{t_1 - t_4} (\alpha[t_1]) \subseteq \langle \alpha, \beta, \delta\rangle[t_4].$$
Using the naturality of the quotient map $p_3: U \twoheadrightarrow T'$, we pull it back to get a differential in the Atiyah-Hirzebruch spectral sequence of $U$:
$$d_{t_1 - t_4} (\alpha[t_1]) \subseteq \langle \alpha, \beta, \delta\rangle[t_4].$$
The second and third statements follow directly from the first one, since the Toda bracket $\langle \alpha, \beta, \delta\rangle$ has zero indeterminacy under conditions $(4)$ and $(5)$, and the $E_{t_1 - t_4 + 1}$-page is the $E_\infty$-page for the Atiyah-Hirzebruch spectral sequence of $U$.
\end{proof}

Now we apply Lemma 6.2 to the complex $X^{22}$.

In $\pi_{39}$, consider the three homotopy classes $\alpha = \sigma \eta_5$, $\alpha' \in \{h_5c_0\}$ such that $2 \cdot \alpha' = 0, \ \sigma \cdot \alpha' =0$, and $\alpha'' = \sigma \{d_1\}$. Here we use the notation $\{a\}$ to denote the set of homotopy classes that are detected by $a$, where $a$ is a surviving element in the $E_\infty$-page of the Adams spectral sequence.
One can choose $\alpha' = \langle \theta_4, 2, \epsilon\rangle$. Moss's theorem tells us $\alpha' \in \{h_5c_0\}$. We have
$$2 \cdot \alpha' = 2 \langle \theta_4, 2, \epsilon\rangle = \langle 2, \theta_4, 2 \rangle \epsilon = \eta \theta_4 \epsilon = 0.$$
The last equation follows from filtration reasons. From the proof of Lemma 6.5, we also have $\sigma \cdot \alpha' =0$.
Note also that there are indeterminacies in the notation $\{d_1\}$ and $\eta_5$, but for our purpose, any choices work. The reader should compare with Isaksen's computations in \cite{Isa, Isa2}.

\begin{prop}
In the Atiyah-Hirzebruch spectral sequence of $X^{22}$, we have the following $d_8$ differentials:
\begin{equation*}
\begin{split}
d_8(\alpha[22]) & = 0, \\
d_8(\alpha'[22]) & = \eta\phi[14],\\
d_8(\alpha''[22]) & \subseteq \eta^2\pi_{44}[14],
\end{split}
\end{equation*}
where $\phi\in\pi_{45}$ is detected by $h_5d_0$, such that $\eta \cdot \phi \in \langle\alpha', 2, \nu^2\rangle$.
\end{prop}

\begin{proof}
The complex $X^{22}$ satisfies the conditions in Lemma 6.2, with $\beta = 2 \in \pi_0$, $\gamma = \eta \in \pi_1$ and $\delta = \nu^2 \in \pi_6$.
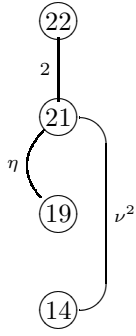
\begin{displaymath}
    \xymatrix{
    *+[o][F-]{22} \ar@{-}[d]_{2}  \\
    *+[o][F-]{21} \ar@{-}@/_1pc/[d]_{\eta} \ar@{-}`r[dd] `[dd]^{\nu^2} [dd] \\
    *+[o][F-]{19} \\
    *+[o][F-]{14}
    }
\end{displaymath}
We verify that the classes $\alpha$ and $\alpha'$ satisfy conditions $(1)$ through $(5)$, and $\alpha''$ satisfy conditions $(1)$ through $(3)$ in Lemma 6.2:
\begin{enumerate}
\item $\alpha \cdot 2 = 0$ in $\pi_{39}$. This follows from $2 \cdot \eta_5 = 0$. \\
      $\alpha' \cdot 2 = 0$ in $\pi_{39}$. This follows from our definition of $\alpha'$. \\
      $\alpha'' \cdot 2 = 0$ in $\pi_{39}$. This follows from $2 \cdot \{d_1\} = 0$.
\item $\alpha \cdot \pi_2 \subseteq \eta \cdot \pi_{40}$ in $\pi_{41}$. \\
      $\alpha' \cdot \pi_2 \subseteq \eta \cdot \pi_{40}$ in $\pi_{41}$. \\
      $\alpha'' \cdot \pi_2 \subseteq \eta \cdot \pi_{40}$ in $\pi_{41}$. \\
      These follow from the fact that $\pi_2$ is generated by $\eta^2$.
\item $0 \in \langle \alpha, 2, \eta \rangle$ in $\pi_{41}$. \\
      $0 \in \langle \alpha', 2, \eta \rangle$ in $\pi_{41}$. \\
      $0 \in \langle \alpha'', 2, \eta \rangle$ in $\pi_{41}$. \\
      These follow from the fact that the Cokernel of $J$ in $\pi_{41}$ is contained in the image of $\eta : \pi_{40} \rightarrow \pi_{41}$. In fact, suppose for example $\langle \alpha, 2, \eta \rangle$ does not contain 0. It therefore must contain an element in the image of $J$. Therefore, mapping this Toda bracket to the $K(1)$-local sphere gives a contradiction, since the class $\alpha$ maps to 0. The cases $\alpha'$ and $\alpha''$ work the same way.
\item $\alpha \cdot \pi_7 = 0$ in $\pi_{46}$. \\
      $\alpha' \cdot \pi_7 = 0$ in $\pi_{46}$. \\
      These follow from the fact that $\pi_7$ is generated by $\sigma$ and the proof of Lemma 6.5.
\item $\nu^2 \cdot \pi_{40} = 0$ in $\pi_{46}$. This follows from $\nu \cdot \pi_{43} = 0$ for filtration reasons.
\end{enumerate}
For the targets of these differentials, we apply Lemma 6.2 by computing the following Toda brackets
$$\langle \alpha, 2, \nu^2\rangle, \ \langle \alpha', 2, \nu^2\rangle, \ \langle \alpha'', 2, \nu^2\rangle.$$

For the element $\alpha = \sigma \eta_5$, we have
$$\langle \sigma \cdot \eta_5, 2, \nu^2 \rangle \supseteq \sigma \langle \eta_5, 2, \nu^2 \rangle = \eta_5 \langle 2, \nu^2, \sigma \rangle = \eta_5 \{0, \sigma^2\} = 0.$$
Note that the last equation holds because in the proof of Lemma 6.5 we have $\sigma^2 \eta_5 =0$. Therefore, by Lemma 6.2, we have the Atiyah-Hirzebruch differential 
$$d_8(\alpha[22]) = 0.$$

For the element $\alpha'\in\{h_5c_0\}$, we have
\begin{equation*}
\begin{split}
\langle \alpha', 2, \nu^2 \rangle & = \langle \alpha', 2, \langle \eta, \nu, \eta \rangle \rangle \\
& \supseteq \langle \alpha', 2, \eta, \nu \rangle \cdot \eta \\
& \subseteq \{h_5d_0\} \cdot \eta,
\end{split}
\end{equation*}
where the last inequality follows from the following Massey product in $Ext$, and Moss's theorem \cite[Theorem 1.2]{Mos}.
$$\langle h_5c_0, h_0, h_1, h_2\rangle = h_5 \langle c_0, h_0, h_1, h_2\rangle = h_5d_0.$$
That is, there exists a class $\phi$ in $\{h_5d_0\}$ in $\pi_{45}$ such that $\eta \cdot \phi \in \langle \alpha', 2, \nu^2 \rangle$. Therefore, by Lemma 6.2, we have the Atiyah-Hirzebruch differential 
$$d_8(\alpha'[22]) = \eta\phi[14].$$ 

For the element $\alpha'' = \sigma \{d_1\}$, we have
$$\langle \sigma\cdot\{d_1\}, 2, \nu^2 \rangle \supseteq \sigma \langle \{d_1\}, 2, \nu^2 \rangle \subseteq \sigma\cdot \pi_{39} \subseteq \eta^2 \pi_{44}.$$
The indeterminacy of the Toda bracket $\langle \sigma\cdot\{d_1\}, 2, \nu^2 \rangle$ is
$$\sigma\{d_1\}\cdot \pi_7 + \nu^2 \cdot \pi_{40} = \sigma\{d_1\}\cdot \pi_7 \subseteq \sigma\cdot \pi_{39} \subseteq \eta^2 \pi_{44}.$$
Therefore, we have
$$\langle \sigma\cdot\{d_1\}, 2, \nu^2 \rangle \subseteq \eta^2 \pi_{44}.$$
By Lemma 6.2, we have the Atiyah-Hirzebruch differential
$$d_8(\alpha''[22]) \subseteq \eta^2\pi_{44}[14].$$ 
\end{proof}

We also apply Lemma 6.2 to the complex $X^{20}/S^{19}$. Note that by Lemma 4.4 and Remark 5.8, we have $S^{19}$ as an $H\mathbb{F}_2$-subcomplex of $X^{20}$.

In $\pi_{41}$, we consider the homotopy class $\alpha''' = \sigma \{h_0h_2h_5\}$. Note that the notation $\{h_0h_2h_5\}$ has indeterminacy. Since $h_0h_2h_5$ does not support any hidden $\eta$-extension in the $E_\infty$-page of the Adams spectral sequence of $S^0$, we choose a class in $\{h_0h_2h_5\}$ such that its $\eta$-multiple is zero. The class $\alpha''' = \sigma \{h_0h_2h_5\}$ is therefore unique.

\begin{prop}
In the Atiyah-Hirzebruch spectral sequence of $X^{20}/S^{19}$, the element $\alpha'''[20]$ is a permanent cycle.
\end{prop}

\begin{proof}
The complex $X^{20}/S^{19}$ satisfies the conditions in Lemma 6.2, with $\beta' = \eta \in \pi_1$, $\gamma' = 2 \in \pi_0$ and $\delta' = \nu \in \pi_3$.
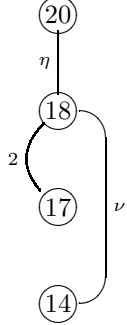
\begin{displaymath}
    \xymatrix{
    *+[o][F-]{20} \ar@{-}[d]_{\eta}  \\
    *+[o][F-]{18} \ar@{-}@/_1pc/[d]_{2} \ar@{-}`r[dd] `[dd]^{\nu} [dd] \\
    *+[o][F-]{17} \\
    *+[o][F-]{14}
    }
\end{displaymath}
We verify that $\alpha''' = \sigma \{h_0h_2h_5\} \in \pi_{41}$ satisfies the conditions $(1)$ through $(5)$ in Lemma 6.2:
\begin{enumerate}
\item $\sigma \{h_0h_2h_5\} \cdot 2 = 0$ in $\pi_{41}$. This follows from $2 \cdot \pi_{41} = 0$.
\item $\sigma \{h_0h_2h_5\} \cdot \pi_2 \subseteq 2 \cdot \pi_{43}$ in $\pi_{43}$. This follows from the fact that $\pi_2$ is generated by $\eta^2$, and that
    $$\eta^2 \cdot \pi_{41} = \{0, 4\{P^5 h_2\}\} \subseteq 2 \cdot \pi_{43}.$$
\item $0 \in \langle \sigma \{h_0h_2h_5\}, \eta, 2 \rangle$ in $\pi_{43}$. This follows from $\sigma \cdot \pi_{36} = 0$ in $\pi_{43}$. In fact, since we chose the element in $\{h_0h_2h_5\}$ such that its $\eta$-multiple is zero, we have
    $$\langle \sigma\cdot\{h_0h_2h_5\}, \eta, 2 \rangle \supseteq \sigma \langle \{h_0h_2h_5\}, \eta, 2 \rangle \subseteq \sigma\cdot \pi_{36} = 0.$$
\item $\sigma \{h_0h_2h_5\} \cdot \pi_5 = 0$ in $\pi_{46}$. This follows from $\pi_5 = 0$.
\item $\nu \cdot \pi_{43} = 0$ in $\pi_{46}$.
\end{enumerate}
We further verify that $0 \in \langle \sigma \{h_0h_2h_5\}, \eta, \nu \rangle$ in $\pi_{46}$. Since we chose the element in $\{h_0h_2h_5\}$ such that its $\eta$-multiple is zero, we have
$$\langle \sigma\cdot\{h_0h_2h_5\}, \eta, \nu \rangle \supseteq \sigma \langle \{h_0h_2h_5\}, \eta, \nu \rangle = \{h_0h_2h_5\} \cdot \langle \eta, \nu, \sigma \rangle \subseteq \{h_0h_2h_5\} \cdot \pi_{12} = 0.$$
The last equation follows from the fact that $\pi_{12}=0$. Therefore, by Lemma 6.2, the element $\alpha'''[20] = \sigma \{h_0h_2h_5\}[20]$ is a permanent cycle in the Atiyah-Hirzebruch spectral sequence of $X^{20}/S^{19}$.
\end{proof}

In the rest of this section, we prove the following relation in the stable homotopy groups of spheres, which was used in Propositions 6.3 and 6.4.

\begin{lem}
$$\sigma \cdot \pi_{39} \subseteq \eta^2 \pi_{44} = \{0, \eta^2\{g_2\}\}.$$
Moreover, there is at most one nontrivial $\sigma$-extension from $\pi_{39}$ to $\pi_{46}$, namely
$$\sigma^2\{d_1\} = \eta^2\{g_2\}.$$
\end{lem}

\begin{proof}
The group $\pi_{39}$ is generated by classes that are detected by $P^2h_0^2i$, $u$, $h_2t$, $h_3d_1$, $h_5c_0$ and $h_1h_3h_5$ in the Adams $E_\infty$-page. To prove this lemma, we check that for each element in the Adams $E_\infty$-page, $\sigma$ annihilates one class it detects, with the possible exception of $h_3d_1$. For the element $h_3d_1$, we show that there is a possible $\sigma$-extension from $h_3d_1$ to $N$, and it is equivalent to an $\eta$-extension from $h_1g_2$ to $N$. It is now known that
this nontrivial $\sigma$-extension does in fact exist, but it is irrelevant to the proofs in this paper.

\begin{enumerate}
\item For $P^2h_0^2i$, we have $\sigma \cdot \{P^2h_0^2i\} = 0$ for filtration reasons.
\item For $u$, suppose $\sigma \cdot \{u\} \neq 0$. The only possibility is $\sigma \cdot \{u\} = \{d_0 l\}$ for filtration reasons. However, this cannot happen, since both $\{u\}$ and $\{d_0 l\}$ are detected by $tmf$, and $\sigma = 0$ in $\pi_\ast tmf$: mapping this relation to $\pi_\ast tmf$ gives a contradiction. Therefore, $\sigma \cdot \{u\} = 0$.
\item For $h_2t$, one class that it detects is $\nu \{t\}$. It follows from $\nu \cdot \sigma = 0$ that $\sigma \cdot \{h_2t\} = 0$.
\item For $h_3d_1$, note that there is a relation in $Ext$: $h_3d_1 = h_1e_1$. Following Bruner's differential \cite[Theorem 4.1]{Br1}
    $$d_3(e_1) = h_1t = h_2^2n,$$
we have a Massey product in the Adams $E_4$-page
$$\langle h_2n, h_2, h_1\rangle = h_1e_1.$$
By Moss's theorem \cite[Theorem 1.2]{Mos}, we have that the Toda bracket $\langle \nu\{n\}, \nu, \eta\rangle$ is detected by $h_1e_1 = h_3d_1$. Therefore,
$$\sigma \cdot \langle \nu\{n\}, \nu, \eta\rangle = \langle \sigma, \nu\{n\}, \nu \rangle \cdot \eta.$$
By Bruner's differential and Moss's theorem, we have that the Toda bracket $\langle \sigma, \nu\{n\}, \nu \rangle$ is detected by
$$h_1g_2 = h_3e_1 = \langle h_3, h_2n, h_2\rangle.$$
Since the only element with higher filtration than $h_1g_2$ that supports an $\eta$-extension is $w$, to show that
$$\sigma \cdot \{h_3d_1\} = \eta^2 \{g_2\},$$
we only need to show that
$$\sigma \cdot \langle \nu\{n\}, \nu, \eta\rangle \neq \{w\}\cdot \eta.$$
Suppose the opposite is true. Multiplying the equation by $\eta$ gives a contradiction, since $h_3d_1$ does not support hidden $\eta$-extension while $d_0 l$ does. Therefore, we have
$$\sigma \cdot \{h_3 d_1\} = \eta^2 \{g_2\}.$$
\item For $h_5c_0$, by Moss's theorem, $\alpha' = \langle \theta_4, 2, \epsilon\rangle$ is detected by $h_5c_0$. We have
$$\langle \theta_4, 2, \epsilon\rangle \cdot \sigma = \theta_4 \cdot \langle 2, \epsilon, \sigma\rangle = \theta_4 \cdot 0 =0.$$
Therefore, we have the class $\alpha' = \langle \theta_4, 2, \epsilon\rangle$ in $\{h_5 c_0\}$ such that $\sigma \cdot \alpha' = 0$.
\item For $h_1h_3h_5$, it detects $\alpha = \sigma \eta_5$. Since $\nu\cdot \eta_5 =0$, we have
$$\sigma \cdot \sigma \eta_5 = \langle\nu, \sigma, \nu\rangle \eta_5 = \nu \langle \sigma, \nu, \eta_5\rangle \subseteq \nu \cdot \pi_{43} = 0.$$
Therefore, we have the class $\alpha = \sigma \eta_5$ in $\{h_1h_3h_5\}$ such that $\sigma \cdot \alpha = 0$.
\end{enumerate}
In sum, we have $\sigma \cdot \pi_{39} \subseteq \eta^2 \pi_{44} = \{0, \eta^2\{g_2\}\}$.
\end{proof}

\section{The cofiber of $\eta$}

In this section, we establish Step 1 by proving the following theorem.

\begin{thm}
In the Adams spectral sequence of $\Sigma^{14} C \eta$, we have a $d_4$ differential in the 61-stem:
$$d_4(h_4^3[16]) = B_1[14].$$
\end{thm}

\begin{proof}
The cofiber sequence
\begin{displaymath}
    \xymatrix{
  S^{15} \ar[r]^{\eta} & S^{14} \ar[r]^i & \Sigma^{14}C\eta \ar[r]^p & S^{16}
    }
\end{displaymath}
gives us a short exact sequence on cohomology
\begin{displaymath}
    \xymatrix{
  0 \ar[r] & H^\ast(S^{16}) \ar[r]^{p^\ast} & H^\ast(\Sigma^{14}C\eta) \ar[r]^{i^\ast} & H^\ast(S^{14}) \ar[r] & 0
    }
\end{displaymath}
and therefore a long exact sequence of $Ext$ groups
\begin{displaymath}
    \xymatrix{
  Ext^{s-1,t-1}(S^{15}) \ar[r]^-{h_1} & Ext^{s,t}(S^{14}) \ar[r]^{i_{\sharp}} & Ext^{s,t}(\Sigma^{14}C\eta) \ar[r]^{p_\sharp} & Ext^{s,t}(S^{16}).
    }
\end{displaymath}
From this long exact sequence, we have in Table 1 the Adams $E_2$ page of $\Sigma^{14}C\eta$ in the 60 and 61 stems for $s\leq 7$.

\begin{table}[h]
\caption{The Adams $E_2$ page of $\Sigma^{14}C\eta$ in the 60 and 61 stems for $s\leq 7$}
\centering
\begin{tabular}{ l l l }
$s\backslash t-s$ & 60 & 61 \\ [0.5ex] 
\hline 
7 & $B_1[14]$ & $h_0^2h_5d_0[16]$\\ \hline
6 & $h_0^2g_2[16]$ & $h_0h_2g_2[14]$\\
  & & $h_0h_5d_0[16]$ \\ \hline
5 & $h_0g_2[16]$ & $h_2g_2[14]$\\
  & & $h_1g_2[16]$ \\ \hline
4 & & $h_0h_4^3[16]$\\ \hline
3 & & $h_4^3[16]$\\
\end{tabular}
\label{Ceta}
\end{table}

Firstly, since there is an $\eta$-extension from $h_4^3$ to $B_1$ in $S^0$, The class $B_1[14]$ in $Ext(\Sigma^{14}C\eta)$ detects zero in $\pi_{60}(\Sigma^{14}C\eta)$, and therefore must be killed by some element. There are four candidates: $h_4^3[16]$ in filtration 3, $h_0h_4^3[16]$ in filtration 4, and $h_2g_2[14], \ h_1g_2[16]$ in filtration 5.

Secondly, the element $h_4^3[16]$ in $Ext(\Sigma^{14}C\eta)$ cannot survive. Suppose it did. We would then have $q_\sharp(h_4^3[16]) = h_4^3[16]$, where the image survives in $Ext(S^{16})$. However, the homotopy class detected by $h_4^3[16]$ in $Ext(S^{16})$ maps nontrivially to a class in $\pi_{60}(\Sigma S^{14})$ because of the same $\eta$-extension. This contradicts the exactness of the long exact sequence of homotopy groups.

Thirdly, the element $h_2g_2[14]$ is a permanent cycle and therefore cannot kill $B_1[14]$. In fact, the element $h_2g_2[14]$ is a permanent cycle in $Ext(S^{14})$. The image $i_\sharp(h_2g_2[14]) = h_2g_2[14]$ must also be a permanent cycle.

At last, the kernel of the map
$$\eta: \pi_{45}\longrightarrow \pi_{46}$$
is $\mathbb{Z}/{8}\oplus\mathbb{Z}/{2}$, generated by an order 8 element detected by $h_0h_4^3$ and $\eta\{g_2\}$. Since $h_0h_4^3$ and $h_1g_2$ have filtration 4 and 5, we must have two more surviving cycles in $\pi_{61}(\Sigma^{14}C\eta)$ with filtration strictly smaller than 6 besides $h_2g_2[14]$. The only possibility is $h_0h_4^3[16]$ and $h_1g_2[16]$, since we know $h_4^3[16]$ cannot survive.

Therefore, the only possibility to kill $B_1[14]$ is $h_4^3[16]$.
\end{proof}

\begin{cor}
The elements $h_0h_4^3[16]$, $h_2g_2[14]$ and $h_1g_2[16]$ survive in the Adams spectral sequence of $\Sigma^{14} C\eta$.
\end{cor}

\begin{proof}
This follows directly from the proof of Theorem 7.1 and filtration reasons.
\end{proof}

\section{The Adams spectral sequence of $\widetilde{X}$}

In this section, based on Theorem 7.1, we prove the following Theorem 8.1 in Step 2.

\begin{thm}
In the Adams spectral sequence of $\widetilde{X}$, we have the differential
$$d_4(h_4^3[16]) = B_1[14].$$
\end{thm}

The proof of Theorem 8.1 is summarized as in the following Table 2.

\begin{table}[h]
\caption{The Adams $E_2$ page of $\widetilde{X}$ in the 60 and 61 stems for $s\leq 7$}
\centering
\begin{tabular}{ l l l | l l l }
$s\backslash t-s$ & 60 & 61 & status & proof & $H\mathbb{F}_2$-subquotients used \\ [0.5ex] 
\hline 
7 & $B_1[14]$ & $\bullet$ & & &\\
  & $\bullet$ & $\bullet$ & & & \\ \hline
6 & $h_0^2f_1[20]$ & $\bullet$ & & &\\
  & $\bullet$ & $\bullet$ & & & \\
  & $\bullet$ & $\bullet$ & & & \\ \hline
5 & $\bullet$ & $h_2g_2[14]$ & permanent cycle & Lemma 8.3 & $\Sigma^{14} C\eta$\\
  & $\bullet$ & $h_1g_2[16]$ & permanent cycle & Lemma 8.3 & $\Sigma^{14} C\eta$\\
  & & $h_1f_1[20]$ & permanent cycle & Lemma 8.10 & $X^{20}$\\
  & & $h_1h_5c_0[21]$ & permanent cycle & Lemma 8.7 & $X^{21}$ \\
  & & $h_3d_1[22]$ & permanent cycle & Lemma 8.8 & $X^{22}$ and $\widehat{X^{22}}$\\ \hline
4 & $\bullet$ & $h_0h_4^3[16]$\\
  & & $g_2[17]$ & permanent cycle & Lemma 8.4 & $S^{17}$\\
  & & $f_1[21]$ & $d_2(f_1[21]) = h_0^2f_1[20]$ & Lemma 8.5 &  $P_{19}^{21}$\\
  & & $h_1^2h_3h_5[21]$ & permanent cycle & Lemma 8.7 & $X^{21}$\\
  & & $h_5c_0[22]$ & permanent cycle & Lemma 8.8 & $X^{22}$ and $\widehat{X^{22}}$\\ \hline
3 & $\bullet$ & $h_4^3[16]$ & $d_4(h_4^3[16]) = B_1[14]$ & & \\
  & & $\underline{h_1h_3h_5[22]}$ & permanent cycle & Lemma 8.8 & $X^{22}$ and $\widehat{X^{22}}$ \\
\end{tabular}
\label{X}
\end{table}
Here the element $\underline{h_1h_3h_5[22]}$ is defined to be the image of $h_1h_3h_5[22]$ in $Ext(X^{22})$. In fact, the group $Ext^{3, 64}(X^{22}) = \mathbb{Z}/2$ is generated by $h_1h_3h_5[22]$ as we will show in Lemma 8.8. Each $\bullet$ represents a nontrivial element in its bidegree. But these elements are irrelevant to our purpose.

\begin{proof}
Firstly, as we will show in Lemma 8.2, the Adams $E_2$-page of $\widetilde{X}$ in the 60 and 61 stems for $s\leq 7$ is as claimed in Table 2. In particular, there are 10 elements in Adams filtration 4 and 5. Secondly, by the Lemmas 8.3, 8.4, 8.5, 8.7, 8.8 and 8.10 in later part of this section, the element $B_1[14]$ in Adams filtration 7 cannot be killed by any $d_2$ or $d_3$ differentials from these 10 elements. In fact, one of these 10 elements in Adams filtration 4 supports a $d_2$ differential, and the rest are permanent cycles. Therefore, the element $B_1[14]$ survives to the $E_4$-page of the Adams spectral sequence of $\widetilde{X}$. Theorem 8.1 follows from naturality of the Adams spectral sequences and Theorem 7.1.
\end{proof}

\begin{lem}
The Adams $E_2$ page of $\widetilde{X}$ in the 60 and 61 stem for $s\leq 7$ is as claimed in Table 2.
\end{lem}

\begin{proof}
Because of the cell structure of $\widetilde{X}$, there exists a cofiber sequence
\begin{displaymath}
    \xymatrix{
  S^{14} \ar[r]^{i} & \widetilde{X} \ar[r]^q &  P_{16}^{22}  \ar[r]^a & \Sigma S^{14}
    }
\end{displaymath}
This cofiber sequence gives us a short exact sequence on cohomology
\begin{displaymath}
    \xymatrix{
  0 \ar[r] & H^\ast(P_{16}^{22}) \ar[r]^{q^\ast} & H^\ast(\widetilde{X}) \ar[r]^{i^\ast} & H^\ast(S^{14}) \ar[r] & 0
    }
\end{displaymath}
and therefore a long exact sequence on $Ext$ groups
\begin{displaymath}
    \xymatrix{
  Ext^{s,t}(S^{14}) \ar[r]^{i_\sharp} & Ext^{s,t}(\widetilde{X}) \ar[r]^{q_{\sharp}} & Ext^{s,t}(P_{16}^{22}) \ar[r]^-\delta & Ext^{s+1,t+1}(\Sigma S^{14}).
    }
\end{displaymath}

Note that the Adams filtration of the attaching map $a: P_{16}^{22} \rightarrow \Sigma S^{14}$ is 1. In fact, in its cofiber $\widetilde{X}$, the 16-cell is attached to the 14-cell by $\eta$, which has the Adams filtration 1. Therefore, the boundary map in the long exact sequence on $Ext$ groups raises the Adams filtration by 1.

In Section 6 of \cite{WX}, we explained how to obtain the Adams $E_2$-page of $P_n^{n+k}$ from our Curtis table of $P_1^\infty$.
In particular, we have the Adams $E_2$ page of $P_{16}^{22}$ in the 60 and 61 stem for $s\leq 7$.

To compute $Ext(\widetilde{X})$ from the long exact sequence on $Ext$ groups, we also need to compute the boundary homomorphism $\delta: Ext^{s,t}(P_{16}^{22}) \rightarrow Ext^{s+1,t+1}(\Sigma S^{14})$. In fact, in the 61 stem for $s\leq 5$, there is only one element $h_5d_0[16]$ (with the right choices of other elements) which maps nontrivially: $\delta(h_5d_0[16]) = h_1h_5d_0[14]$. This follows from the naturality of the boundary homomorphism induced by the inclusion map $\Sigma^{14} C\eta \rightarrow \widetilde{X}$, and the fact that
\begin{equation*}
\begin{split}
Ext^{s,s+46}(S^0) & = 0 \text{~~for~~}s\leq 5 \\
Ext^{6,6+46}(S^0) & = \mathbb{Z}/2, \text{generated by~~} h_1h_5d_0
\end{split}
\end{equation*}

\begin{displaymath}
    \xymatrix{
    Ext^{s,t}(S^{14}) \ar[r]^{i_{\sharp}} \ar[d] & Ext^{s,t}(\Sigma^{14}C\eta) \ar[r]^{q_\sharp} \ar[d] & Ext^{s,t}(S^{16}) \ar[r] \ar[d] & Ext^{s+1,t+1}(\Sigma S^{14}) \ar[d] \\
  Ext^{s,t}(S^{14}) \ar[r]^{i_\sharp} & Ext^{s,t}(\widetilde{X}) \ar[r]^{q_{\sharp}} & Ext^{s,t}(P_{16}^{22}) \ar[r]^-\delta & Ext^{s+1,t+1}(\Sigma S^{14})
    }
\end{displaymath}

Note that the boundary homomorphism $\delta$ corresponds to differentials in the algebraic Atiyah-Hirzebruch spectral sequence of $\widetilde{X}$. One can check, using the naturality of the algebraic Atiyah-Hirzebruch spectral sequence for the quotient map $P_{14}^{22} \twoheadrightarrow \widetilde{X}$, the other elements (with the right choices) maps to zero under the boundary homomorphism $\delta$.

This completes the proof.
\end{proof}

The following lemma is a consequence of Corollary 7.2 and naturality of the Adams spectral sequence.

\begin{lem}
In the Adams spectral sequence of $\widetilde{X}$, the elements $h_2g_2[14]$, $h_1g_2[16]$ and $h_0h_4^3[16]$ are permanent cycles.
\end{lem}

\begin{proof}
By Corollary 7.2, the elements $h_2g_2[14]$, $h_1g_2[16]$ and $h_0h_4^3[16]$ are surviving cycles in the Adams spectral sequence of $\Sigma^{14} C\eta$. In particular, they are permanent cycles. Since $\Sigma^{14} C\eta$ is the 16-skeleton of $\widetilde{X}$, by naturality for the map
\begin{displaymath}
    \xymatrix{
    \Sigma^{14} C\eta \ar@{^{(}->}[r] & \widetilde{X},
    }
\end{displaymath}
these elements are also permanent cycles in the Adams spectral sequence of $\widetilde{X}$.
\end{proof}

\begin{lem}
In the Adams spectral sequence of $\widetilde{X}$, the element $g_2[17]$ is a permanent cycle.
\end{lem}

\begin{proof}
By Lemma 5.3, $S^{17}$ is an $H\mathbb{F}_2$-subcomplex of $\widetilde{X}$. Since $g_2$ is a permanent cycle in the Adams spectral sequence of $S^0$, by the naturality for the inclusion map, it is also a permanent cycle in the Adams spectral sequence of $\widetilde{X}$.
\end{proof}

\begin{lem}
In the Adams spectral sequence of $\widetilde{X}$, we have a $d_2$ differential
$$d_2(f_1[21]) = h_0^2f_1[20].$$
\end{lem}

To prove Lemma 8.5, we need to prove the following lemma.

\begin{lem}
We have a quotient map $q: P_{19}^{21} \twoheadrightarrow S^{20}$. Moreover, we have $q_\sharp (f_1[21]) = h_0c_2[20]$, where $q_\sharp: Ext(P_{19}^{21}) \rightarrow Ext(S^{20})$ is the induced map on the Adams $E_2$-page.
\end{lem}

\begin{proof}
By James periodicity, the quotient map $q: P_{19}^{21} \twoheadrightarrow S^{20}$ maps through $P_{20}^{21}$.
\begin{displaymath}
    \xymatrix{
    P_{19}^{21} \ar@{->>}[r]^{q_1} & P_{20}^{21} \ar@{->>}[r]^{q_2} & S^{20}
}
\end{displaymath}
The cell diagram of $P_{19}^{21}$ is the following:
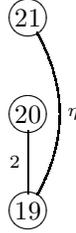
\begin{displaymath}
    \xymatrix{
    *+[o][F-]{21} \ar@{-}@/^1pc/[dd]^{\eta} \\
    *+[o][F-]{20} \ar@{-}[d]_{2}\\
    *+[o][F-]{19} }
\end{displaymath}
In $Ext(P_{20}^{21})$, we define the element $\underline{f_1[21]}$ to be the image of $f_1[21]$ in $Ext(S^{21})$ under the inclusion map $i: S^{21} \hookrightarrow P_{20}^{21}$, i.e., $\underline{f_1[21]} = i_\sharp(f_1[21])$.
\begin{displaymath}
    \xymatrix{
    & Ext(S^{21}) \ar[d]^{i_\sharp} & \\
 Ext(P_{19}^{21}) \ar[r]^{q_{1\sharp}} & Ext(P_{20}^{21}) \ar[r]^{q_{2\sharp}} & Ext(S^{20}) \\
 f_1[21] \ar@{|->}[r]  & *\txt{ $\underline{f_1[21]}$\\$+h_0c_2[20]$} \ar@{|->}[r] & h_0c_2[20]
    }
\end{displaymath}
By naturality of the algebraic Atiyah-Hirzebruch spectral sequence, we have $q_{2\sharp}(\underline{f_1[21]}) = 0$. Therefore, in $Ext(P_{20}^{21})$, the element $\underline{f_1[21]} + h_0c_2[20]$ maps to $h_0c_2[20]$ in $Ext(S^{20})$, i.e.,
$$q_{2\sharp}(\underline{f_1[21]} + h_0c_2[20]) = h_0c_2[20].$$

Now we consider the cofiber sequence associated to the map $q_1$.
\begin{displaymath}
    \xymatrix{
  S^{19} \ar@{^{(}->}[r] &  P_{19}^{21} \ar@{->>}[r]^-{q_1} & P_{20}^{21} = S^{21}\vee S^{20} \ar[r] & \Sigma S^{19}
}
\end{displaymath}
Both elements $\underline{f_1[21]}$ and $h_0c_2[20]$ map to $h_1f_1[19]$ in $Ext(\Sigma S^{19})$. In fact, it follows from the fact that the 21-cell is attached to the 19-cell by $\eta$, and the 20-cell is attached to the 19-cell by $2$. Note also that there is a relation $h_0^2c_2 = h_1f_1$ in $Ext$. Therefore, the sum $\underline{f_1[21]} + h_0c_2[20]$ maps to 0 in $Ext(\Sigma S^{19})$, and must comes from $Ext(P_{19}^{21})$ by exactness. By naturality of the algebraic Atiyah-Hirzebruch spectral sequence, it must come from $f_1[21]$, i.e.,
$$q_{1\sharp}(f_1[21]) = \underline{f_1[21]} + h_0c_2[20].$$
Combining with
$$q_{2\sharp}(\underline{f_1[21]} + h_0c_2[20]) = h_0c_2[20],$$
we have
$$q_\sharp(f_1[21]) = h_0c_2[20].$$
\end{proof}

Now we present the proof of Lemma 8.5.

\begin{proof}
In the Adams spectral sequence of $S^0$, we have a differential
$$d_2(h_0c_2) = h_0^2f_1.$$
Now consider the following commutative diagram:
\begin{displaymath}
    \xymatrix{
  \widetilde{X} \ar@{->>}[r]^{q_3} &  P_{19}^{22} \ar@{->>}[r]^{q_4} & S^{20} \ar@{=}[d] \\
  & P_{19}^{21} \ar@{^{(}->}[u]^i \ar@{->>}[r]^{q} & S^{20}
}
\end{displaymath}
where $q_3$ is obtained from $\widetilde{X}$ by quotienting out its 18-skeleton, $q_4$ is a quotient map that follows essentially from Theorem 4.7 and James periodicity, and $i$ is an inclusion map. By Lemma 8.6, the $d_2$ differential in $S^{20}$:
$$d_2(h_0c_2[20]) = h_0^2f_1[20]$$
can be pulled back to get a $d_2$ differential in $P_{19}^{21}$:
$$d_2(f_1[21]) = h_0^2f_1[20].$$
This differential can be further pushed forward by $i$, and then pulled back by $q_3$ to get the $d_2$ differential in $\widetilde{X}$:
$$d_2(f_1[21]) = h_0^2f_1[20].$$
Note that in $Ext(\widetilde{X})$, elements of lower Atiyah-Hirzebruch filtrations, i.e., $h_2g_2[14]$ and $h_1g_2[16]$, have already been shown to survive by Lemma 8.3.
\end{proof}

\begin{lem}
The elements $h_1h_5c_0[21]$ and $h_1^2h_3h_5[21]$ are permanent cycles in $Ext(\widetilde{X})$.
\end{lem}

\begin{proof}
We consider the $H\mathbb{F}_2$-subcomplex $X^{21}$. Since there are only three cells in $X^{21}$, the computation of the Adams $E_2$ page of $X^{21}$ in the 61 stem for $s\leq 5$ is straightforward by using the algebraic Atiyah-Hirzebruch spectral sequence.

\begin{table}[h]
\caption{The Adams $E_2$ page of $X^{21}$ and $S^{21}$ in the 61 stem for $s\leq 5$}
\centering
\begin{tabular}{ l l l }
$s\backslash 61-\text{stem of}$ & $X^{21}$ & $S^{21}$ \\ [0.5ex] 
\hline 
5 & $h_1h_5c_0[21]$  & $h_1h_5c_0[21]$\\
  & $h_2g_2[14]$ & $\bullet$ \\ \hline
4 & $h_1^2h_3h_5[21]$ & $h_1^2h_3h_5[21]$\\
  & & $f_1[21]$  \\
\end{tabular}
\label{X21}
\end{table}

By Theorem 5.14, the $H\mathbb{F}_2$-subcomplex $X^{21}$ fits into a cofiber sequence

\begin{displaymath}
    \xymatrix{
    X^{21} \ar@{->>}[rr]^{q_{21}} & & S^{21} \ar[rr]^-{(\eta,\nu^2)} & & S^{20}\vee S^{15} \\
    *+[o][F-]{21} \ar@{-}@/_1pc/[d]_{\eta} \ar@{-}`r[dd] `[dd]^{\nu^2} [dd] \ar@{->>}[rr] & & \xybox{(0,0.5)*+[o][F-]{21}} \ar[rrd]^{\eta} \ar[rrdd]_{\nu^2} & & \\
    *+[o][F-]{19} & & & & \xybox{(0,0.5)*+[o][F-]{20}}\\
    *+[o][F-]{14} & & & & \xybox{(0,0.5)*+[o][F-]{15}}
    }
\end{displaymath}
Here $q_{21}$ is the quotient map. We therefore have a long exact sequence of homotopy groups. Suppose $\alpha \in \pi_{61}(S^{21})$, and $\alpha$ lies in the kernel of the map
$$(\eta,\nu^2) : \pi_{61}(S^{21}) \longrightarrow \pi_{61}(S^{20})\oplus \pi_{61}(S^{15}).$$
Then $\alpha$ must satisfy the following conditions:
$$\eta \cdot \alpha = 0, \ \nu^2 \cdot \alpha =0.$$
We verify that the elements $h_1h_5c_0[21]$ and $h_1^2h_3h_5[21]$ each detect a class that satisfies the above condition. In fact, we have that
$$0 \in \eta \cdot \{h_1^2h_3h_5\}, \ 0 \in \eta \cdot \{h_1h_5c_0\},\text{~~and~~} \nu \cdot \pi_{40}(S^0)=0.$$
Therefore, by exactness of homotopy groups, in $\pi_{61}(X^{21})$, there exist classes that map nontrivially to $\pi_{61}(S^{21})$. Furthermore, these classes are in Adams filtration at most 5. By naturality of the algebraic Atiyah-Hirzebruch spectral sequence, the classes detected by $h_2g_2[14]$ map trivially to $\pi_{61}(S^{21})$. It follows that $h_1h_5c_0[21]$ and $h_1^2h_3h_5[21]$ survive in the Adams spectral sequence of $X^{21}$. In particular, they are permanent cycles. Since $X^{21}$ is an $H\mathbb{F}_2$-subcomplex of $\widetilde{X}$, both elements are permanent cycles in the Adams spectral sequence of $\widetilde{X}$.
\end{proof}

\begin{lem}
The elements $h_3d_1[22]$, $h_5c_0[22]$ and $\underline{h_1h_3h_5[22]}$ are permanent cycles in the Adams spectral sequence of $\widetilde{X}$.
\end{lem}

\begin{proof}
For the element $\underline{h_1h_3h_5[22]}$, we consider the $H\mathbb{F}_2$-subcomplex $X^{22}$, since it is defined by the image of $h_1h_3h_5[22]$ in $Ext(X^{22})$. For the elements $h_3d_1[22]$ and $h_5c_0[22]$, we use both of the $H\mathbb{F}_2$-subcomplexes $X^{22}$ and $\widehat{X^{22}}$. The reason we use different $H\mathbb{F}_2$-subcomplexes here is explained in Remark 8.9.

Using the algebraic Atiyah-Hirzebruch spectral sequences, and their naturality for the maps
\begin{displaymath}
    \xymatrix{
    X^{22} \ar@{^{(}->}[r] & \widehat{X^{22}} \ar@{^{(}->}[r] & \widetilde{X},
    }
\end{displaymath}
we compute the Adams $E_2$-page of $X^{22}$ and $\widehat{X^{22}}$ in the 61 stem for $s\leq 5$.
\begin{table}[h]
\caption{The Adams $E_2$ page of $X^{22}, \widehat{X^{22}}$ and $S^{22}$ in the 61 stem for $s\leq 5$}
\centering
\begin{tabular}{ l l l l }
$s\backslash 61-\text{stem of}$ & $X^{22}$ & $\widehat{X^{22}}$ & $S^{22}$ \\ [0.5ex] 
\hline 
5 & $h_3d_1[22]$ & $h_3d_1[22]$ & $h_3d_1[22]$\\
  & $h_1h_5c_0[21]$ & $h_1h_5c_0[21]$ &  \\
  & $h_2g_2[14]$ & $h_2g_2[14]$ & \\
  & & $h_1g_2[16]$ & \\ \hline
4 & $h_5c_0[22]$ & $h_5c_0[22]$ & $h_5c_0[22]$\\
  & $h_1^2h_3h_5[21]$ & $h_1^2h_3h_5[21]$ &  \\
  & & $h_0h_4^3[16]$ & \\\hline
3 & $h_1h_3h_5[22]$ & $h_1h_3h_5[22]$ & $h_1h_3h_5[22]$ \\
  & & $h_4^3[16]$ & \\
\end{tabular}
\label{X22}
\end{table}

By Definition 5.6, the complex $X^{22}$ fits into a cofiber sequence
\begin{displaymath}
    \xymatrix{
    X^{22} \ar@{->>}[rr]^{q} & & S^{22} \ar[rr]^{a} & & \Sigma X^{21} \\
    *+[o][F-]{22} \ar@{-}[d]^{2} \ar@{->>}[rr] & &  \xybox{(0,0.5)*+[o][F-]{22}} \ar[rr]^2 & & *+[o][F-]{22} \ar@{-}@/_1pc/[d]_{\eta} \ar@{-}`r[dd] `[dd]^{\nu^2} [dd] \\
    *+[o][F-]{21} \ar@{-}@/_1pc/[d]_{\eta} \ar@{-}`r[dd] `[dd]^{\nu^2} [dd] & & & & *+[o][F-]{20} \\
    *+[o][F-]{19} & & & & *+[o][F-]{15} \\
    *+[o][F-]{14}
       }
\end{displaymath}
Here $q$ is the quotient map, and $a$ is the suspension of the attaching map of the 22-cell in $X^{22}$. We have a long exact sequence of homotopy groups associated to this cofiber sequence. Suppose $\alpha[22]$ is an element in $\pi_{61}(S^{22})$. Suppose further that $\alpha[22]$ supports a differential in the Atiyah-Hirzebruch spectral sequence of $X^{22}$. By the construction of the Atiyah-Hirzebruch spectral sequence, the target of the differential that $\alpha[22]$ supports detects $\Delta(\alpha[22])$ in the homotopy groups of lower skeleton, where the map
$$\Delta : \pi_{61}(S^{22}) \longrightarrow \pi_{61}(\Sigma X^{21})$$
is the boundary homomorphism in the long exact sequence of homotopy groups.

For the element $\underline{h_1h_3h_5[22]}$, we consider the homotopy class $\alpha = \sigma \eta_5 \in \pi_{39}$, which is detected by $h_1h_3h_5$ in the $E_\infty$-page of the Adams spectral sequence of $S^0$. By Proposition 6.3, the element $\alpha[22]$ is a permanent cycle in the Atiyah-Hirzebruch spectral sequence of $X^{22}$. Therefore, by exactness of the long exact sequence of homotopy groups, there exists a homotopy class in $\pi_{61}(X^{22})$, which has Adams filtration at most 3. This implies the element $h_1h_3h_5[22]$ survives in $Ext(X^{22})$, since it is the only element with Adams filtration at most 3. In particular, it is a permanent cycle. Therefore, its image in $Ext(\widetilde{X})$, i.e., $\underline{h_1h_3h_5[22]}$, is also a permanent cycle.

By Definition 5.5, the complex $\widehat{X^{22}}$ fits into a cofiber sequence
\begin{displaymath}
    \xymatrix{
    \widehat{X^{22}} \ar@{->>}[rr]^{q'} & & S^{22} \ar[rr]^{a'} & & \Sigma \widehat{X^{21}} \\
    *+[o][F-]{22} \ar@{-}[d]^{2} \ar@{->>}[rr] & & \xybox{(0,0.5)*+[o][F-]{22}} \ar[rr]^2 & & *+[o][F-]{22} \ar@{-}@/_1pc/[d]_{\eta} \ar@{-}`r[ddd] `[ddd]^{\nu^2} [ddd] \\
    *+[o][F-]{21} \ar@{-}@/_1pc/[d]_{\eta} \ar@{-}`r[ddd] `[ddd]^{\nu^2} [ddd] & & & & *+[o][F-]{20} \\
    *+[o][F-]{19} & & & & *+[o][F-]{17} \ar@{-}@/_1pc/[d]_{\eta} \\
    *+[o][F-]{16} \ar@{-}@/_1pc/[d]_{\eta} & & & & *+[o][F-]{15} \\
    *+[o][F-]{14}
    }
\end{displaymath}
Here $q'$ is the quotient map, and $a'$ is the suspension of the attaching map of the 22-cell in $\widehat{X^{22}}$. We have a long exact sequence of homotopy groups associated to this cofiber sequence. Suppose $\alpha'[22]$ is an element in $\pi_{61}(S^{22})$.

Suppose further that $\alpha'[22]$ supports a differential in the Atiyah-Hirzebruch spectral sequence of $\widehat{X^{22}}$. By the construction of the Atiyah-Hirzebruch spectral sequence, the target of the differential that $\alpha'[22]$ supports detects $\Delta'(\alpha'[22])$ in the homotopy groups of lower skeleton, where the map
$$\Delta' : \pi_{61}(S^{22}) \longrightarrow \pi_{61}(\Sigma \widehat{X^{21}}).$$
is the boundary homomorphism in the long exact sequence of homotopy groups.

For the element $h_5c_0[22]$, we consider a homotopy class $\alpha'$ in $\{h_5c_0\} \in \pi_{39}$, such that $2 \cdot \alpha' = 0$. Such a class exists, since there is no 2-extension from $h_5c_0$ in the $E_\infty$-page of the Adams spectral sequence of $S^0$. By Proposition 6.3, we have a differential in the Atiyah-Hirzebruch spectral sequence of $X^{22}$:
$$d_8(\alpha'[22]) = \eta\phi[14],$$
where $\phi\in\pi_{45}$ is detected by $h_5d_0$, such that $\eta \cdot \phi \in \langle\alpha', 2, \nu^2\rangle$.

We map this differential to the Atiyah-Hirzebruch spectral sequence of $\widehat{X^{22}}$. Since the 16-skeleton of $\widehat{X^{22}}$ is $\Sigma^{14} C \eta$, we have a differential in the Atiyah-Hirzebruch spectral sequence of $\widehat{X^{22}}$:
$$d_2(\phi[16]) = \eta \phi[14].$$
This implies the following differential
$$d_8(\alpha'[22]) = 0.$$
That is, $\alpha'[22]$ is a permanent cycle in the Atiyah-Hirzebruch spectral sequence of $\widehat{X^{22}}$. Therefore, by exactness of the long exact sequence of homotopy groups, there exists a homotopy class in $\pi_{61}(\widehat{X^{22}})$, which has Adams filtration at most 4. By naturality of the Adams spectral sequence for the quotient map $\widehat{X^{22}} \twoheadrightarrow S^{22}$, the class that detects $\alpha'[22]$ in $Ext(\widehat{X^{22}})$ must map nontrivially to $Ext(S^{22})$.
\begin{displaymath}
    \xymatrix{
  Ext(\widehat{X^{22}}) \ar[r] \ar@{=>}[d] & Ext(S^{22}) \ar@{=>}[d] \\
  \pi_\ast(\widehat{X^{22}}) \ar[r] & \pi_\ast (S^{22})
    }
\end{displaymath}
Since the element $\underline{h_1h_3h_5[22]}$ is already accounted for, by filtration arguments, the only possibility is that $h_5c_0[22]$ detects $\alpha'[22]$. In particular, the element $h_5c_0[22]$ is a permanent cycle in the Adams spectral sequence of $\widehat{X^{22}}$. Therefore, its image in $Ext(\widetilde{X})$ is also a permanent cycle.

For the element $h_3d_1[22]$, we consider the homotopy class $\alpha'' = \sigma \{d_1\} \in \pi_{39}$, which is detected by $h_3d_1$ in the $E_\infty$-page of the Adams spectral sequence of $S^0$. Note that the notation $\{d_1\}$ has indeterminacy, but for our purpose, any class in the set $\{d_1\}$ works.
By Proposition 6.3, we have a differential in the Atiyah-Hirzebruch spectral sequence of $X^{22}$:
$$d_8(\alpha''[22]) \subseteq \eta^2\pi_{44}[14].$$
We map this differential to the Atiyah-Hirzebruch spectral sequence of $\widehat{X^{22}}$. Since the 16-skeleton of $\widehat{X^{22}}$ is $\Sigma^{14} C \eta$, we have some $d_2$ differentials in the Atiyah-Hirzebruch spectral sequence of $\widehat{X^{22}}$ that kill $\eta^2\pi_{44}[14]$. This implies the following differential
$$d_8(\alpha''[22]) = 0.$$
That is, $\alpha''[22]$ is a permanent cycle in the Atiyah-Hirzebruch spectral sequence of $\widehat{X^{22}}$. Therefore, by exactness of the long exact sequence of homotopy groups, there exists a homotopy class in $\pi_{61}(\widehat{X^{22}})$, which has Adams filtration at most 5. By naturality of the Adams spectral sequence for the quotient map $\widehat{X^{22}} \twoheadrightarrow S^{22}$, the class that detects $\sigma \{d_1\}[22]$ in $Ext(\widehat{X^{22}})$ must map nontrivially to $Ext(S^{22})$. Since the elements $\underline{h_1h_3h_5[22]}$ and $h_5c_0[22]$ are already accounted for, by filtration arguments, the only possibility is $h_3d_1[22]$. In particular, the element $h_3d_1[22]$ is a permanent cycle in the Adams spectral sequence of $\widehat{X^{22}}$. Therefore, its image in $Ext(\widetilde{X})$ is also a permanent cycle.
\end{proof}

\begin{rem}
For the element $h_5c_0[22]$, if we use the $H\mathbb{F}_2$-subcomplex $X^{22}$ instead of $\widehat{X^{22}}$, it would support an Adams $d_2$ differential that kills $h_1h_5d_0[14]$. With the 16-cell, $h_1h_5d_0[14]$ is killed by $h_5d_0[16]$ in the Curtis table, therefore isn't present in the Adams $E_2$ page of $\widehat{X^{22}}$.
\end{rem}

\begin{lem}
The element $h_1f_1[20]$ is a permanent cycle in the Adams spectral sequence of $\widetilde{X}$.
\end{lem}

\begin{proof}
We consider the $H\mathbb{F}_2$-subcomplex $X^{20}$. Using the algebraic Atiyah-Hirzebruch spectral sequence, we compute the Adams $E_2$ page of $X^{20}$ in the 61 stem for $s\leq 5$. This computation is straightforward: all differentials in this range follow by the multiplication by 2 attaching maps.

\begin{table}[h]
\caption{The Adams $E_2$ page of $X^{20}$ and $S^{20}$ in the 61 stem for $s\leq 5$}
\centering
\begin{tabular}{ l l l }
$s\backslash 61-\text{stem of}$ & $X^{20}$ & $S^{20}$ \\ [0.5ex] 
\hline 
5 & $h_1f_1[20]$ & $h_1f_1[20]$\\
  & $h_2g_2[14]$ &  \\ \hline
4 & $g_2[17]$ & $\bullet$ \\ \hline
3 & & $\bullet$ \\
\end{tabular}
\label{X20}
\end{table}

The complex $X^{20}$ fits into a cofiber sequence
\begin{displaymath}
    \xymatrix{
    X^{20} \ar@{->>}[rr]^{q'''} & & S^{20} \ar[rr]^{a'''} & & \Sigma (S^{19}\vee X^{18}) \\
    *+[o][F-]{20} \ar@{-}[d]_{2} \ar@{-}@/^1pc/[dd]^{\eta} \ar@{->>}[rr] & & \xybox{(0,0.5)*+[o][F-]{20}} \ar[rr]^{2} \ar[rrd]^{\eta} & & \xybox{(0,0.5)*+[o][F-]{20}} \\
    *+[o][F-]{19} & & & & *+[o][F-]{19} \ar@{-}[d]^{2} \ar@{-} `r[dddd] `[dddd]^{\nu} [dddd] \\
    *+[o][F-]{18} \ar@{-}[d]^{2} \ar@{-} `r[dddd] `[dddd]^{\nu} [dddd] & & & & *+[o][F-]{18} \\
    *+[o][F-]{17} & & & &  \\
   & & & & \\
   & & & & *+[o][F-]{15} \\
    *+[o][F-]{14}
    }
\end{displaymath}
Here $q'''$ is the quotient map, $X^{18}$ is the 18-skeleton of $X^{20}$, $a'''$ is suspension of the attaching map of the 20-cell in $X^{20}$. We have a long exact sequence of homotopy groups associated to this cofiber sequence. Suppose $\alpha'''[20]$ is an element in $\pi_{61}(S^{20})$. Suppose further that $\alpha'''[20]$ supports a differential in the Atiyah-Hirzebruch spectral sequence of $X^{20}$. By the construction of the Atiyah-Hirzebruch spectral sequence, the target of the differential that $\alpha'''[20]$ supports detects $\Delta'''(\alpha'''[20])$ in the homotopy groups of lower skeleton, where the map
$$\Delta''' : \pi_{61}(S^{22}) \longrightarrow \pi_{61}(\Sigma X^{21})$$
is the boundary homomorphism in the long exact sequence of homotopy groups.

By Lemma 4.4 and Remark 5.8, we have $S^{19}$ as an $H\mathbb{F}_2$-subcomplex of $X^{20}$. We consider its cofiber $X^{20}/S^{19}$.
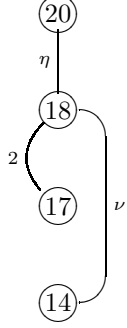
\begin{displaymath}
    \xymatrix{
    *+[o][F-]{20} \ar@{-}[d]_{\eta}  \\
    *+[o][F-]{18} \ar@{-}@/_1pc/[d]_{2} \ar@{-}`r[dd] `[dd]^{\nu} [dd] \\
    *+[o][F-]{17} \\
    *+[o][F-]{14} \\
    X^{20}/S^{19}
    }
\end{displaymath}

For the element $h_1f_1[20]$, we consider the homotopy class $\alpha''' = \sigma \{h_0h_2h_5\} \in \pi_{41}$. Because of Lemma 11.4, $h_1f_1$ detects $\sigma \{h_0h_2h_5\}$ in the Adams $E_\infty$-page of $S^0$. By Proposition 6.4, the element $\alpha'''[20]$ is a permanent cycle in the Atiyah-Hirzebruch spectral sequence of $X^{20}/S^{19}$.

In the Atiyah-Hirzebruch spectral sequence of $X^{20}$, we have the differential
$$d_1(\alpha'''[20]) = 0$$
since the attaching map from the 20-cell to the 19-cell is multiplication by 2 and
$$2 \cdot \alpha''' \in 2 \cdot \pi_{41} =0.$$
Using the fact that the 19-cell of the 19-skeleton of $X^{20}$ splits off, and the naturality of the Atiyah-Hirzebruch spectral sequences for the quotient map $X^{20} \twoheadrightarrow X^{20}/S^{19}$, the element $\alpha'''[20]$ survives in the Atiyah-Hirzebruch spectral sequence of $X^{20}$.

Therefore, by exactness of the long exact sequence of homotopy groups, there exists a homotopy class in $\pi_{61}(X^{20})$, which has Adams filtration at most 5. By naturality of the Adams spectral sequence for the quotient map $X^{20} \twoheadrightarrow S^{20}$, the class that detects $\alpha'''[20]$ in $Ext(X^{20})$ must map nontrivially to $Ext(S^{20})$.
\begin{displaymath}
    \xymatrix{
  Ext(X^{20}) \ar[r] \ar@{=>}[d] & Ext(S^{20}) \ar@{=>}[d] \\
  \pi_\ast(X^{20}) \ar[r] & \pi_\ast (S^{20})
    }
\end{displaymath}
By filtration arguments, the only possibility is $h_1f_1[20]$. In particular, the element $h_1f_1[20]$ is a permanent cycle in the Adams spectral sequence of $X^{20}$. Therefore, its image in $Ext(\widetilde{X})$ is also a permanent cycle.
\end{proof}

\section{The Adams spectral sequence of $X$} \label{34}

In this section, we establish Step 3 and Step 4 by proving Theorems 9.1 and 9.2. Combining them together, we have Corollary 9.3.

\begin{thm}
In the Adams spectral sequence of $X$, we have the differential
$$d_4(h_4^3[16]) = B_1[14].$$
\end{thm}

The following Theorem 9.2 is a consequence of Lemma 8.8.

\begin{thm}
In the Adams spectral sequence of $X$, the chosen element $\underline{h_1h_3h_5[22]}$ is a permanent cycle. Here $\underline{h_1h_3h_5[22]}$ is defined to be the image of $h_1h_3h_5[22]$ in $Ext(X^{22})$.
\end{thm}

\begin{proof}
Since the map $X^{22} \hookrightarrow X$ maps through $\widetilde{X}$, we have $\underline{h_1h_3h_5[22]}$ in $Ext(\widetilde{X})$ maps to $\underline{h_1h_3h_5[22]}$ in $Ext(X)$.
\begin{displaymath}
    \xymatrix{
  Ext^{3,61+3}(X^{22}) \ar[rr] & & Ext^{3,61+3}(\widetilde{X}) \ar[rr] & & Ext^{3,61+3}(X) \\
 h_1h_3h_5[22] \ar@{|->}[rr] & & \underline{h_1h_3h_5[22]} \ar@{|->}[rr] & & \underline{h_1h_3h_5[22]}
    }
\end{displaymath}
By Lemma 8.8, $\underline{h_1h_3h_5[22]}$ is a permanent cycle in $Ext(\widetilde{X})$. Therefore, by naturality of the Adams spectral sequences, $\underline{h_1h_3h_5[22]}$ is also a permanent cycle in $Ext(X)$.
\end{proof}

From Theorem 9.1 and 9.2, we have the following corollary.

\begin{cor}
In the Adams spectral sequence of $X$, we have the differential
$$d_4(\underline{h_1h_3h_5[22]} + h_4^3[16]) = B_1[14].$$
\end{cor}

In the rest of this section, we prove Theorem 9.1. The idea is to push the $d_4$ differential in the Adams spectral sequence of $\widetilde{X}$ into that of $X$, and check the element $B_1[14]$ is not killed by an Adams $d_2$ or $d_3$ differential.

\begin{proof}
Recall from Remark 5.11 that the Adams $E_2$ page of $X$ splits as follows:
$$Ext(X) = Ext(\widetilde{X}) \oplus Ext(S^{23}).$$
Therefore, by Lemma 8.2, we have the Adams $E_2$ page of $X$ in the 60 and 61 stems for $s\leq 7$ in the following Table 6.

\begin{table}[h]
\caption{The Adams $E_2$ page of $X$ in the 60 and 61 stems for $s\leq 7$}
\centering
\begin{tabular}{ l l l }
$s\backslash t-s$ & 60 & 61 \\ [0.5ex] 
\hline 
7 & $B_1[14]$ & $\bullet$\\
  & $\bullet$ & $\bullet$ \\
  & $h_1t[23]$ & $\bullet$ \\
  & $h_0^2x[23]$ & \\ \hline
6 & $\bullet$ & $\bullet$\\
  & $\bullet$ & $\bullet$ \\
  & $\bullet$ & $\bullet$ \\
  & $h_0x[23]$ & $\bullet$ \\
  & & $\bullet$ \\ \hline
5 & $\bullet$ & $\bullet$\\
  & $\bullet$ & $\bullet$\\
  & $x[23]$ & $\bullet$\\
  & & $\bullet$\\
  & & $\bullet$\\
  & & $h_0^3h_3h_5[23]$\\ \hline
4 & $\bullet$ & $\bullet$\\
  & & $\bullet$\\
  & & $\bullet$\\
  & & $\bullet$\\
  & & $\bullet$\\
  & & $e_1[23]$\\
  & & $h_0^2h_3h_5[23]$\\ \hline
3 & $\bullet$ & $h_4^3[16]$\\
  & $\bullet$ & $\underline{h_1h_3h_5[22]}$ \\
  & & $h_0h_3h_5[23]$ \\ \hline
2 & & $h_3h_5[23]$ \\
\end{tabular}
\label{XX}
\end{table}

Note that by naturality of the Adams spectral sequences for the inclusion map $\widetilde{X} \hookrightarrow X$, and the proof of the Theorem 8.1, no $\bullet$'s in Adams filtration 4 and 5 can kill $B_1[14]$. Therefore, to prove Theorem 9.1, we only need to show that
$$d_2(h_0^3h_3h_5[23]) \neq B_1[14],$$
$$d_3(h_0^2h_3h_5[23]) \neq B_1[14],$$
$$d_3(e_1[23]) \neq B_1[14].$$

For the elements $h_0^2h_3h_5[23]$ and $h_0^3h_3h_5[23]$, we will show that
$$d_2(h_3h_5[23]) = 0,$$
$$d_3(h_3h_5[23]) = 0,$$
which by Leibniz's rule implies that
$$d_2(h_0^3h_3h_5[23]) = h_0^3 \cdot d_2(h_3h_5[23]) = 0,$$
$$d_3(h_0^2h_3h_5[23]) = h_0^2 \cdot d_3(h_3h_5[23]) = 0.$$
We consider the $H\mathbb{F}_2$-subcomplex $X^{23}$ in Definition 5.10. Recall that $X^{23}$ consists of two cells in dimension 14 and 23. Since there is no primary Steenrod operation connecting them, we have
$$Ext(X^{23}) = Ext(S^{14}) \oplus Ext(S^{23}).$$
Therefore, we have the Adams spectral sequence of $X^{23}$ in the 60 and 61 stems for $s\leq 5$ in the following Table 7.
\begin{table}[h]
\caption{The Adams $E_2$ page of $X^{23}$ in the 60 and 61 stems for $s\leq 5$}
\centering
\begin{tabular}{ l l l }
$s\backslash t-s$ & 60 & 61 \\ [0.5ex] 
\hline 
5 & $x[23]$ & $h_0^3h_3h_5[23]$\\
  & & $\bullet$\\ \hline
4 & & $e_1[23]$\\
  & & $h_0^2h_3h_5[23]$\\ \hline
3 & $\bullet$ & $h_0h_3h_5[23]$ \\ \hline
2 & & $h_3h_5[23]$ \\
\end{tabular}
\label{X23}
\end{table}
In the Adams spectral sequence of $X^{23}$, we have $d_2(h_3h_5[23])=0$, since the target lies in the zero group. If $d_3(h_3h_5[23]) \neq 0$, then we must have that $d_3(h_3h_5[23]) = x[23]$, since that is the only possibility. By mapping through the quotient map $X^{23} \twoheadrightarrow S^{23}$, this differential would imply that $d_3(h_3h_5[23]) = x[23]$ in the Adams spectral sequence of $S^{23}$. However, in $S^0$, we have that $d_3(h_3h_5) = 0$. Contradiction! Therefore, we must have the differential $d_3(h_3h_5[23]) = 0$ in the Adams spectral sequence of $X^{23}$, and therefore also in that of $X$.

For the element $e_1[23]$, suppose we have $d_3(e_1[23]) = B_1[14]$ in the Adams spectral sequence of $X$. By naturality for the quotient map $X \twoheadrightarrow S^{23}$, we have $d_3(e_1[23]) = 0$ in the Adams spectral sequence of $S^{23}$, since the target $B_1[14]$ maps to zero in the $E_2$-page by naturality of the algebraic Atiyah-Hirzebruch spectral sequences. However, this contradicts Bruner's differential \cite[Theorem 4.1]{Br1} in $S^0$:
$$d_3(e_1) = h_1t.$$
Therefore, we must have $d_3(e_1[23]) \neq B_1[14]$, which completes the proof.
\end{proof}

\section{The pull back}

In this section, we prove Step 5: based on Corollary 9.3, we prove the following Theorem 10.1.

\begin{thm}
In the Adams spectral sequence of $P_1^{23}$, we have a $d_3$ differential:
$$d_3(h_1h_3h_5[22]) = G[6].$$
\end{thm}

\begin{proof}
We have the Adams $E_2$-page of $P_1^{23}$ from the Curtis table.
\begin{table}[h]
\caption{The Adams $E_2$-page of $P_1^{23}$ in the 60 and 61 stems for $s\leq 7$}
\centering
\begin{tabular}{ l l l }
$s\backslash t-s$ & 60 & 61 \\ [0.5ex] 
\hline 
7 & $\bullet[3]$ & $\bullet$\\
  & $\bullet[5]$ & $\bullet$ \\
  & $\bullet[21]$ & $\bullet$ \\
  & $\bullet[23]$ & $\bullet$ \\
  & $\bullet[23]$ & $\bullet$ \\
  & & $\bullet$ \\ \hline
6 & $G[6]$ & $\bullet$\\
  & $\bullet[20]$ & $\bullet$ \\
  & $\bullet[22]$ & $\bullet$ \\
  & $\bullet[23]$ & $\bullet$ \\
  & & $\bullet$ \\ \hline
5 &  & $\bullet$\\
  &  & $\bullet$ \\
  &  & $\bullet$ \\
  & & $\bullet$ \\
  & & $\bullet$ \\
  & & $\bullet$ \\ \hline
4 & $\bullet$ & $\bullet$\\
  & & $\bullet$\\
  & & $\bullet$ \\ \hline
3 & $\bullet$ & $\bullet$ \\
 & & $h_1h_3h_5[22]$\\
\end{tabular}
\label{P23}
\end{table}

We will show in Lemma 10.3 that
$$f_\sharp (h_1h_3h_5[22]) = \underline{h_1h_3h_5[22]} + h_4^3[16],$$
where $f_\sharp: Ext(P_1^{23}) \rightarrow Ext(X)$ is induced by the composition of the two quotient maps $f_1: P_1^{23} \twoheadrightarrow P_{14}^{23}$, $f_2: P_{14}^{23} \twoheadrightarrow X$.
By Corollary 9.3 that in the Adams spectral sequence of $X$, we have the differential
$$d_4(\underline{h_1h_3h_5[22]} + h_4^3[16]) = B_1[14],$$
and the naturality of the Adams spectral sequence, the element $h_1h_3h_5[22]$ in $Ext(P_1^{23})$ must support a nontrivial $d_2$, $d_3$ or $d_4$ differential.

\begin{displaymath}
    \xymatrix{
 Ext(P_1^{23}) \ar[r]^{f_\sharp} & Ext(X) \\
 & B_1[14] \\
 \bullet & \\
 h_1h_3h_5[22] \ar@{-->}[u]^{d_r, \ 2\leq r \leq4}  \ar@{|->}[r] & *\txt{ $\underline{h_1h_3h_5[22]}$\\$+h_4^3[16]$} \ar@{-->}[uu]^{d_4}
    }
\end{displaymath}

From the table of the Adams $E_2$-page of $P_1^{23}$, we have the following three possibilities.

\begin{enumerate}
\item
It supports a nontrivial $d_3$ or $d_4$ differential that kills one of the elements $\bullet[i]$ with $20\leq i \leq 23$.
\item
It supports a nontrivial $d_4$ differential that kills one of the elements $\bullet[i]$ with $i = 3, 5$.
\item
It supports a nontrivial $d_3$ differential that kills $G[6]$.
\end{enumerate}

For $(1)$, since these target elements map nontrivially to $Ext(X)$, this would contradict Theorem 9.1. For $(2)$, from the Curtis table, these two elements exist in $Ext(P_1^n)$ for all $n\geq 5$. In particular, they exist in $Ext(P_1^{13})$, and map trivially to $Ext(P_{14}^{23})$ in the following long exact sequence
\begin{displaymath}
    \xymatrix{
    \cdots \ar[r] & Ext(P_1^{13}) \ar[r] & Ext(P_1^{23}) \ar[r] & Ext(P_{14}^{23}) \ar[r] & \cdots,
    }
\end{displaymath}
and hence trivially to $Ext(X)$. Since they have the same filtration as $B_1[14]$, this would contradict Theorem 9.1.

Therefore, $(3)$ is the only possibility.
\end{proof}

\begin{rem}
The reason we use $P_1^{23}$ instead of $P_1^{22}$ is that, in the bidegree $(s,t-s) = (5, 60)$ of the Curtis table, the element $h_5f_0[11]$ is killed by a $\bullet[23]$. Therefore, in $Ext(P_1^{22})$, the element $h_5f_0[11]$ is present, and leaves a possibility of a nontrivial Adams $d_2$ differential. We add the 23-cell to make this go away.
\end{rem}

We now prove Lemma 10.3.

\begin{lem}
We have
$$f_\sharp (h_1h_3h_5[22]) = \underline{h_1h_3h_5[22]} + h_4^3[16],$$
where $f_\sharp: Ext(P_1^{23}) \rightarrow Ext(X)$ is the homomorphism induced by the composition of the two quotient maps
$$f_1: P_1^{23} \twoheadrightarrow P_{14}^{23}, \ f_2: P_{14}^{23} \twoheadrightarrow X.$$
\end{lem}

\begin{proof}
By naturality of the algebraic Atiyah-Hirzebruch spectral sequences, we have
$$f_{1\sharp} (h_1h_3h_5[22]) = h_1h_3h_5[22].$$
We only need to show that
$$f_{2\sharp} (h_1h_3h_5[22]) = \underline{h_1h_3h_5[22]} + h_4^3[16].$$
\begin{displaymath}
    \xymatrix{
    & & Ext(X^{22}) \ar[d]^{i_\sharp} \\
 Ext(P_{1}^{\infty}) \ar[r]^{f_{1\sharp}} & Ext(P_{14}^{23}) \ar[r]^{f_{2\sharp}} & Ext(X) \\
 h_1h_3h_5[22] \ar@{|->}[r] & h_1h_3h_5[22] \ar@{|->}[r] & *\txt{ $\underline{h_1h_3h_5[22]}$\\$+h_4^3[16]$}
    }
\end{displaymath}
Consider the cofiber sequence that defines $X$
\begin{displaymath}
    \xymatrix{
S^{15} \ar@{^{(}->}[r] & P_{14}^{23} \ar@{->>}[r]^{f_{2}} & X \ar[r] & \Sigma S^{15}.
    }
\end{displaymath}
This gives a long exact sequence of $Ext$ groups:
\begin{displaymath}
    \xymatrix{
\cdots \ar[r] & Ext(S^{15}) \ar[r] & Ext(P_{14}^{23}) \ar@{->>}[r]^{f_{2\sharp}} & Ext(X) \ar[r]^-{\Delta_2} & Ext(\Sigma S^{15}) \ar[r] & \cdots.
    }
\end{displaymath}
We only need to show that the boundary map $\Delta_2$ satisfies
$$\Delta_2 (\underline{h_1h_3h_5[22]} + h_4^3[16]) = 0.$$
In fact, by exactness, the element $\underline{h_1h_3h_5[22]} + h_4^3[16]$ must come from $Ext(P_{14}^{23})$. By naturality of the algebraic Atiyah-Hirzebruch spectral sequence, it must comes from $h_1h_3h_5[22]$, i.e., we must have
$$f_{2\sharp} (h_1h_3h_5[22]) = \underline{h_1h_3h_5[22]} + h_4^3[16],$$
which completes the proof.

To show $\Delta_2 (\underline{h_1h_3h_5[22]} + h_4^3[16]) = 0$, we consider an $H\mathbb{F}_2$-subcomplex $W$ of $X$. Since $X^{22}$ is an $H\mathbb{F}_2$-subcomplex of $X$, we define $W$ to be the homotopy pull back of $X^{22}$ along the quotient map $f_2: P_{14}^{23} \twoheadrightarrow X$. By Lemma 4.4, we have $(W, j)$ as an $H\mathbb{F}_2$-subcomplex of $P_{14}^{23}$ in the following commutative diagram of cofiber sequences:
\begin{displaymath}
    \xymatrix{
S^{15} \ar@{^{(}->}[r] \ar@{=}[d] & W \ar@{->>}[r] \ar@{^{(}->}[d]^j & X^{22} \ar[r]^{a_1} \ar@{^{(}->}[d]^i & \Sigma S^{15} \ar@{=}[d]\\
S^{15} \ar@{^{(}->}[r] & P_{14}^{23} \ar@{->>}[r]^{f_{2}} & X \ar[r]^{a_2} & \Sigma S^{15}
    }
\end{displaymath}
As an illustration, the cell diagram of $W$ is the following:
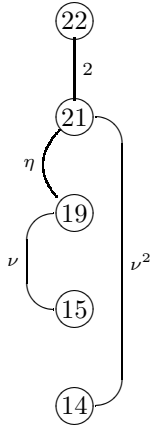
\begin{displaymath}
    \xymatrix{
    *+[o][F-]{22} \ar@{-}[d]^{2}  \\
    *+[o][F-]{21} \ar@{-}@/_1pc/[d]_{\eta} \ar@{-}`r[ddd] `[ddd]^{\nu^2} [ddd]\\
    *+[o][F-]{19} \ar@{-}`l[d] `[d]_{\nu} [d]\\
    *+[o][F-]{15} \\
    *+[o][F-]{14}
    }
\end{displaymath}
We will show in the following Lemma 10.4 that
$$\Delta_1 (h_1h_3h_5[22]) = h_0h_4^3[15],$$
where $\Delta_1$ is the boundary map of $Ext$ groups associated to the cofiber sequence defining $W$.
Therefore, following the commutative diagram of cofiber sequences and the definition of the element $\underline{h_1h_3h_5[22]}$, we have
$$\Delta_2 (\underline{h_1h_3h_5[22]}) = h_0h_4^3[15].$$
The fact that the 16-cell in $P_{14}^{23}$ is attached to the 15-cell by $2$ gives us
$$\Delta_2 (h_4^3[16]) = h_0h_4^3[15].$$
Therefore, we have
$$\Delta_2 (\underline{h_1h_3h_5[22]} + h_4^3[16]) = 0,$$
as claimed.
\end{proof}

\begin{lem}
$\Delta_1 (h_1h_3h_5[22]) = h_0h_4^3[15]$.
\end{lem}

\begin{proof}
We use the Lambda complex (see Section 7.1 of \cite{Pri}) to compute the $E_2$-page of the Adams spectral sequence in a functorial way. Recall from \cite{Pri} that, for any spectrum $Y$, we can construct a differential graded module $H_*(Y)\otimes\Lambda^{*,*}$ over the Lambda algebra $\Lambda^{*,*}$. Differentials in this complex are generated by
$$d(x) = \Sigma_{i\geq1} Sq^i_*(x)\otimes\lambda_{i-1}$$
for $x\in H_*(Y)$, where $Sq^i_*$ is the transpose of $Sq^i$.

In our case, we abuse notation to denote the unique generator of $H_i(Y)$ by $e_i$, for any $H\mathbb{F}_2$-subquotient of $X$.

By naturality of the Steenrod operations, we have nontrivial $Sq^4$ and $Sq^8$ in the cohomology of $W$.
\begin{displaymath}
    \xymatrix{
H^{15}(W) \ar[r]^{Sq^4 \neq 0} & H^{19}(W) & & H^{14}(W) \ar[r]^{Sq^8 \neq 0} & H^{22}(W) \\
H^{15}(P_{14}^{23}) \ar[r]^{Sq^4 \neq 0} \ar[u]^{j^\ast}_\cong & H^{19}(P_{14}^{23}) \ar[u]^{j^\ast}_\cong & & H^{14}(P_{14}^{23}) \ar[r]^{Sq^8 \neq 0} \ar[u]^{j^\ast}_\cong & H^{22}(P_{14}^{23}) \ar[u]^{j^\ast}_\cong
    }
\end{displaymath}
Moreover, in the cohomology of $W$, we have $Sq^1 Sq^2 Sq^4 \neq 0$ on $H^{15}$. Dually, we have the following nontrivial operations:
\begin{equation*}
\begin{split}
Sq^1_*(e_{22}) & = e_{21}, \\
Sq^3_*(e_{22}) & = e_{19}, \\
Sq^7_*(e_{22}) & = e_{15}, \\
Sq^8_*(e_{22}) & = e_{14}.
\end{split}
\end{equation*}
By naturality, we have the following nontrivial operations in $H_\ast(X^{22})$:
\begin{equation*}
\begin{split}
Sq^1_*(e_{22}) & = e_{21}, \\
Sq^3_*(e_{22}) & = e_{19}, \\
Sq^8_*(e_{22}) & = e_{14}.
\end{split}
\end{equation*}

We claim that in $H_*(X^{22})\otimes\Lambda^{*,*}$ the cycle
$$x = e_{22}\otimes \lambda_1\lambda_7\lambda_{31}+e_{14}\otimes\lambda_{13}\lambda_{19}\lambda_{15}$$
represents the class $h_1h_3h_5[22]$ in $Ext(X^{22})$.

In fact, we can check directly that $x$ is a cycle:
\begin{equation*}
\begin{split}
d(e_{22}\otimes \lambda_1\lambda_7\lambda_{31}) & = e_{21}\otimes \lambda_0\lambda_1\lambda_7\lambda_{31} + e_{19}\otimes \lambda_2\lambda_1\lambda_7\lambda_{31} + e_{14}\otimes \lambda_7\lambda_1\lambda_7\lambda_{31} \\
& = e_{14}\otimes \lambda_7\lambda_1\lambda_7\lambda_{31} \\
& = e_{14}\otimes(\lambda_{13}\lambda_{15}\lambda_{11}\lambda_7 + \lambda_{11}\lambda_{17}\lambda_{11}\lambda_7 + \lambda_7\lambda_{13}\lambda_{11}\lambda_{15}),\\
d(e_{14}\otimes\lambda_{13}\lambda_{19}\lambda_{15}) & = e_{14}\otimes d(\lambda_{13}\lambda_{19}\lambda_{15}) \\
& = e_{14}\otimes(\lambda_{13}\lambda_{15}\lambda_{11}\lambda_7 + \lambda_{11}\lambda_{17}\lambda_{11}\lambda_7 + \lambda_7\lambda_{13}\lambda_{11}\lambda_{15}).
\end{split}
\end{equation*}
We compute
$$\lambda_1\lambda_7\lambda_{31} = \lambda_{21}\lambda_{11}\lambda_7 + \lambda_{13}\lambda_{11}\lambda_{15},$$
and check the Curtis table in \cite{Tan1} that $Ext^{3,3+39} = \mathbb{Z}/2$, generated by an element with the leading term $\lambda_{21}\lambda_{11}\lambda_7$. Since $Ext^{3,3+39} = \mathbb{Z}/2$ is generated by $h_1h_3h_5$, we conclude that $x$ represents the class $h_1h_3h_5[22]$ in $Ext(X^{22})$.

However, in $H_*(W) \otimes \Lambda^{*,*}$ the element
$$x = e_{22}\otimes \lambda_1\lambda_7\lambda_{31}+e_{14}\otimes\lambda_{13}\lambda_{19}\lambda_{15}$$
is not a cycle anymore: there is one more term in $d(x)$ due to the extra nontrivial operation
$$Sq^7_*(e_{22}) = e_{15}.$$
In fact, we have that
\begin{equation*}
\begin{split}
d(x) & = d(e_{22}\otimes \lambda_1\lambda_7\lambda_{31} + e_{14}\otimes\lambda_{13}\lambda_{19}\lambda_{15}) \\
& = e_{21}\otimes \lambda_0\lambda_1\lambda_7\lambda_{31} + e_{19}\otimes \lambda_2\lambda_1\lambda_7\lambda_{31} + e_{15}\otimes \lambda_6\lambda_1\lambda_7\lambda_{31} \\
& \ \ \ + e_{14}\otimes \lambda_7\lambda_1\lambda_7\lambda_{31} + e_{14}\otimes d(\lambda_{13}\lambda_{19}\lambda_{15}) \\
& = e_{15}\otimes \lambda_6\lambda_1\lambda_7\lambda_{31} \\
& = e_{15}\otimes \lambda_{14}\lambda_{13}\lambda_{11}\lambda_7.
\end{split}
\end{equation*}
Therefore, by the definition of the boundary homomorphism $\Delta_1: Ext(X^{22}) \rightarrow Ext(\Sigma S^{15})$, we have
$$\Delta_1 (x) = e_{15}\otimes \lambda_{14}\lambda_{13}\lambda_{11}\lambda_7.$$

\begin{displaymath}
    \xymatrix{
   H_\ast(S^{15})\otimes\Lambda^{*,*}  \ar[r] & H_\ast(W)\otimes\Lambda^{*,*} \ar[r] & H_\ast(X^{22})\otimes\Lambda^{*,*} \\
   & x \ar@{|->}[r] \ar[d]^d & x \\
   e_{15}\otimes \lambda_{14}\lambda_{13}\lambda_{11}\lambda_7  \ar@{|->}[r] & e_{15}\otimes \lambda_{14}\lambda_{13}\lambda_{11}\lambda_7  &
    }
\end{displaymath}

We check the Curtis table in \cite{Tan1} that $Ext^{4,4+45} = \mathbb{Z}/2$, generated by an element with the leading term $\lambda_{14}\lambda_{13}\lambda_{11}\lambda_7$. Since $Ext^{4,4+45} = \mathbb{Z}/2$ is generated by $h_0h_4^3$, we conclude that $e_{15}\otimes \lambda_{14}\lambda_{13}\lambda_{11}\lambda_7$ represents the class $h_0h_4^3[15]$ in $Ext(\Sigma S^{15})$.
\end{proof}

\begin{rem}
One can think of the boundary homomorphism in Lemma 10.4 as an algebraic attaching map, and therefore its computation corresponds to a 4-fold Massey product. In $Ext(S^0)$, we have the strictly defined 4-fold Massey product
$$h_0h_4^3 = \langle h_2, h_1, h_0, h_1h_3h_5 \rangle$$
with zero indeterminacy. It is straightforward to check this by a Lambda algebra computation:
\begin{displaymath}
    \xymatrix@=0.02in{
  \langle h_2 & , & h_1 & , & h_0 & , & h_1h_3 \rangle \\
  \lambda_3 & & \lambda_1 & & \lambda_0 & & \lambda_5 \lambda_3 \\
  & \lambda_5 & & \lambda_2 & & \ast & \\
  & & \lambda_6 & & \ast & &
    }
\end{displaymath}
Here $\ast$ means the products are zero in the Lambda algebra. Note that the leading term of $h_0h_3^2$ is $\lambda_6 \lambda_5 \lambda_3$ from the Curtis table for $S^0$. Therefore,
$$h_0h_3^2 = \langle h_2, h_1, h_0, h_1h_3 \rangle.$$
Then it follows from a relation in $Ext$: $h_0h_4^3 = h_0h_3^2h_5$.
\end{rem}

\section{A homotopy relation}

In this section, we prove a relation in the homotopy groups of spheres. This relation will lead to an Adams differential that kills the element $gz$ in the 61-stem. We will explain in Remark 11.2 which element supports the differential that kills $gz$. But to prove $\pi_{61} = 0$, all we need is that $gz$ is gone. We will use certain relations in $Ext$ in the proofs, see \cite{Br2} for these relations.

\begin{thm} \label{gz}
We have the homotopy relation $\eta \overline{\kappa}^3 =0$ in $\pi_{61}$. Therefore the element $gz$ must be killed by some Adams differential.
\end{thm}

Using several lemmas that will be proved later in this section, we present the proof of Theorem 11.1.

\begin{proof}
We first prove the second claim. By \cite[Corollary 3.4.2]{BMT}, the permanent cycle $z$ in the 41-stem detects the homotopy class $\eta \overline{\kappa}^2$. It follows that the element $gz$ detects $\eta \overline{\kappa}^3$, since $g$ detects $\overline{\kappa}$. Therefore, if $\eta \overline{\kappa}^3 =0$, we must have $gz$ killed by some Adams differential.

Now we prove the relation $\eta \overline{\kappa}^3 =0$.

We have a 4-fold Toda bracket for $\overline{\kappa}$ \cite[page 43-44]{MT2}:
$$\overline{\kappa} \in \langle \kappa, 2, \eta, \nu\rangle \text{~~with indeterminacy even multiples of~~}\overline{\kappa}.$$
The indeterminacy will be killed after multiplying by $\eta$. We will prove in Lemma 11.3 that
$$\langle \eta \overline{\kappa}^2, \kappa, 2\rangle =0\text{~~in~~}\pi_{56}.$$
Therefore
\begin{equation*}
\begin{split}
\eta \overline{\kappa}^3 & = \eta \overline{\kappa}^2 \langle \kappa, 2, \eta, \nu\rangle \\
& \subseteq \langle\langle \eta \overline{\kappa}^2 , \kappa, 2 \rangle, \eta, \nu\rangle \\
& = \langle 0, \eta, \nu\rangle \\
& = \nu \cdot \pi_{58} \\
& = 0
\end{split}
\end{equation*}
The last equation is stated as Lemma 11.7 that we will prove later in this section. Therefore, we have the homotopy relation
$$\eta \overline{\kappa}^3=0\text{~~in~~}\pi_{61}.$$
\end{proof}

\begin{rem}
Alternatively, we can show that $h_1X_1$ must support an Adams differential, and
$$d_4(h_1X_1) = gz$$
is the only possibility. The idea is to consider the Massey product $\langle g^2, d_0^2, h_1\rangle = h_1W_1 + g^2 r$ in the Adams $E_4$-page, and to conclude that $h_1W_1$ must support a nontrivial differential as $g^2 r$ does (See Lemma 3.3.49 of \cite{Isa}), since the sum is a permanent cycle by Moss's Theorem. Suppose that $h_1X_1$ is a permanent cycle. We have that
\begin{equation*}
\begin{split}
h_1W_1 & = Ph_1 X_1 \\
 & = X_1 \langle h_1, h_0^3h_3, h_0 \rangle \\
 & = \langle h_1X_1, h_0^3h_3, h_0 \rangle
\end{split}
\end{equation*}
is also a permanent cycle by Moss's Theorem. We therefore have a contradiction.
\end{rem}

We first prove Lemma 11.3.

\begin{lem}
We have a Toda bracket $\langle \eta \overline{\kappa}^2, \kappa, 2\rangle =0$ in $\pi_{56}$.
\end{lem}

\begin{proof}
By \cite{Isa,Isa2},
$$\pi_{55} \cong \mathbb{Z}/16 \text{~~and is generated by an element~~} \rho_{55} \text{~~in~~} Im J.$$
Therefore, we have the relation
$$\eta \overline{\kappa}^2 \kappa =0 \text{~~in~~} \pi_{55}.$$
This follows from the fact that both $\kappa$ and $\overline{\kappa}$ map trivially to the $K(1)$-local sphere. In fact, suppose that $\eta \overline{\kappa}^2 \kappa$ is some multiple of $\rho_{55}$. Then mapping the relation to the $K(1)$-local sphere tells us the multiple must be zero. Therefore, this Toda bracket is defined.

By \cite{Isa,Isa2},
$$\pi_{56} \cong \mathbb{Z}/2 \text{~~and is generated by~~} \eta \rho_{55} \text{~~in~~} Im J.$$
Therefore, we have the relation
$$\langle \eta \overline{\kappa}^2, \kappa, 2\rangle =0.$$
This follows similarly by mapping the Toda bracket to the $K(1)$-local sphere.\\
\end{proof}

To prove Lemma 11.7, we need the following three lemmas.

\begin{lem}
The product $\sigma \cdot \{h_0h_2h_5\}$ is nontrivial in $\pi_{41}$, and is detected by $h_1f_1$.
\end{lem}

\begin{proof}
By \cite{MT}, we have the following two Adams differentials
$$d_3(h_2h_5) = h_1d_1, \text{~~and~~} d_2(h_0c_2) = h_1h_3d_1.$$
Note that we have a relation $h_3d_1=h_1e_1$ in $Ext$. Therefore, we have a Massey product in the Adams $E_4$-page
$$\langle d_1, h_1, h_0\rangle = h_0h_2h_5$$
and a Massey product in the Adams $E_3$-page
$$\langle h_3d_1, h_1, h_0\rangle = h_0^2c_2 = h_1f_1.$$
Note that the second equation is a relation in $Ext$. Then by Moss's Theorem \cite[Theorem 1.2]{Mos}, we have the following Toda brackets
$$\langle \{d_1\}, \eta, 2\rangle \text{~~contains an element that is detected by~~} h_0h_2h_5,$$
$$\langle \sigma \{d_1\}, \eta, 2\rangle \text{~~contains an element that is detected by~~} h_1f_1.$$
Since
$$\sigma \langle \{d_1\}, \eta, 2\rangle \subseteq \langle \sigma \{d_1\}, \eta, 2\rangle,$$
the product $\sigma \cdot \{h_0h_2h_5\}$ is nontrivial, and is detected by $h_1f_1$.
\end{proof}

\begin{lem}
We have the relation $\langle \{t\}, \eta, \nu\rangle \subseteq \sigma\{h_0h_2h_5\}$ in $\pi_{41}$.
\end{lem}

\begin{proof}
By \cite[Theorem 4.1]{Br1} we have Bruner's differential
$$d_3(e_1)=h_1t.$$
Therefore, we have a Massey product in the Adams $E_4$-page
$$\langle t, h_1, h_2\rangle = h_2e_1 = h_1f_1.$$
The second equation is a relation in $Ext$. Therefore, by Moss's Theorem \cite{Mos}, we have the following Toda bracket:
$$\langle \{t\}, \eta, \nu\rangle \text{~~is detected by~~} h_1f_1.$$
Note that the Toda bracket $\langle \{t\}, \eta, \nu\rangle$ has no indeterminacy. \\

Combining with Lemma 11.4, both $\sigma\{h_0h_2h_5\}$ and $\langle \{t\}, \eta, \nu\rangle$ are detected by $h_1f_1$. But in the same column of the $E_\infty$ page of the Adams spectral sequence, there are several elements with higher filtration than $h_1f_1$.  Therefore, to prove this lemma, we need to show that their difference is actually zero. We prove this by multiplying by $\eta$. First note that
$$\eta \cdot \sigma\{h_0h_2h_5\} = 0.$$
In fact, $\eta\{h_0h_2h_5\}$ contains non-zero classes $\eta\kappa\overline{\kappa} = \nu\{q\}$ and $\eta^2\{P^4h_1\}$. Both classes are annihilated by $\sigma$. Next note that
$$\langle \{t\}, \eta, \nu\rangle \eta = \{t\} \langle \eta, \nu, \eta\rangle = \{t\} \nu^2 =0.$$
For the last equation, by filtration arguments, the only other possibility is that $\{t\} \nu^2 = \kappa^3$. (For reader's convenience, note that $\kappa^3 = \eta^2\overline{\kappa}^2$.) However, mapping this relation to $\pi_\ast(tmf)$ gives a contradiction.

Since all elements of higher filtration than $h_1f_1$ in the cokernel of $J$ support non-zero $\eta$-extensions, this proves the lemma.
\end{proof}

\begin{lem}
We have a Toda bracket $\langle \overline{\kappa}, \{t\}, \eta \rangle = \{h_1Q_2\}$ in $\pi_{58}$.
\end{lem}

\begin{proof}
By \cite[Table 20]{Isa}, \cite{Isa2}, we have Isaksen's differential
$$d_3(Q_2) = gt.$$
Therefore, combining with Bruner's differential \cite[Theorem 4.1]{Br1} $d_3(e_1) = h_1t$, we have a Massey product in the Adams $E_4$-page
$$\langle g, t, h_1\rangle = h_1Q_2.$$
Note that $g e_1 =0$ in $Ext$. Therefore, the lemma follows from Moss's Theorem \cite[Theorem 1.2]{Mos}. Both sides of $\langle \overline{\kappa}, \{t\}, \eta \rangle = \{h_1Q_2\}$ have the same indeterminacy that lies in the image of J.
\end{proof}

Now we prove Lemma 11.7.

\begin{lem}
$\nu\cdot \pi_{58} =0.$
\end{lem}

\begin{proof}
By \cite{Isa,Isa2},
$$\pi_{58} \text{~~is~~}\mathbb{Z}/2\oplus\mathbb{Z}/2, \text{~~and generated by~~} \{h_1Q_2\} \text{~~and~~}\eta\{P^7h_1\}.$$
By Lemma 11.5 that
$$\langle \{t\}, \eta, \nu\rangle \subseteq \sigma\{h_0h_2h_5\}\text{~~in~~}\pi_{41},$$
and Lemma 11.6 that
$$\langle \overline{\kappa}, \{t\}, \eta \rangle = \{h_1Q_2\}\text{~~in~~}\pi_{58},$$
we have that
\begin{equation*}
\begin{split}
\nu \cdot \{h_1Q_2\} & = \langle \overline{\kappa}, \{t\}, \eta\rangle \nu \\
& = \overline{\kappa} \langle \{t\}, \eta, \nu\rangle \\
& \subseteq  \overline{\kappa} \sigma \{h_0h_2h_5\} = 0.
\end{split}
\end{equation*}
The last equation follows from the relation that $\overline{\kappa} \sigma =0$. Therefore, we have that
$$\nu\cdot \pi_{58} =0.$$
\end{proof}

\section{Another homotopy relation and the Adams differential $d_5(A') = h_1B_{21}$}

In this section, we prove another relation in the homotopy groups of spheres. This relation will lead to an Adams differential, which is the only possibility to kill the element $h_1B_{21}$ in the 60-stem.

\begin{thm} \label{A}
We have the relation $\eta \kappa \theta_{4.5} =0$ in $\pi_{60}$. Here $\theta_{4.5}$ is a homotopy class in $\pi_{45}$ defined by Isaksen in Section 1.7 of \cite{Isa}, with an extra condition that it maps to zero in $\pi_{45}(tmf)$. This implies the Adams differential
$$d_5(A') = h_1B_{21}.$$
\end{thm}

In Isaksen's definition, $\theta_{4.5}$ is a homotopy class detected by $h_4^3$ in the 45-stem, with indeterminacy containing even multiples of itself and the element $\{w\}$. Our definition of $\theta_{4.5}$ is a refinement of Isaksen's. Since $\{w\}$ has a strictly higher Adams filtration than $\theta_{4.5}$, and is detected by tmf, the indeterminacy of our $\theta_{4.5}$ does not contain the element $\{w\}$.

Using several lemmas that will be proved later in this section, we present the proof of Theorem 12.1.

\begin{proof}
We first prove the second claim. By \cite[Theorem 3.1(i)]{BJM}, the permanent cycle $B_{1}$ detects the homotopy class $\eta \theta_{4.5}$. We have the following relation in $Ext$:
$$h_1B_{21} = d_0B_1.$$
Since $d_0$ detects $\kappa$, the permanent cycle $h_1B_{21} = d_0B_1$ detects the homotopy class $\eta \kappa \theta_{4.5}$.
Therefore, if $\eta \kappa \theta_{4.5} =0$, we must have $h_1B_{21}$ killed by some Adams differential. By Theorem 3.1, we have that
$$d_3(D_3) = B_3, \ d_3(h_1D_3) = h_1B_3.$$
This leaves the element $A'$ to be the only possibility to kill $h_1B_{21}$ as the source. Therefore, we have the Adams $d_5$ differential $d_5(A') = h_1B_{21}$.

Now we prove the relation $\eta \kappa \theta_{4.5} =0$.

Recall that there is a strictly defined 4-fold Toda bracket for $\kappa\in\pi_{14}$ with zero indeterminacy:
$$\kappa = \langle \epsilon, \nu, \eta, 2\rangle.$$
It follows that
$$\eta \kappa = \eta \langle \epsilon, \nu, \eta, 2\rangle \in \langle \eta\epsilon, \nu, \eta, 2\rangle,$$
and that
$$\eta \kappa \theta_{4.5} \in \theta_{4.5} \langle \eta\epsilon, \nu, \eta, 2\rangle.$$
We will show in Lemma 12.6 that there is a strictly defined 4-fold Toda bracket in $\pi_{15}$:
$$\rho_{15} \in \langle \{Ph_1\}, \nu, \eta, 2\rangle \text{~~with indeterminacy even multiples of~~}\rho_{15}.$$
We will show in Lemma 12.7 that
$$\rho_{15} \theta_{4.5} = 0\text{~~in~~}\pi_{60}.$$
Thus
$$ 0 = \rho_{15} \theta_{4.5} = \theta_{4.5} \langle \{Ph_1\}, \nu, \eta, 2\rangle.$$
We will show in Lemma 12.5 that
$$\theta_{4.5} (\eta \epsilon + \{Ph_1\} ) = 0,$$
and in Lemma 12.9 that
$$\langle \theta_{4.5}, \{Ph_1\}+\eta\epsilon, \nu\rangle = 0 \text{~~with zero indeterminacy in~~}\pi_{58}.$$
Therefore
\begin{equation*}
\begin{split}
\eta \kappa \theta_{4.5} & = \eta \kappa \theta_{4.5} + \rho_{15} \theta_{4.5} \\
& \in \theta_{4.5} \langle \eta\epsilon, \nu, \eta, 2\rangle + \theta_{4.5} \langle \{Ph_1\}, \nu, \eta, 2\rangle \\
& = \theta_{4.5} \langle \{Ph_1\}+\eta\epsilon, \nu, \eta, 2\rangle \\
& \subseteq \langle\langle \theta_{4.5}, \{Ph_1\}+\eta\epsilon, \nu\rangle, \eta, 2\rangle \\
& = \langle 0, \eta, 2\rangle\\
& = 2 \cdot \pi_{60} = \{0, 2\overline{\kappa}^3\}.
\end{split}
\end{equation*}
Note that the following three Toda brackets
$$\langle \eta\epsilon, \nu, \eta, 2\rangle, \ \langle \{Ph_1\}, \nu, \eta, 2\rangle, \ \langle \{Ph_1\}+\eta\epsilon, \nu, \eta, 2\rangle$$
have the same indeterminacy: $2 \cdot \pi_{15} =$ even multiples of $\rho_{15}$, which is annihilated by $\theta_{4.5}$. \\

To prove that $\eta \kappa \theta_{4.5}=0$, we only need to show that
$$\eta \kappa \theta_{4.5} \neq 2\overline{\kappa}^3.$$
Note that $2\overline{\kappa}^3$ is detected by $tmf$, while $\theta_{4.5}$ is chosen not to be detected by $tmf$. Suppose we have the relation
$$\eta \kappa \theta_{4.5} = 2\overline{\kappa}^3.$$
Then mapping this relation into tmf gives us $2\overline{\kappa}^3=0$, which contradicts the fact that $2\overline{\kappa}^3$ is detected in $\pi_\ast(tmf)$. Therefore, we must have that
$$\eta \kappa \theta_{4.5} = 0.$$
\end{proof}

Now we present the proofs of Lemmas 12.5, 12.6, 12.7, 12.9, and a few other lemmas that will be needed for the proofs.

\begin{lem}
In the Adams $E_2$ page, we have a Massey product
$$h_1 x' = \langle h_0^2g_2, h_0, Ph_1\rangle.$$
\end{lem}

\begin{proof}
In Proposition 4.19 of \cite{Tan1}, Tangora showed that we have a May $d_6$ differential
$$d_6(Y) = h_0^3g_2.$$
Here we follow Isaksen's notation \cite{Isa} for names of the elements in the May spectral sequence. Then combining with the fact that $h_1 x' = Y Ph_1$ in the May $E_6$ page, this lemma follows from May's convergence theorem \cite{May}.
\end{proof}

\begin{lem}
We have the relation
$$\{Ph_1\} \cdot \{h_5d_0\} =0 \text{~~in~~} \pi_{54}.$$
\end{lem}

\begin{proof}
First note that the Toda bracket
$$\langle 2, \theta_4, \kappa\rangle \text{~~is detected by~~} h_5d_0.$$
This follows from the Adams $d_2$ differential $d_2(h_5) = h_0h_4^2$ and Moss's theorem. Note that to apply the Moss's theorem here, we need to use the fact that $\theta_4 \kappa =0$, which is obtained by filtration reasons.\\

We compute the product $\{Ph_1\} \langle 2, \theta_4, \kappa\rangle$ next.

\begin{equation*}
\begin{split}
\{Ph_1\} \langle 2, \theta_4, \kappa\rangle & = \langle \{Ph_1\}, 2, \theta_4\rangle \kappa \\
& \subseteq \langle \kappa \{Ph_1\}, 2, \theta_4\rangle \\
& = \langle \eta^3 \overline{\kappa}, 2, \theta_4\rangle \\
& \supseteq \eta^2 \overline{\kappa} \langle \eta , 2, \theta_4\rangle \\
& = \eta^3 \langle 2, \theta_4, \overline{\kappa}\rangle \subseteq \eta^3 \pi_{51} =0.
\end{split}
\end{equation*}

In other words, both $\{Ph_1\} \langle 2, \theta_4, \kappa\rangle$ and $0$ are contained in the same Toda bracket
$$\langle \eta^3 \overline{\kappa}, 2, \theta_4\rangle.$$
Therefore, their difference must be contained in the indeterminacy of this Toda bracket, which is
$$\eta^3 \overline{\kappa} \cdot \pi_{31} + \pi_{24} \cdot \theta_4.$$
It is clear that $\eta^3 \overline{\kappa} \cdot \pi_{31} \subseteq \eta^3 \pi_{51} =0$. Recall that
$$\pi_{24} \cong \mathbb{Z}/2\oplus \mathbb{Z}/2 \text{~~and is generated by~~} \eta\sigma\eta_4 \text{~~and~~} \eta\rho_{23}\text{~~in the~~} Im J.$$
Multiplying by $\theta_4$, both products are zero. This is due to the fact that $\eta\eta_4\theta_4=0$ (See Lemma 4.1 in \cite{BJM}) and filtration reasons. Therefore, we have achieved that
$$\{Ph_1\} \langle 2, \theta_4, \kappa\rangle=0.$$

Then, from the fact that $2\{Ph_1\}=0$ and filtration reasons, the product of $\{Ph_1\}$ and all elements in the $E_\infty$ page of higher filtration than $h_5d_0$ are zero. Therefore, combining with the fact that the Toda bracket
$$\langle 2, \theta_4, \kappa\rangle \text{~~is detected by~~} h_5d_0,$$
we have the homotopy relation that
$$\{Ph_1\} \cdot \{h_5d_0\} =0 \text{~~in~~} \pi_{54}.$$
\end{proof}

\begin{lem}
The permanent cycle $h_1 x'$ in the 54-stem detects the homotopy class $\theta_{4.5} \{Ph_1\}$.
\end{lem}

\begin{proof}
By Lemma 12.2 and Moss's theorem, we have that
$$h_1 x' \text{~~detects an element in the Toda bracket~~} \langle \sigma^2\theta_4, 2, \{Ph_1\}\rangle.$$
Recall that Barratt, Mahowald and Tangora \cite{BMT} showed that
$$h_0^2g_2\text{~~ detects~~} \sigma^2\theta_4.$$
We have the relation that
$$\theta_4 \langle \sigma^2, 2, \{Ph_1\}\rangle \subseteq \langle \sigma^2\theta_4, 2, \{Ph_1\}\rangle.$$
Since also
$$\theta_4 \langle \sigma^2, 2, \{Ph_1\}\rangle \subseteq \theta_4 \cdot \pi_{24} =0,$$
which we showed in the proof of Lemma 12.3, we have that
$$0 \in \langle \sigma^2\theta_4, 2, \{Ph_1\}\rangle.$$
Note that one can also show directly that $\langle \sigma^2, 2, \{Ph_1\}\rangle=0$.\\

Recall that Isaksen \cite{Isa} showed that $h_1 x'$ is a surviving permanent cycle, and it detects both $\nu^3\theta_{4.5}$ and equally $\eta \epsilon \theta_{4.5}$. Therefore, $h_1 x'$ must detect a nontrivial homotopy class in the indeterminacy of the Toda bracket
$$\langle \sigma^2\theta_4, 2, \{Ph_1\}\rangle.$$
The indeterminacy of this Toda bracket is
$$\sigma^2\theta_4 \cdot \pi_{10} + \pi_{45} \cdot \{Ph_1\}.$$
First note that
$$\pi_{10} \cong \mathbb{Z}/2 \text{~~and is generated by~~} \eta\{Ph_1\}.$$
Since $\eta \sigma^2 =0$, we must have that
$$\sigma^2\theta_4 \cdot \pi_{10} =0.$$
Next note that $2\{Ph_1\}=0$, and the generators of $\pi_{45}$ can be chosen to be the following
$$\theta_{4.5} \in \{h_4^3\}, \ \eta \{g_2\}, \ \{h_5d_0\}, \ \{w\}.$$
We have that
$$\{w\}\cdot \{Ph_1\}=0 \text{~~for filtration reasons}.$$
We also have that
\begin{equation*}
\begin{split}
\{Ph_1\}\cdot \eta \{g_2\} & \subseteq \langle \eta, 2, 8\sigma\rangle \eta \{g_2\} \\
& = \eta \langle 2, 8\sigma, \{g_2\} \rangle \eta \\
& = \eta^2 \langle 2, 8\sigma, \{g_2\} \rangle \\
& \subseteq \eta^2 \pi_{52} = 0.
\end{split}
\end{equation*}
Note here we use the fact that $8\sigma\{g_2\}=0$. Then combining with Lemma 12.3 that
$$\{Ph_1\} \cdot \{h_5d_0\} =0,$$
the only possibility is that
$$h_1 x'\text{~~detects the homotopy class~~} \theta_{4.5} \{Ph_1\}.$$
\end{proof}

\begin{lem}
In $\pi_{54}$, we have a relation $\theta_{4.5} (\eta \epsilon + \{Ph_1\} ) = 0$.
\end{lem}

\begin{proof}
The element $d_0g^2$ is the only element in the 54-stem of the $E_\infty$ page with higher filtration than $h_1 x'$. It detects the homotopy class $\kappa \overline{\kappa}^2$, which is also detected in the Hurewicz image of tmf. Since $\theta_{4.5}$ is chosen not to be detected in the Hurewicz image of tmf, and $h_1 x'$ detects both $\eta \epsilon \theta_{4.5}$ and $\{Ph_1\} \theta_{4.5}$, we must have a relation
$$\theta_{4.5} (\eta \epsilon + \{Ph_1\} ) = 0.$$
\end{proof}

\begin{lem}
We have a strictly defined 4-fold Toda bracket
$$\rho_{15} \in \langle \{Ph_1\}, \nu, \eta, 2\rangle \text{~~in~~} \pi_{15},$$
with indeterminacy $2\pi_{15}$ given by even multiples of $\rho_{15}$, where $\rho_{15}$ is a generator of the $Im J$ in $\pi_{15}$.
\end{lem}

\begin{proof}
We first check that this 4-fold Toda bracket is strictly defined. It is clear that
$$\langle \nu, \eta, 2\rangle \subseteq \pi_5 =0.$$
In the Adams $E_2$ page, we have that
$$\langle Ph_1, h_2, h_1\rangle = Ph_2^2 = h_0^2 d_0.$$
The element $h_0^2d_0$ is killed by the Adams $d_3$ differential
$$d_3(h_0^2h_4) = h_0^2d_0.$$
Therefore,
$$0 \in \langle \{Ph_1\}, \nu, \eta\rangle.$$
It is straightforward to check the indeterminacy of this 3-fold Toda bracket is zero. Therefore, this 4-fold Toda bracket is strictly defined.\\

We next check the indeterminacy of this 4-fold Toda bracket. The indeterminacy is contained in the union of the following
$$\langle \{Ph_1\}, \nu, \pi_2\rangle, \langle \{Ph_1\}, \pi_5, 2\rangle, \langle \pi_{13}, \eta, 2\rangle.$$
Note that $\pi_5=0, \pi_{12}=0, \pi_{13}=0$, $\pi_2$ is generated by $\eta^2$ and $\pi_6$ is generated by $\nu^2$. We have
$$\langle \{Ph_1\}, \nu, \eta^2\rangle \supseteq \langle \{Ph_1\}, \nu, \eta\rangle \eta =0.$$
$$\{Ph_1\} \cdot \nu^2 \in \nu \cdot \pi_{12} =0.$$
Therefore, the indeterminacy is $2\pi_{15}$.\\

Now we multiply this 4-fold Toda bracket by $\eta^2$:
$$\langle \{Ph_1\}, \nu, \eta, 2\rangle \eta^2 = \{Ph_1\} \langle  \nu, \eta, 2, \eta^2\rangle = \{Ph_1\} \epsilon.$$
The 4-fold Toda bracket $\epsilon = \langle  \nu, \eta, 2, \eta^2\rangle$ is strictly defined with zero indeterminacy. The homotopy class $\{Ph_1\} \epsilon$ is detected by the surviving cycle $Ph_1c_0$. We have a nontrivial extension:
$$\eta^2 \rho_{15} \in \{Ph_1c_0\}.$$
Therefore, we must have that the 4-fold Toda bracket
$$\langle \{Ph_1\}, \nu, \eta, 2\rangle \text{~~contains~~} \rho_{15} \text{~~or~~} \rho_{15}+\eta\kappa.$$

To eliminate the second possibility, we multiply this 4-fold Toda bracket by $\overline{\kappa}$. Note that
$$\overline{\kappa} \{Ph_1\} \subseteq \pi_{29}=0,$$
$$\langle \overline{\kappa}, \{Ph_1\}, \nu \rangle =0 \text{~~with indeterminacy~~}\{0, \nu\theta_4\}\text{~~in~~}\pi_{33}.$$
In fact, in the Adams $E_2$ page, we have the Massey product
$$\langle g, Ph_1, h_2\rangle = 0 \text{~~in Adams filtration 9}.$$
The homotopy classes that survive in $\pi_{33}$ with filtration higher than 9 are detected by the $K(1)$-local sphere. Since the class $\overline{\kappa}$ maps trivially to the $K(1)$-local sphere, we must have that
$$\langle \overline{\kappa}, \{Ph_1\}, \nu \rangle \text{~~contains 0}.$$
Then it is straightforward to check the indeterminacy is
$$\overline{\kappa}\cdot \pi_{13} + \pi_{30}\cdot\nu = \{0, \nu\theta_4\}.$$
Now we have that
\begin{equation*}
\begin{split}
\overline{\kappa} \langle \{Ph_1\}, \nu, \eta, 2\rangle & \subseteq \langle \langle \overline{\kappa}, \{Ph_1\}, \nu\rangle, \eta, 2\rangle \\
& = \langle \{0, \nu\theta_4\}, \eta, 2\rangle \\
& = \text{~~the union of~~} \langle 0, \eta, 2\rangle \text{~~and~~} \langle \nu\theta_4,\eta, 2\rangle \\
& = 2 \cdot\pi_{35}.
\end{split}
\end{equation*}
Note that $2\cdot \pi_{35}$ is detected in the $K(1)$-local sphere. Since the class $\overline{\kappa}$ maps trivially to the $K(1)$-local sphere, we have that
$$\overline{\kappa} \langle \{Ph_1\}, \nu, \eta, 2\rangle =0.$$
On the other hand, it is clear that
$$\eta \kappa \overline{\kappa} \neq 0 \text{~~and is detected by~~}h_1d_0g,$$
and that
$$\rho_{15} \overline{\kappa} \in \langle 8, 2\sigma, \sigma\rangle \overline{\kappa} = 8 \langle 2\sigma, \sigma, \overline{\kappa} \rangle \subseteq 8 \pi_{35}=0.$$
Here by Moss's theorem, the relation
$$\rho_{15} \in \langle 8, 2\sigma, \sigma\rangle$$
follows from the Adams differential $d_2(h_4) = h_0h_3^2$ and the Massey product in the $E_3$ page
$$\langle h_0^3, h_0h_3, h_3\rangle = h_0^3h_4 \text{~~with zero indeterminacy}.$$

Therefore, the 4-fold Toda bracket
$$\langle \{Ph_1\}, \nu, \eta, 2\rangle \text{~~contains~~} \rho_{15}.$$
\end{proof}

\begin{lem}
We have the relation $\rho_{15}\theta_{4.5}=0$ in $\pi_{60}$.
\end{lem}

\begin{proof}
We first claim that
$$\rho_{15} \theta_4 = 8 \theta_{4.5}.$$
In fact, they are both detected by the surviving cycle $h_0^2h_5d_0$ (See Tangora \cite{Tan2}). However, there is one more element $w$ in higher filtration in the $E_\infty$ page, so the two classes might differ by that. Since
$$\eta^2 \theta_4 =0, \text{~~and~~} \eta^2 \{w\} \neq 0,$$
their difference is not $\{w\}$, and hence must be zero. Note that one can also show this by mapping the relation into tmf.\\

Then we have that
\begin{equation*}
\begin{split}
\rho_{15} \theta_{4.5} & \subseteq \langle 8, 2\sigma, \sigma\rangle \theta_{4.5} \\
& \subseteq \langle 8 \theta_{4.5}, 2\sigma, \sigma\rangle \\
& = \langle \rho_{15} \theta_4, 2\sigma, \sigma\rangle \\
& = 0 \text{~~with zero indeterminacy}.
\end{split}
\end{equation*}
The last equation is proved by the second author as Lemma 2.4 in \cite{Xu}.
\end{proof}

\begin{lem}
We have a Toda bracket in $\pi_{20}$:
$$\langle \{Ph_1\}+\eta\epsilon, \nu, \sigma\rangle=0\text{~~with zero indeterminacy}.$$
\end{lem}

\begin{proof}
We consider the two brackets $\langle \{Ph_1\}, \nu, \sigma\rangle$ and $\langle \eta\epsilon, \nu, \sigma\rangle$ one by one.\\

For the first bracket, in the Adams $E_2$ page we have the Massey product
$$\langle Ph_1, h_2, h_3\rangle =0$$
with zero indeterminacy in Adams filtration 6. Since there is no surviving class in Adams filtration 7 or higher, it contains zero. For filtration reasons and the fact that $\pi_{13} = 0$, the indeterminacy of the first bracket is
$$\{Ph_1\} \cdot \pi_{11} + \pi_{13} \cdot \sigma =0.$$
Therefore,
$$\langle \{Ph_1\}, \nu, \sigma\rangle = 0\text{~~with zero indeterminacy}.$$

For the second bracket, we have that
$$\langle \eta\epsilon, \nu, \sigma\rangle \supseteq \epsilon \langle \eta, \nu, \sigma\rangle \subseteq \epsilon \cdot \pi_{12}=0.$$
Therefore, it contains 0. Again, by filtration reasons and the fact that $\pi_{13} = 0$, the indeterminacy of the second bracket is
$$\eta\epsilon \cdot \pi_{11} + \pi_{13} \cdot \sigma =0.$$
Therefore,
$$\langle \eta\epsilon, \nu, \sigma\rangle =0\text{~~with zero indeterminacy}.$$

Summing these two relations, we have that
$$\langle \{Ph_1\}+\eta\epsilon, \nu, \sigma\rangle =0\text{~~with zero indeterminacy}.$$
\end{proof}

\begin{lem}
We have a Toda bracket in $\pi_{58}$:
$$\langle \theta_{4.5}, \{Ph_1\}+\eta\epsilon, \nu\rangle=0\text{~~with zero indeterminacy}.$$
\end{lem}

\begin{proof}
First, by Lemma 12.5, we have the relation
$$\theta_{4.5} \cdot(\{Ph_1\}+\eta\epsilon)=0.$$
Therefore, this Toda bracket is defined.\\

Recall that
$$\text{the cokernel of J in~~} \pi_{58} \text{~~is~~}\mathbb{Z}/2, \text{~~and generated by~~} \{h_1Q_2\}.$$
The indeterminacy equals
$$\theta_{4.5} \cdot \pi_{13} + \pi_{55} \cdot \nu =0.$$
The relation $\pi_{55} \cdot \nu =0$ follows from filtration reasons. As a side remark, one can actually prove that
$$\{h_1Q_2\}\text{~~is indecomposable}.$$
This can be shown by the Adams-Novikov filtration of this element. See Isaksen \cite{Isa} for details.\\

In \cite{Isa}, Isaksen showed that the permanent cycle $h_1h_3Q_2$ cannot be killed by $r_1$. The only other candidate to kill $h_1h_3Q_2$ is $h_1^3h_6$, which is obviously a permanent cycle: it detects $\eta^2\eta_6$. Therefore,
$$h_1h_3Q_2\text{~~is a surviving cycle, and detects~~}\sigma\{h_1Q_2\}.$$
By Lemma 12.8, we have that
$$\langle \theta_{4.5}, \{Ph_1\}+\eta\epsilon, \nu\rangle \sigma =  \theta_{4.5} \langle\{Ph_1\}+\eta\epsilon, \nu, \sigma\rangle =0.$$
Therefore,
$$\langle \theta_{4.5}, \{Ph_1\}+\eta\epsilon, \nu\rangle \text{~~does not contain~~}\{h_1Q_2\},$$
and hence is 0 with zero indeterminacy.
\end{proof}

\section{Appendix I}

The theory of cell diagrams is very helpful when thinking of finite CW spectra. We use them as illustration purpose in Section 5. In this section, we recall the definition of cell diagrams from \cite{BJM}. We also include several examples.

\begin{defn}
Let $Z$ be a finite CW spectrum. Then a cell diagram for $Z$ consists of nodes and edges. The nodes are in 1-1 correspondence with a chosen basis of the mod 2 homology of $Z$, and may be labeled with symbols to indicate the dimension. When two nodes are joined by an edge, then it is possible to form an $H\mathbb{F}_2$-subquotient
$$Z'/Z'' = S^n \smile_f e^m,$$
\begin{displaymath}
    \xymatrix{
 *+[o][F-]{m} \ar@{-}[d]^{f}  \\
 *+[o][F-]{n}  }
\end{displaymath}
which is the cofiber of $f$ with certain suspension. Here $f$, the attaching map, is an element in the stable homotopy groups of spheres. For simplicity, we do not draw an edge if the corresponding $f$ is null.

Suppose we have two nodes labeled $n$ and $m$ with $n<m$, and there is no edge joining them. Then there are two possibilities.

The first one is that there is an integer $k$, and a sequence of nodes labeled $n_i, 0\leq i \leq k$, with $n=n_0<n_1<\cdots<n_k=m$, and edges joining the nodes $n_i$ to the nodes $n_{i+1}$. In this case we do not assert that there is an $H\mathbb{F}_2$-subquotient of the form above; this does not imply that there is no such $H\mathbb{F}_2$-subquotient.

The second one is that there is no such sequence as in the first case. In this case, there exists an $H\mathbb{F}_2$-subquotient which a wedge of spheres $S^n\vee S^m$.
\end{defn}

\begin{rem}
In \cite{BJM}'s original definition, they use subquotients instead of $H\mathbb{F}_2$-subquotients.
\end{rem}

\begin{exmp}
Let $f$ be the composite of the following two maps:

\begin{displaymath}
    \xymatrix{
 S^2 \ar[r]^{\eta^2} & S^0 \ar[r]^i & C\eta,
 }
\end{displaymath}
where the second map $i$ is the inclusion of the bottom cell. Consider the cofiber of $f$: $Cf$, which is a 3 cell complex
with the following cell diagram:

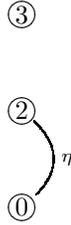
\begin{displaymath}
    \xymatrix{
    *+[o][F-]{3} \\
    *+[o][F-]{2} \ar@{-}@/^1pc/[d]^{\eta} \\
    *+[o][F-]{0} }
\end{displaymath}
It is clear that the top cell of $Cf$ splits off, since $\eta^2$ can be divided by $\eta$. So we do not have to draw any attaching map from the cell in dimension 3 to the one in dimension 0. Note that the cofiber of $\eta^2$ is in fact an $H\mathbb{F}_2$-subcomplex of $Cf$. One could think this as the indeterminacy of cell diagrams associated to a given CW spectrum.
\end{exmp}

\begin{exmp}
Let $X_1=P_1^4$. The cell diagram of $X_1$ is the following:
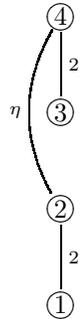
\begin{displaymath}
    \xymatrix{
 *+[o][F-]{4} \ar@{-}[d]^{2} \ar@{-}@/_1pc/[dd]_{\eta} \\
    *+[o][F-]{3} \\
    *+[o][F-]{2} \ar@{-}[d]^{2}\\
    *+[o][F-]{1} }
\end{displaymath}

As a comparison, let $X_2= C 2\wedge C \eta$, where $C 2$ and $C \eta$ are the cofibers of $2$ and $\eta$. Then the cell diagram of $X_2$ is the following:
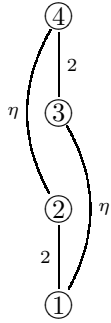
\begin{displaymath}
    \xymatrix{
 *+[o][F-]{4} \ar@{-}[d]^{2} \ar@{-}@/_1pc/[dd]_{\eta} \\
    *+[o][F-]{3} \ar@{-}@/^1pc/[dd]^{\eta} \\
    *+[o][F-]{2} \ar@{-}[d]_{2}\\
    *+[o][F-]{1} }
\end{displaymath}
\end{exmp}

We give a more interesting example.

\begin{exmp}
Consider the suspension spectrum of $\mathbb{C}P^3$. It consists of three cells: one each in dimensions 2, 4 and 6. It is shown in \cite{Ada2} by Adams that, the secondary cohomology operation $\Psi$, which is associated to the relation
$$Sq^4 Sq^1 + Sq^2 Sq^1 Sq^2 + Sq^1 Sq^4 = 0,$$
is nonzero on this spectrum. In other words, there exists an attaching map between the cells in dimension 2 and 6, which is detected by $h_0h_2$ in the 3-stem of the Adams $E_\infty$ page. Note that $h_0h_2$ detects two homotopy classes: $2\nu, \ 6\nu$. Their difference is $4\nu = \eta^3$, which is divisible by $\eta$. Therefore, we have its cell diagram as the following:

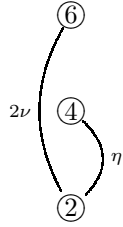
\begin{displaymath}
    \xymatrix{
    *+[o][F-]{6} \ar@{-}@/_1pc/[dd]_{2\nu} \\
    *+[o][F-]{4} \ar@{-}@/^1pc/[d]^{\eta} \\
    *+[o][F-]{2} }
\end{displaymath}
\end{exmp}

\section{Appendix II}

This section is about intuition.

We summarize and explain the major ideas of how we think of the ``road map" of the proof of the differential $d_3(D_3) = B_3$, especially of Step 4. The ``zigzag" part of the explanation is crucial if one wants to generalize this method to other Adams differentials.

We try to prove an Adams $d_3$ differential in $P_1^\infty$:
$$d_3(h_1h_3h_5[22]) = G[6].$$

The element $G$ supports a differential \cite{Isa, Isa2} in the Adams spectral sequence of $S^0$:
$$d_3(G)=Ph_5d_0.$$
From the computation of the transfer map, we have that
$$Ph_5d_0[6]\text{~~ maps to~~} B_{21}$$
It is shown in \cite{Isa} that $d_3(B_3)\neq B_{21}$. Therefore, the only possibility is that
$$G[6]\text{~~ supports a~~} d_2\text{~~ differential in~~} P_1^6.$$
Checking the bidegree gives us the only element there: $h_5i[5]$. This argument can be summarized in the following diagram:

\begin{displaymath}
    \xymatrix{
  Ext(S^6) & Ext(P_1^6) \ar[l] \ar[r] & Ext(P_1^\infty) \ar[r] & Ext(S^0) \\
 Ph_5d_0[6]  & Ph_5d_0[6] \ar@{|->}[l] \ar@{|->}[r] & Ph_5d_0[6] \ar@{|->}[r] & B_{21} \\
 & h_5i[5] & & \\
  G[6] \ar@{-->}[uu]^{d_3} & G[6] \ar@{-->}[u]^{d_2} \ar@{|->}[l] \ar@{|->}[r] & G[6] \ar@{|->}[r] & B_3
    }
\end{displaymath}

\begin{rem}
The above argument implies that in the Adams spectral sequence of $P_1^2$, we have a differential
$$d_2(G[2])= h_5i[1].$$
This differential in the mod 2 Moore spectrum is not obtained by a zigzag.
\end{rem}

The Curtis table shows that
$$h_5i[5]\text{~~ is killed by~~} B_1[14].$$
Note that the element $B_1$ in $Ext(S^0)$ is a surviving cycle.

This zigzag suggests that, if the element $G[6]$ were going to survive in the Adams spectral sequence of $P_1^{23}$, then it would jump the Adams filtration by 1 to the element $B_1[14]$ in the Adams spectral sequence of $P_{14}^{23}$. This is the first half of the intuition of Step 4: we reduce the Adams $d_3$ differential in $P_1^{23}$ to an Adams $d_4$ differential in $P_{14}^{23}$.

The second half of the intuition is related to the source element $h_1h_3h_5[22]$. The Massey product $h_0h_4^3 = \langle h_2, h_1, h_0, h_1h_3h_5 \rangle$ and the nonzero Steenrod operation $Sq^1 Sq^2 Sq^4$ on the 15 dimensional class in $H^\ast(P_{14}^{23})$ suggest that we should have a differential
$$h_1h_3h_5[22] \text{~~kills~~} h_0h_4^3[15]$$
in the Curtis table of $P_1^\infty$. However, the element $h_0h_4^3[15]$ is killed by $h_4^3[16]$ in the Curtis table because $P_{15}^{16}$ is a suspension of the mod 2 Moore spectrum. Therefore, if we remove the 15-cell in $P_{14}^{23}$, we can ``separate" the two elements $h_1h_3h_5[22]$ and $h_4^3[16]$. To do this, we take the cofiber of the inclusion of the 15-cell to get the spectrum $X$, and reduce the Adams $d_4$ differential in $P_{14}^{23}$ to an Adams $d_4$ differential in $X$.

It is therefore clear that the $\eta$-extension from $h_4^3$ to $B_1$ gives us the $d_4$ differential in $X$, since the 16-cell is attached to the 14-cell by $\eta$.

\end{document}